\documentclass[11pt,a4paper,oneside,reqno]{amsart}
\usepackage{amssymb}
\usepackage{mathrsfs}

\usepackage{amsmath}
\usepackage{amsthm}
\usepackage{amssymb}
\usepackage{amsbsy}
\usepackage{amsfonts}
\usepackage{amstext}
\usepackage{epsfig}

\newtheorem{theorem}{Theorem}[section]

\newtheorem{corollary}[theorem]{Corollary}

\newtheorem{lemma}[theorem]{Lemma}
\newtheorem{proposition}[theorem]{Proposition}
\newtheorem{definition}[theorem]{Definition}
\newtheorem{remark}[theorem]{Remark}

\usepackage{graphicx}

\usepackage{caption}
\usepackage{nicematrix}
\usepackage{tikz-cd}
\usepackage{tikz}

\usepackage{colorinfo}

\usetikzlibrary{knots}

\usetikzlibrary{matrix}
\usetikzlibrary{arrows}

\usetikzlibrary{patterns}
\usetikzlibrary{shapes}
\usetikzlibrary{arrows}
\usetikzlibrary{decorations.markings}
\usetikzlibrary{decorations.pathreplacing}

\def\CC{{\mathbb C}}
\def\RR{{\mathbb R}}
\def\ZZ{{\mathbb Z}}

\def\QQ{{\mathbb Q}}
\def\Hess{{\mathrm{Hess}}}
\newcommand{\B}{\mathsf{b}}
\def \BLog{\operatorname{\textbf{Log}}}

\def\Vol{\operatorname{Vol}}

\def \Re{\operatorname{Re}}
\def \im{\operatorname{Im}}

\def \Li{\operatorname{Li_2}}
\def \im{\operatorname{Im}}

\def \Log{\operatorname{Log}}

\usepackage{graphicx} 

\title[On the KLV Partition Function for Cusped 3-Manifolds]{On the Kashaev--Luo--Vartanov Partition Function\\ for Cusped 3-Manifolds: \\
Asymptotics and a squared norm property}
\author{Ka Ho Wong}
\date{}

\address{
	Yale University, New Haven, CT06520, USA}
\email{kaho.wong@yale.edu}

\begin{document}

\maketitle

\begin{abstract}
We study the asymptotic and TQFT-type properties of the Kashaev--Luo--Vartanov (KLV) partition function. We show that the partition function depends on the prescribed angle structure only through its peripheral angular holonomies. For any geometric triangulation of a cusped 3-manifold realizing a possibly incomplete hyperbolic structure, we express the exponential decay rate and the 1-loop term of the partition function in terms of, respectively, the volume and the adjoint twisted Reidemeister torsion of the corresponding incomplete hyperbolic structure. Consequently, these asymptotic formulas hold for any angle structure having the same peripheral angular holonomies as the geometric one.

We further introduce a Reshetikhin--Turaev-type function associated with certain admissible Neumann--Zagier data, extending the Jones function in Teichm\"uller TQFT previously studied by Ben Aribi and the author. We show that, after a combinatorial normalization, for any triangulation, the KLV partition function can be expressed as a weighted integral of the squared norm of the Reshetikhin--Turaev-type function, where the weight is determined by the angular holonomies of peripheral curves. In particular, for FAMED triangulations, the Reshetikhin--Turaev-type function agrees with the Jones function, so that the normalized KLV partition function is obtained by integrating the squared norm of the integrand defining the Teichm\"uller TQFT partition function. This provides a noncompact analogue of the relationship between Turaev--Viro and Reshetikhin--Turaev invariants.
\end{abstract}

\tableofcontents

\section{Introduction}
\subsection{Overview}
In \cite{KLV}, Kashaev, Luo and Vartanov define an absolute convergent integral associated with any ideal triangulation of a cusped 3-manifold equipped with an angle structure. The partition function is known to be invariant under any 3--2 angled Pachner move and any shape gauge transformation that is equivalent to the leading-trailing deformation of the angle structure along any edge of the triangulation \cite[Theorem 1]{KLV}. Despite the fact that the full invariance of the partition function is still unknown, the partition function defined by Kashaev--Luo--Vartanov in \cite{KLV} and the one defined by Andersen-Kashaev in \cite{AK} are believed to be parts of mathematical models of TQFTs obtained by quantizing the Teichm\"uller space. Similar to the norm-square relation between the Turaev-Viro and the Reshetikhin--Turaev theories, the KLV partition function is expected to be the Turaev-Viro type counterpart of the partition function in Teichm\"uller TQFT in a certain sense \cite[Conjecture 1]{KLV}. 
Especially, for any cusped 3-manifold, inspired by the fact that the Truaev-Viro invariants equals the sum of squared norm of the relative Reshetikhin--Turaev invariants (see \cite[Theorem 1.1]{DKY} and \cite[Proposition 1.5]{WY}), a certain Reshetikhin--Turaev-type function should exist in such a way that the Kashaev--Luo--Vartanov partition function equals the integration of the squared norm of such a function, possibly with some additional weight factor determined by the prescribed angle structure.

Another important aspect of TQFT is the semiclassical limits of  partition functions associated with the underlying hyperbolic 3-manifolds. A famous example is the Kashaev-Murakami-Murakami volume conjecture \cite{Ka, MM},  which suggests that for any hyperbolic knot in the three sphere, the sequence of $N$-th colored Jones polynomial of the knot evaluated at the first primitive $N$-th root of unity grows exponentially with growth rate equals the hyperbolic volume of the knot complement. For Teichm\"uller TQFT, in \cite{AK}, Andersen and Kashaev conjecture that for any hyperbolic knot in the three sphere, there exists a Jones function in such a way that the partition function can be written as an integral of the Jones function weighted by a simple factor depending on the angular holonomy of a peripheral curve. Moreover, in the semiclassical limit, the Jones function is expected to decay exponentially with decay rate equals the hyperbolic volume of the knot complement. Due to the conjectural relationship between the Kashaev--Luo--Vartanov partition function and the Andersen-Kashaev partition function, it is expected that the Kashaev--Luo--Vartanov partition function decay exponentially with decay rate equals twice the hyperbolic volume of the underlying 3-manifolds. Moreover, the 1-loop term in the asymptotic expansion should be related to the adjoint twisted Reidemeister torsion.

In this paper, given any ideal triangulation of a cusped 3-manifold, associated with the Neumann--Zagier data satisfying some mild conditions, we construct a Reshetikhin--Turaev type function satisfying the aforementioned TQFT-type squared norm property. The construction of the function recovers the Jones function in the earlier work of Ben Aribi and the author in their study of Andersen-Kashaev volume conjecture for FAMED triangulations, which are triangulations satisfying certain combinatorial properties \cite{BAW}. In particular, for FAMED triangulations, our formula in Theorem \ref{thm1} below gives the conjectural doubling relation between the partition functions defined by Kashaev--Luo--Vartanov and the one defined by Andersen--Kashaev. Furthermore, using the formula, for geometric triangulations, we establish an asymptotic expansion formula for the KLV partition function and show that the exponential term and the 1-loop term capture the hyperbolic volume and the adjoint twisted Reidemeister torsion respectively.

\subsection{Main results} 
Let $M$ be a cusped 3-manifold with $k$ toroidal boundary $\mathbb{T}_1,\dots,\mathbb{T}_k$ for some $k\geq 1$. Let $X=\{T_1,\dots, T_N\}$ be an ideal triangulation of $M$ and let $\{e_1,\dots, e_N\}$ be the set of edges. On each boundary torus $\mathbb{T}_j$, we pick a non-trivial simple closed curve $l_j$ and a representative in its homotopy class. Let $\boldsymbol{l}=(l_1,\dots,l_k)$. To study the geometry of the manifold, for each tetrahedron $T_j$, we assign the shape parameters $z_j, z_j', z_j'' \in \CC\setminus\{0,1\}$ in a counterclockwise order around each vertex so that the opposite edges share the same shape parameter as shown in Figure \ref{quad}. Each of such assignment of shape parameters is called a quad type, and two different quad types are related by suitably permuting $(z_j,z_j',z_j'')$ for $j=1,\dots, N$ in a cyclic order,. 
We write $\BLog \mathbf{z} = (\Log z_1,\dots, \Log z_N)^{\!\top}, \BLog \mathbf{z'} = (\Log z_1',\dots, \Log z_N')^{\!\top}$ and $ \BLog \mathbf{z''} = (\Log z_1'',\dots, \Log z_N'')^{\!\top}$.
\begin{figure}[h]
    \centering
    \includegraphics[width=\textwidth]{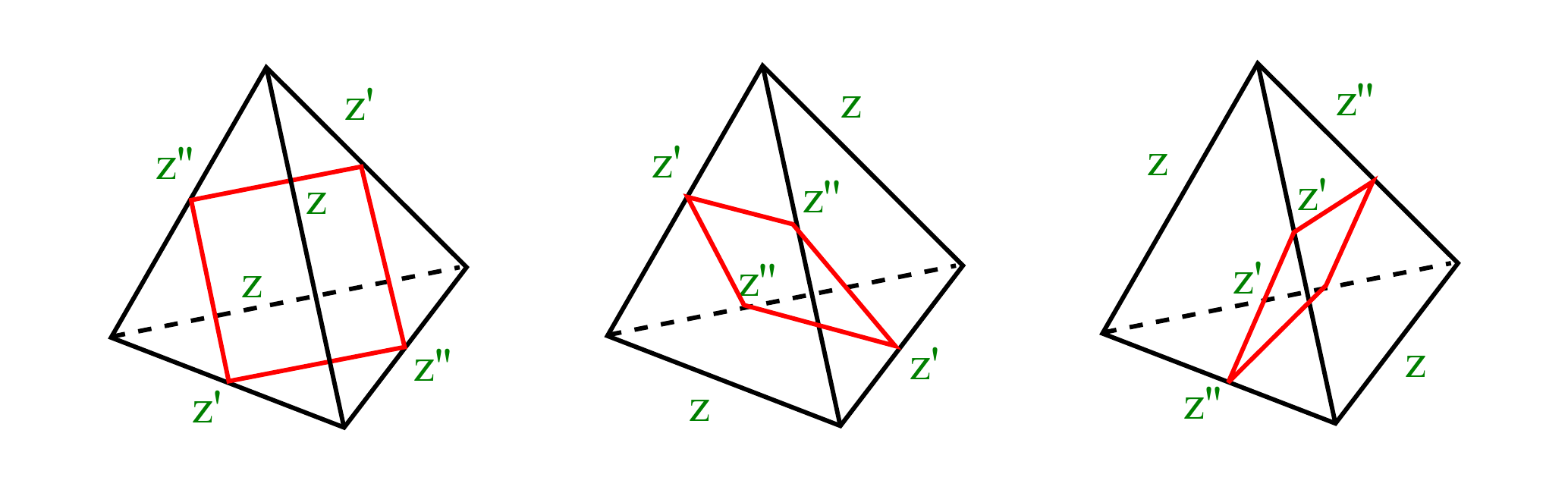}
    \caption{Three possible choices of quad types.}
    \label{quad}
\end{figure}
\begin{figure}[h]
    \centering
    \includegraphics[width=\textwidth]{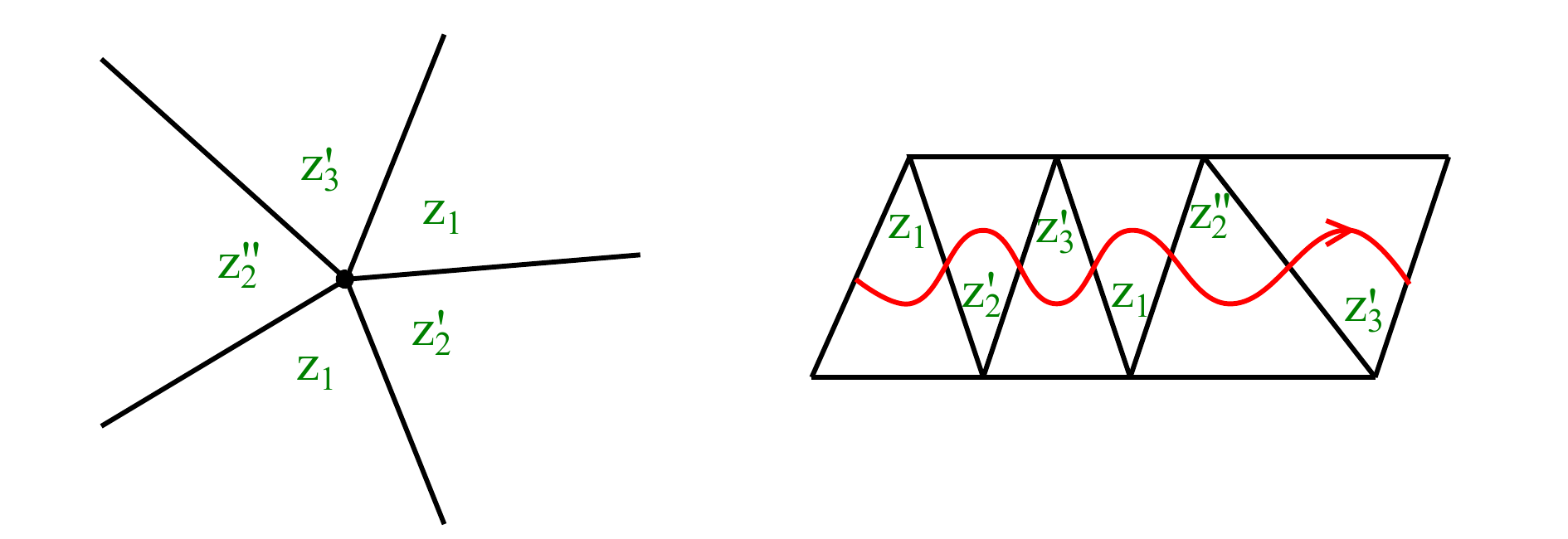}
    \caption{In the left figure, the black dot at the center represents an edge $e_1$ of the triangulation. In this case, we have $(\mathbf{G_e})_{11}=2, (\mathbf{G_e'})_{12}=1,(\mathbf{G_e''})_{12}=1, (\mathbf{G_e'})_{13}=1$, and $(\mathbf{G_e})_{1j}=(\mathbf{G_e'})_{1j}=(\mathbf{G_e''})_{1j}=0$ otherwise. In the right figure, the red curve represents a simple closed curve $l_1$ on a fundamental domain of boundary torus $\mathbb{T}_1$. In this case, we have $(\mathbf{G}_{\boldsymbol{l}})_{11}=1-1=0,(\mathbf{G'_{\boldsymbol{l}}})_{12}=-1, (\mathbf{G''_{\boldsymbol{l}}})_{12}=1, (\mathbf{G'_{\boldsymbol{l}}})_{13}=1-1=0$, and $(\mathbf{G}_{\boldsymbol{l}})_{1j}=(\mathbf{G'_{\boldsymbol{l}}})_{1j}=(\mathbf{G''_{\boldsymbol{l}}})_{1j}=0$ otherwise.}
    \label{edgehol}
\end{figure}
For any $ (\xi_1,\dots, \xi_k)\in \CC^k$, we consider the system of edge equations
\begin{align}\label{edgeintro}
    \mathbf{G_e} \BLog \mathbf{z} + \mathbf{G_e'} \BLog \mathbf{z'} + \mathbf{G_e''} \BLog \mathbf{z''} = 
\begin{pmatrix}
    2\pi i , \dots , 2\pi i
    \end{pmatrix}^{\!\top}
\end{align}
together with the holonomy equations
\begin{align}\label{holintro}
    \mathbf{G}_{\boldsymbol{l}} \BLog \mathbf{z} + \mathbf{G}_{\boldsymbol{l}}'  \BLog \mathbf{z'} + \mathbf{G}_{\boldsymbol{l}} '' \BLog \mathbf{z''} = 
\begin{pmatrix}
    \xi_1 , \dots , \xi_k
    \end{pmatrix}^{\!\top},
\end{align}
where $\mathbf{G_e},\mathbf{G_e'}$ and $\mathbf{G_e''}$ are $N$ by $N$  matrices and $\mathbf{G}_{\boldsymbol{l}},\mathbf{G}_{\boldsymbol{l}}'$ and $\mathbf{G}_{\boldsymbol{l}}''$ are $k$ by $N$ matrices defined as follows (see Figure \ref{edgehol} for a concrete example). The $(i,j)$-th entry of $\mathbf{G_e},\mathbf{G_e'}$ and $\mathbf{G_e''}$ respectively record the number of edges of the tetrahedron $T_j$ labeled with shape parameters $z_j, z_j'$ and $z_j''$ that are glued to the edge $e_i$ of the triangulation. The $(i,j)$-th entry of $\mathbf{G}_{\boldsymbol{l}},\mathbf{G}_{\boldsymbol{l}}'$ and $\mathbf{G}_{\boldsymbol{l}}''$ record the number of shape parameters $z_j$ on the left the peripheral curves $l_i$ minus the number of shape parameters $z_j$ on the right of $l_i$.

From \cite{NZ}, it is known that there exists a choice of $N-k$ independent edges, in the sense that if the $N-k$ edge equations are satisfied, the remaining $k$ edge equations will also be satisfied. Up to renumbering, we assume that $\{e_1,\dots,e_{N-k}\}$ is a set of independent edges, and combine Equations \eqref{edgeintro} and \eqref{holintro} into a system of $N$ equations
\begin{align}\label{NZintro}
    \mathbf{G} \BLog \mathbf{z} + \mathbf{G'} \BLog \mathbf{z'} + \mathbf{G''} \BLog \mathbf{z''} = 
\begin{pmatrix}
    2\pi i, \dots, 2\pi i,  \xi_1, \dots , \xi_k
    \end{pmatrix}^{\!\top},
\end{align}
where the first $N-k$ rows of $\mathbf{G}, \mathbf{G'}$ and $\mathbf{G''}$ are those of $\mathbf{G_e}, \mathbf{G_e'}$ and $\mathbf{G_e''}$, and the last $k$ rows of $\mathbf{G}, \mathbf{G'}$ and $\mathbf{G''}$ are the rows of $\mathbf{G}_{\boldsymbol{l}},\mathbf{G}_{\boldsymbol{l}}'$ and $\mathbf{G}_{\boldsymbol{l}}''$. 
We let $\mathbf{A}=\mathbf{G}-\mathbf{G}'$ and $\mathbf{B} = \mathbf{G''} - \mathbf{G'}$. From \cite[Lemma A.3]{DG}, it is known that there always exists a quad type such that $\mathbf{B}$ is invertible.

\begin{theorem}\label{thm1}
Let $M$ be a cusped 3-manifold with $k$ toroidal boundary. Let $X=\{T_1,\dots,T_N\}$ be an ideal triangulation of $M$ that supports angle structures. Let $\boldsymbol{l}=(l_1,\dots,l_k)$ be a system of peripheral curves as above.
Associated with any choice of quad type $\square$ with an invertible $\mathbf{B}$, there exists a Reshetikhin--Turaev type function $\mathrm{RT}_\hbar^{(X,\boldsymbol{l},\square)}( \boldsymbol{x})$, which is independent of the choice of the angle structure, the set of independent edges and the representatives in the homotopy classes of $l_1,\dots,l_k,$ such that the following properties hold.
\begin{enumerate}
    \item For any angle structure $\alpha$ of $X$, up to a  normalization that depends only on the combinatorics of $X$ but not on $\hbar$ and $\alpha$, the KLV partition function $ W_\B(X,\alpha)$ satisfies
    \begin{align*}
       &\ W_\B(X,\alpha) 
    = \int_{x_1\in \RR + i \kappa_{1}(\alpha)}
    \dots
    \int_{x_k\in \RR + i \kappa_{k}(\alpha)}\left|\mathrm{RT}_\hbar^{(X,\boldsymbol{l},\square)}( \boldsymbol{x}) e^{\frac{1}{4\pi\hbar} \sum_{j=1}^k x_j\lambda_{j}(\alpha)} \right|^2   dx_k\dots dx_1,
    \end{align*} 
    where for $j=1,\dots, k$, $\lambda_{j}(\alpha)$ and $\kappa_{j}(\alpha)$ are, respectively, the angular holonomy of $l_j$ and a certain linear combination of angular holonomies of peripheral curves. In particular, the KLV partition function depends on the prescribed angle structure only through its peripheral angular holonomies.
    \item Suppose $X$ is a FAMED triangulation of hyperbolic knot complement in the three sphere with the quad type induced by the ordering and with $l$ being the homological longitude. Then the function $\mathrm{RT}_\hbar^{(X,\boldsymbol{l},\square)}( \boldsymbol{x})$ equals  the Jones function defined in \cite{BAW}. In this case, for the Andersen-Kashaev partition function $\mathcal{Z}_\hbar(X,\alpha)$, we have
    $$
    \mathcal{Z}_\hbar(X,\alpha)
    = \int_{x\in \RR + i \kappa(\alpha)}\mathrm{RT}_\hbar^{(X,\boldsymbol{l},\square)}( \boldsymbol{x}) e^{\frac{1}{4\pi\hbar}  x\lambda(\alpha)}    dx_1.
    $$
\end{enumerate}
\end{theorem}    

\begin{remark}
    Inspired from Theorem \ref{thm1}(1)\&(2), under the same setting and assumptions, one can define an analogue of  Andersen-Kashaev partition function using the formula
    $$
    \mathcal{Z}^{(\boldsymbol{l},\square)}_\hbar(X,\alpha)
    = 
    \int_{x_1\in \RR + i \kappa_{1}(\alpha)}
    \dots
    \int_{x_k\in \RR + i \kappa_{k}(\alpha)} \mathrm{RT}_\hbar^{(X,\boldsymbol{l},\square)}( \boldsymbol{x}) e^{\frac{1}{4\pi\hbar} \sum_{j=1}^k x_j\lambda_{j}(\alpha)}    dx_k\dots dx_1.
    $$
    One can check that such an integral converges absolutely and recover the original Andersen-Kashaev partition function for FAMED triangulations introduced in \cite{BAW}. The arguments for establishing the asymptotics result for $\mathcal{Z}_\hbar(X,\alpha)$  in \cite{BAW}, which relates the exponential decay term and the 1-loop term with the hyperbolic cone structure of the manifold specified by the prescribed angle structure, can extend  to $\mathcal{Z}^{(\boldsymbol{l},\square)}_\hbar(X,\alpha)$. Similarly, the arguments for establishing the asymptotics result for the Jones function in \cite{BAW} can extend to $\mathrm{RT}_\hbar^{(X,\boldsymbol{l},\square)}(\boldsymbol{x})$. We will investigate these properties in future publications.
\end{remark}

\begin{remark}
    The parameter $\boldsymbol{x}$ of $\mathrm{RT}_\hbar^{(X,\boldsymbol{l},\square)}( \boldsymbol{x})$ can be interpreted as the logarithmic holonomies of a certain system of peripheral curves.  See Lemma \ref{xitox} for details.
\end{remark}

Our next result provides the asymptotics of the KLV partition function for geometric triangulations and the associated angle structures. 
In contrast to the result in \cite{BAW} on the Andersen-Kashaev volume conjecture for FAMED geometric triangulations, our main result only requires geometricity and does not require any extra combinatorial conditions on the triangulations. The longstanding problem about the existence of geometric triangulations for every cusped 3-manifold has recently be resolved by Ge in \cite{G}. 
In particular, Theorem \ref{thm3} applies for those geometric triangulations and the corresponding angle structures.

\begin{theorem}\label{thm3}
    Let $\boldsymbol{\xi}=(\xi_1,\dots,\xi_k)\in \CC^k$. Suppose Equation \eqref{NZintro} admits a solution $\mathbf{z}$ of shape parameters with $\im z_j >0$ for $j=1,\dots, N$. Let $\alpha$ be the angle structure induced by $\mathbf{z}$. Then up to a  normalization that depends only on the combinatorics of $X$ but not on $\hbar$ and $\alpha$, we have
    \begin{align*}
       W_\hbar(X,\alpha)
        = \exp\left( \frac{i}{\pi}R(\boldsymbol{\xi}) \right)\frac{\exp\left(-\frac{1}{\pi\hbar} \Vol(M; \boldsymbol{l}, \boldsymbol{\xi})\right)}{\sqrt{\det\left( \im\left(\frac{\partial \omega^{loc}_{m_p}}{\partial \omega^{loc}_{l_q}}\right)\right)\left| \mathrm{Tor}(M,\rho,\boldsymbol{l})\right|^2}} \left(1+O(\hbar)\right),
    \end{align*}
    where 
    \begin{itemize}
        \item $\Vol(M; \boldsymbol{l}, \boldsymbol{\xi}) = \sum_{j=1}^N D(z_j)$ is the volume of the representation determined by $\mathbf{z}$, and $D(z)$ denotes the Bloch-Wigner dilogarithm function;
        \item $\rho:\pi_1(M)\to\mathrm{PSL}(2;\CC)$ is the representation determined by $\mathbf{z}$ through the developing map;
        \item $\mathrm{Tor}(M,\rho,\boldsymbol{l})$ is the adjoint twisted Reidemeister torsion of $M$ associated with $\rho$ and  $\boldsymbol{l}$; 
        \item for $j=1,\dots,k$, $m_j$ is a non-trivial simple closed curve on $\mathbb{T}_j$ with algebraic intersection number $m_j\cdot l_j=1$;
        \item $(\omega^{loc}_{m_1},\dots,\omega^{loc}_{m_k})$ and $(\omega^{loc}_{l_1},\dots,\omega^{loc}_{l_k})$ are logarithmic holonomies of $m_1,\dots,m_k$ and $l_1,\dots,l_k$ as local coordinates of the $\mathrm{PSL}(2;\CC)$-character variety of $M$; and
        \item $R$ is some function independent of $\hbar$ (see Proposition \ref{defnR}).
    \end{itemize}
\end{theorem}

\begin{remark}
     The expression of the 1-loop term was first introduced by Hodgson, Kricker and Siejakowski in their study on the asymptotics of meromorphic 3D index \cite{HKS}. In Lemma \ref{HKSindep}, we will show that the expression is independent of the choice of $(\boldsymbol{m},\boldsymbol{l})$ (see also \cite[Remarks 7.3 and 7.5]{HKS}). To obtain the 1-loop term, we apply the recent result of Ming--Wu in \cite{MW} that resolves Dimofte--Garoufalidis 1-loop conjecture. See Section \ref{HKSintro} for more details.
\end{remark}

\subsection{Outline of the paper}
This paper is organized as follows. In Section \ref{secprelim}, we review the preliminary results, including the definition of the KLV partition function and the symplectic properties of the Neumann-Zagier matrices. In particular, using Proposition \ref{KLVexpress1}, Lemma \ref{gfdelta} and Lemma \ref{Kfor}, we provide geometric interpretation of the integration kernel and express the KLV partition function in terms of Neumann-Zagier matrices. 
Then in Section \ref{KLVRT}, we define the Reshetikhin--Turaev-type function in Definition \ref{defnRT} and prove Theorem \ref{thm1} (1) and (2) in Proposition \ref{KLVRTsquare} and \ref{RTprop} respectively. 
Due to the norm-squared property of the KLV partition function in terms of the Reshetikhin--Turaev-type function, we define the potential functions for the RT-type function and its conjugate counterpart, as well as the joint potential function. Then we study the properties of these potential functions, including the geometry of their critical point equations (Proposition \ref{critThurscorrespondence} and \ref{expogrowthrate}), their critical values (Proposition \ref{realpartpotential}, \ref{J1andNZ} and \ref{J2andNZ}), the concavity of their real parts (Proposition \ref{concavity} and \ref{concaveF}) and the determinant of the Hessian matrices at critical points (Proposition \ref{detHessF12}, \ref{1loopJ1}, \ref{1loopJ2} and \ref{detHessF}). Finally, we prove Theorem \ref{thm3} in Section \ref{asymana} by using the results established in the previous sections together with the saddle point approximation.

\section*{Acknowledgment}
The author would like to thank Shuang Ming, Baojun Wu and Tian Yang for inspiring discussion on the connection with their Turaev-Viro invariants. The author used ChatGPT for language editing and for conceptual discussions related to the broader context of the work, which helped to clarify and develop the presentation and perspective of the work. 
All mathematical results, proofs, and conclusions in the manuscript were independently developed by the author, who takes full responsibility for the content of the article.

\section{Preliminary}\label{secprelim}
\subsection{The KLV partition function}
Let $M$ be a cusped 3-manifolds with $k$ toroidal boundary. Let $X = \{T_1,\dots, T_N\}$ be an ideal triangulation of $M$.

\begin{definition}\label{defnanglestr}
    An angle structure $\alpha$ of $X$ is an assignment of dihedral angles in $(0,\pi)$ on edges of $T_1,\dots,T_N$ such that
    \begin{enumerate}
        \item for each tetrahedron, each pair of opposite edges share the same dihedral angles;
        \item the angle sum of the dihedral angle at each truncated triangle is equal to $\pi$; and
        \item for each edge of the triangulation, the angle sum around the edge is equal to $2\pi$.
    \end{enumerate}
    We denote the space of angle structure of $X$ by $\mathcal{A}_X$.
\end{definition}
Throughout this paper, we assume that $\mathcal{A}_X \neq \emptyset$. Given an ideal triangulation $X$, we let $\Delta_0(X), \Delta_1(X)$ and $\Delta_3(X)$ be the sets of $0$-cells, $1$-cells and $3$-cells of $X$ respectively. For $j=0,1$, we let $\RR^{\Delta_j(X)} = \{f: \Delta_j(X) \to \RR\}$ be the set of all possible assignments of real numbers on the $i$-cells of $X$. We call an element in $\RR^{\Delta_1(X)}$ a \emph{state} of $X$ and an element in $\RR^{\Delta_0(X)}$ a \emph{state potential}. Besides, we let $(\RR^{\Delta_1(X)})^*$ be the set of linear map from $\RR^{\Delta_1(X)}$ to $\RR$.

\begin{definition}
Let $X$ be an ideal triangulation.
\begin{enumerate}
 \item The linear state gauge map $b: \RR^{\Delta_0(X)} \to \RR^{\Delta_1(X)}$ is defined by 
    $$ 
    bg(e) = g(\partial_1 e) + g(\partial_2 e)
    $$
    for $g\in \RR^{\Delta_0(X)}$, where $\partial_1 e$ and $\partial_2 e$ are the two endpoints of the edge $e$.
    \item A state gauge fixing at a vertrx $v\in \Delta_0(X)$ is an element $\lambda$ in $(\RR^{\Delta_1(X)})^*$ such that 
    $$\langle \lambda, bg\rangle = g(v)$$
    for all $g \in \RR^{\Delta_0(X)}$, where $\langle \cdot , \cdot \rangle$ denotes the pairing of $(\RR^{\Delta_1(X)})^*$ and $\RR^{\Delta_1(X)}$.
\end{enumerate}
     
\end{definition}

For $j=1,\dots,N$, as shown in Figure \ref{quad}, for each pair of opposite edges, there exists a quadrilateral that separates the two edges. This provides a bijection between the three pairs of opposite edges and the three quadrilaterals. Given an edge $e$ of the triangulation and a state $s \in \RR^{\Delta_1(X)}$,  let $s(e)\in \RR$ be the value of the state evaluated at the edge $e$. We let $\square(T_j)$ be the set of quadrilaterals of $T_j$. Given a quad type $q_j \in \square(T_j)$, we let $s(q_j) = s(e) + s(\tilde e)$, where $e,\tilde e \in \Delta_1(X)$ are the edge equivalent classes of two edges separated by the quadrilateral $q_j$. Besides, we let $q_j'$ and $q_j''$ be the other two quad types as shown in Figure \ref{quad2}. Moreover, given an angle structure $\alpha\in \mathcal{A}_X$ and a tetrahedron $T_j$, we let $\alpha(q_j), \alpha(q_j')$ and $\alpha(q_j'')$ be the dihedral angles assigned to the pair of edges separated by $q_j,q_j'$ and $q_j''$ respectively. 

\begin{figure}[h]
    \centering
    \includegraphics[width=\textwidth]{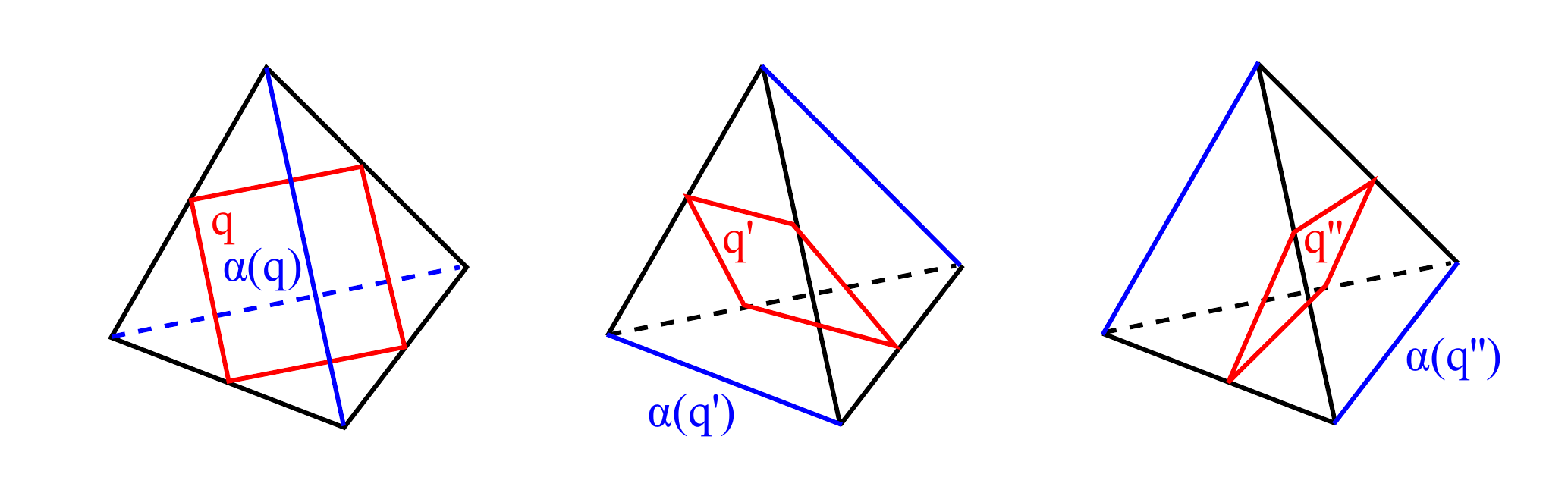}
    \caption{Once we fix a quad type $q$ as shown in the left figure, we let $q'$ and $q''$ be the two other quad types as shown in the middle and right figures.}
    \label{quad2}
\end{figure}

\begin{definition}\label{defnKLV}
    Let $\B \in \RR_{>0}$ and $\hbar = (\B+\B^{-1})^{-2}$.
    \begin{enumerate}
    \item The Faddeev's quantum dilogarithm function $\Phi_\B$ is the holomorphic function defined on $\RR + i \left(-\frac{1}{2\sqrt{\hbar}},\frac{1}{2\sqrt{\hbar}}\right)$ 
     given by $$
    \Phi_\B (z)
    = \exp\left(
    \int_{\RR+i0^+}
    \frac{e^{-2izw}dw}{4\sinh(w\B)\sinh(w/\B)w}
    \right),
    $$
    where $\RR+i0^+$ denotes any contour of the form $(-\infty,-\epsilon] \cup \{\epsilon e^{-i\theta} \mid \theta \in [-\pi,0]\} \cup [\epsilon,\infty)$ with sufficiently small $\epsilon>0$.
    \item The hyperbolic gamma function $\gamma^{(2)}_\B
    $ is the holomorphic function defined by
    $$
    \gamma^{(2)}(z)
    = \frac{\exp\left(
    -\pi i B_{2,2}(z)/2
    \right)}{\Phi_{\B}
    \left(
    \frac{\B^2+1-2z}{2i\B}\right)},
    $$
    where 
    $$ B_{2,2}(z)
    =
    \frac{z^2}{\B^2} - \left( 1+\frac{1}{\B^2}\right)z
    + \frac{1}{6}\left(1+\frac{1}{\B^2}\right)
    + \frac{1}{2}.$$ 
        \item Let $X=\{T_1,\dots,T_N\}$ be an ideal triangulation of $M$ with an angle structure $\alpha\in \mathcal{A}_X$. Let $j=1,\dots,N$. Given a state $s\in \RR^{\Delta_1(X)}$, let $s|_{\Delta_1(T_j)}$ be the restriction of $s$ on the edges of $T_j$. The Boltzmann weight associated with the tetrahedron $T_j$ is given by
    $$
    B(T_j,s|_{T_j})
    = \prod_{q \in \square(T_j)} \gamma^{(2)}\left(\frac{\B^2+1}{\pi}\alpha(q)
    + i\B(s(q') - s(q''))\right).
    $$
    \item Given $\alpha \in \mathscr{A}_X,$
    the Kashaev-Luo-Vartanov partition function (or in short, the KLV partition function) of $(X,\alpha)$ is defined by
    $$
    W_\B(X,\alpha)
    = \int_{\mathbf{s} \in \RR^{\Delta_1(X)}}
    \prod_{T\in \Delta_3(X)} B(T, s|_{\Delta_1(T)} )\delta(\langle \boldsymbol{\lambda},\mathbf{s} \rangle)d\mathbf{s},
    $$
    where $\boldsymbol{\lambda}=(\lambda_{v_1},\dots,\lambda_{v_k})$ is any gauge fixing at the vertices $v_1,\dots,v_k \in \Delta_0(X)$ and $\delta$ is the Dirac delta distribution.
    \end{enumerate}
\end{definition}
\begin{remark}
    The function $\gamma^{(2)}(z)$ in Definition \ref{defnKLV}(3) is the function $\gamma^{(2)}(z;\omega_1,\omega_2)$ in \cite[Equation (30)]{KLV} with $\omega_1 = \B^2$ and $\omega_2 = 1$. Since $\gamma^{(2)}(z;\omega_1,\omega_2)$ is invariant under simultaneous rescaling $z, \omega_1$ and $\omega_2$, up to rescaling we can assume that $\omega_1 = \B^2$ and $\omega_2 = 1$. Thus, the definition of the KLV partition function above agrees with that in \cite[Equation (9)]{KLV}. 
\end{remark}

In \cite[Theorem 1]{KLV}, it is proved that the partition function $W_\B(X,\alpha)$ is an absolutely convergent integral. Moreover, it is independent of the choice of the state gauge fixing, invariant under
shaped 3-2 Pachner moves, and invariant under the shape gauge transformations induced by interior edges.

We recall the following properties satisfied by the quantum dilogarithm function.
\begin{proposition}[Properties of $\Phi_\B$]\label{prop:quant:dilog}
\

\begin{enumerate}
\item (unitarity) For any $\B \in \RR_{>0}$ and any  $z \in \RR + i \left (\frac{-1}{2 \sqrt{\hbar}}, \frac{1}{2 \sqrt{\hbar}}\right )$, 
$$\overline{\Phi_\B(z)} = \frac{1}{\Phi_\B(\overline{z})}.$$
\item (semi-classical limit in $\B$) For any $z \in \RR + i \left (-\pi,\pi \right )$,
$$\Phi_\B\left (\frac{z}{2 \pi \B}\right ) = \exp\left (\frac{-i}{2 \pi \B^2} \Li (- e^z)\right ) \left ( 1 + O_{\B \to 0^+}(\B^2)\right ).$$
\item \cite[Lemma 2.11]{BAGPN}
Let $\B \in \RR_{>0}$ and $a,b,c \in (0,\pi)$ such that $a+b+c=\pi$. 
Let 
$$\varphi_{c,b}(t) = e^{\frac{1}{\sqrt{\hbar}}c t}
    \Phi_\B\left(t + \frac{i}{2\sqrt{\hbar}}(\pi - a) \right).$$
Then
    $$
	 \left |
	\varphi_{c,b}(t)
	\right | 
	\underset{\RR \ni x \to -\infty}{\sim} e^{\frac{1}{ \sqrt{\hbar}} c t}
   \quad\text{and}\quad
   \left |
	\varphi_{c,b}(t)
	\right | 
	\underset{\RR \ni x \to +\infty}{\sim} e^{-\frac{1}{ \sqrt{\hbar}} b t}.
	$$

\end{enumerate}
\end{proposition}

In particular, Proposition \ref{prop:quant:dilog}(2) implies that the quantum dilogarithm function can be approximated using the classical dilogarithm function. The next two propositions from \cite{BAW} provide estimations of  the error terms.
\begin{proposition}\label{prop:quant:dilog:uniform}\cite[Proposition 2.5]{BAW}
There exist constants $C, C'>0$ such that for all $\delta \in (0,\pi)$, $\B \in(0,\sqrt{\delta})$ and 
$y \in \RR+i[-(\pi-\delta),\pi-\delta]$, we have the following estimations.
    \begin{enumerate}
        \item (uniform semi-classical limit with $\B$)
        $$
        e^{-
        \left (\frac{C}{\delta} + C' \right) \B^2}
        \leqslant
        \left | 
        \Phi_\B\left (\frac{y}{2 \pi \B}\right )  \exp\left (\frac{i}{2 \pi \B^2} \Li (- e^y)\right )
        \right | \leqslant 
        e^{
        \left (\frac{C}{\delta} + C' \right) \B^2},
        $$
        \item similar inequalities hold if one replaces $y$ with $y(1+\B^2)$:
        $$
        e^{-
        \left (\frac{C}{\delta} + C' \right) \B^2}
        \leqslant
        \left | 
        \Phi_\B\left (\frac{y(1+\B^2)}{2 \pi \B}\right )  \exp\left (\frac{i}{2 \pi \B^2} \Li (- e^{y(1+\B^2)})\right )
        \right | \leqslant 
        e^{
        \left (\frac{C}{\delta} + C' \right) \B^2},
        $$
        \item (comparison of classical dilogarithms with $\B$ and $\hbar$)
        $$
        e^{-C_\delta}
        \leqslant
        \left | 
        \exp\left (\frac{-i}{2 \pi \B^2}(1+\B^2)^2 \Li (- e^{y})\right )
        \exp\left (\frac{i}{2 \pi \B^2} \Li (- e^{y(1+\B^2)})\right )
        \right | \leqslant 
        e^{C_\delta},
        $$
        where $C_\delta>0$ is a constant independant of $\B$ (see \cite[Lemma 7.17]{BAGPN}).
        \item (uniform semi-classical limit with $\hbar$)
        $$
        e^{-
        \left (\frac{C}{\delta} + C' \right) \B^2-C_\delta}
        \leqslant
        \left | 
        \Phi_\B\left (\frac{y}{2 \pi \sqrt{\hbar}}\right )  \exp\left (\frac{i}{2 \pi \hbar} \Li (- e^{y})\right )
        \right | \leqslant 
        e^{
        \left (\frac{C}{\delta} + C' \right) \B^2+C_\delta},
        $$
    \end{enumerate}
Moreover, the constants $C, C'$ can be computed explicitly (see \cite[Section 7]{BAGPN}).
\end{proposition}

\begin{proposition}\label{semihbar} \cite[Proposition 2.6]{BAW} For any $z \in \RR + i \left (-\pi,\pi \right )$,
$$\Phi_\B\left (\frac{z}{2 \pi \sqrt{\hbar}}\right ) = \exp\left (\frac{i}{2\pi}z\log(1+e^z) + \frac{i}{\pi}\Li \left(- e^{z}\right) - \frac{i}{2 \pi \hbar}  \Li \left(- e^{z}\right) \right ) \left ( 1 + O_{\hbar \to 0^+}(\hbar)\right ).$$
Furthermore, the convergence is uniform on any compact subset of $\RR+i(-\pi,\pi)$.
\end{proposition}

 Let $\Delta_1(X) = \{e_1,\dots,e_n\}$ be the set of edges of $X$. 
For each tetrahedron $T_j$, we assign shape parameters $$z_j, \quad z_j'=\frac{1}{1-z_j},\quad  z_j'' = 1-\frac{1}{z_j}$$ to the edges of $T_j$ such that the opposite edges share the same shape parameters and, when the tetrahedron is viewed from outside, around each vertex the shape parameters $z_j, z_j',z_j''$ are in counterclockwise order as shown in Figure \ref{quad}. Let $\mathbf{G_e}\in M_{N\times N}(\ZZ)$ be the full edge incidence matrix whose $(i,j)$-th entry equals the number of edges with shape parameter $z_j$ that is incident to the edge $e_j$. Define $\mathbf{G_e'}$ and $\mathbf{G_e''}$ be respectively the corresponding incidence matrices with respect to the shape parameters $z'$ and $z''$. Let $\mathbf{A_e} = \mathbf{G_e} - \mathbf{G_e'}$ and $\mathbf{B_e} = \mathbf{G_e''}-\mathbf{G_e'}$. Note that the definitions of $\mathbf{A_e}$ and $\mathbf{B_e}$ depend on the choice of quad type. For $j=1,\dots, N$, we let $a_j= \alpha(q_j), a_j'=\alpha(q_j')$ and $a_j''=\alpha(q_j'')$. With these notations, the KLV partition function can be written as follows. 
\begin{proposition}\label{KLVexpress1}
    The KLV partition function is given by
    $$
    W_\B(X,\alpha)
    = \int_{(\mathbf{t},\mathbf{\hat{t}})\in\RR^{2N}}
     \mathscr{K}(X,\mathbf{t},\mathbf{\hat{t}}) \mathscr{D}(X,\mathbf{t},\mathbf{\hat{t}}, \alpha)
    d\mathbf{t} d\mathbf{\hat{t}},
    $$
    where $\mathbf{t}=(t_1,\dots,t_N)^{\!\top}, \mathbf{\hat{t}}=(\hat{t}_1,\dots,\hat{t}_N)^{\!\top}$,
    $$
    \mathscr{K}(X,\mathbf{t},\mathbf{\hat{t}})
    = \int_{\mathbf{s}\in \RR^N} e^{\pi i(\mathbf{t}+\mathbf{\hat{t}})^{\!\top}\mathbf{A_e^{\!\top}\mathbf{s}}} \delta(-\mathbf{B_e^{\!\top}}\mathbf{s} - \mathbf{t}+\mathbf{\hat{t}})\delta(\langle \boldsymbol{\lambda},\mathbf{s}\rangle) d\mathbf{s} 
    $$
    with $\mathbf{s}=(s_1,\dots,s_N)^{\!\top}$ and 
    $$
    \mathscr{D}(X,\mathbf{t},\mathbf{\hat{t}}, \alpha)
    = \prod_{j=1}^N \overline{\varphi}_{a_j'', a_j'}(t_j) \prod_{j=1}^N \varphi_{a_j'', a_j'}(\hat{t}_j)
    $$
    with
    $$
    \varphi_{(a_j'',a_j')}(z)
    =  \frac{e^{\frac{1}{\sqrt{\hbar}}a_j'' z}}{\Phi_\B\left(z-\frac{i}{2\sqrt{\hbar}}(\pi - a_j) \right)}.
    $$
\end{proposition}
\begin{proof}
We follow the method of computing the KLV partition function introduced in \cite[Section 4.2]{KLV}. Note that in \cite{KLV}, the quad type $q,q'$ and $q''$ are in clockwise direction, while in this paper they are in counter-clockwise direction, which matches with the usual convention used in hyperbolic geometry. 
From \cite[Equation (59)]{KLV}, the Boltzmann weight restricted on each tetrahedron can be written as 
\begin{align}\label{mcalBfor}
B(T,s|_{\Delta_1(T)})
=
\int_{\RR^2}
\mathcal{K}\big(s(q') - s(q''),s(q) - s(q'),t,\hat{t}\big)\mathcal{M}_{\alpha(q''), \alpha(q')}\big(t,\hat{t}\big)dtd\hat{t},
\end{align}
where 
\begin{enumerate}
    \item $\alpha(q')$ and $\alpha(q'')$ are the dihedral angles assigned to the quad type $q'$ and $q''$ of $T$ respectively (see Figure \ref{quad2});
    \item $\mathcal{K}$ is the tempered distribution given by
    \begin{align*} 
&\ \mathcal{K}\big({s}(q') - {s}(q''),{s}(q) - \tilde{s}(q'),t,\hat{t}\big)
=  e^{\pi i (t+\hat{t}) ({s}(q) - {s}(q'))} \delta\big(({s}(q') - {s}(q''))-t + \hat{t}\big) ,
\end{align*}
where $\delta$ denotes the Dirac delta distribution; and 
\item $\mathcal{M}$ is the function given by
$$
\mathcal{M}_{\alpha(q''),\alpha(q')}(t,\hat{t})
= \overline{\varphi_{\alpha(q''),\alpha(q')}(t)}\varphi_{\alpha(q''),\alpha(q')}\big(\hat{t}\big)
$$
with 
$$
\varphi_{\alpha(q''),\alpha(q')}(z)
=
\frac{e^{\frac{1}{\sqrt{\hbar}}\alpha(q'') z}}{\Phi_\B\left(z-\frac{i}{2\sqrt{\hbar}}(\pi - \alpha(q)) \right)} .
$$
\end{enumerate} 
Let $\boldsymbol{\lambda}$ be any gauge fixing at the ideal vertices $v_1,\dots,v_k$. For any angle structure $\alpha\in \mathcal{A}_X$, by Fubini's theorem, we have
\begin{align*}
 &\ W_b(X,\alpha) \\
    =&\ \int_{\mathbf{s} \in \RR^{\Delta_1(X)}} \int_{(\mathbf{t},\mathbf{\hat{t}})\in \RR^{2N}}
    \prod_{j=1}^N\mathcal{K}\big({s}(q_k') - {s}(q_k''),{s}(q_k) - \tilde{s}(q_k'),t_j,\hat{t}_j\big)
    \mathcal{M}_{\alpha(q_k''),\alpha(q_k')}\big(t_j,\hat{t}_j\big)  \delta(\langle \boldsymbol{\lambda},\mathbf{s} \rangle) d\mathbf{t}d\mathbf{\hat{t}}d\mathbf{s}\\
    = &\ \int_{(\mathbf{t},\mathbf{\hat{t}})\in\RR^{2N}}
     \mathscr{K}\big(X,\mathbf{t},\mathbf{\hat{t}}\big) \mathscr{D}\big(X,\mathbf{t},\mathbf{\hat{t}}, \alpha\big)
    d\mathbf{t} d\mathbf{\hat{t}}, 
\end{align*}
where
$$
    \mathscr{K}\big(X,\mathbf{t},\mathbf{\hat{t}}\big)
    = \int_{\mathbf{s}\in \RR^{\Delta_1(X)}} \prod_{j=1}^N\mathcal{K}\big({s}(q_k') - {s}(q_k''),{s}(q_k) - \tilde{s}(q_k'),t_j,\hat{t}_j\big)
     \delta(\langle \boldsymbol{\lambda},\mathbf{s} \rangle) d\mathbf{s} 
    $$
    and 
    $$
    \mathscr{D}\big(X,\mathbf{t},\mathbf{\hat{t}}, \alpha\big)
    = \prod_{j=1}^N \overline{\varphi}_{a_j'', a_j'}\big(\hat{t}_j\big) \prod_{j=1}^N \varphi_{a_j'', a_j'}(t_j)
    $$
Finally, notice that 
\begin{align*}
    \begin{pmatrix}
        s(q_1) \\
        \vdots \\
        s(q_N)
    \end{pmatrix}
    = \mathbf{G}_{\boldsymbol{e}}^{\!\top} \begin{pmatrix}
        s(e_1)\\
        \vdots\\
        s(e_N)
    \end{pmatrix}, \ 
    \begin{pmatrix}
        s(q_1') \\
        \vdots \\
        s(q_N')
    \end{pmatrix}
    = (\mathbf{G}_{\boldsymbol{e}}')^{\!\top} \begin{pmatrix}
        s(e_1)\\
        \vdots\\
        s(e_N)
    \end{pmatrix}, \
    \begin{pmatrix}
        s(q_1'') \\
        \vdots \\
        s(q_N'')
    \end{pmatrix}
    = (\mathbf{G}_{\boldsymbol{e}}'')^{\!\top} \begin{pmatrix}
        s(e_1)\\
        \vdots\\
        s(e_N)
    \end{pmatrix},
\end{align*}
which together imply that
\begin{align*}
    \begin{pmatrix}
        s(q_1')-s(q_1'') \\
        \vdots \\
        s(q_N')-s(q_N'')
    \end{pmatrix}
    = -\mathbf{B_e^{\!\top}} \begin{pmatrix}
        s(e_1)\\
        \vdots\\
        s(e_N)
    \end{pmatrix}, \quad
    \begin{pmatrix}
        s(q_1)-s(q_1') \\
        \vdots \\
        s(q_N)-s(q_N')
    \end{pmatrix}
    = \mathbf{A_e^{\!\top}} \begin{pmatrix}
        s(e_1)\\
        \vdots\\
        s(e_N)
    \end{pmatrix}.
\end{align*}
This gives the desired result.
\end{proof}

\subsection{Neumann-Zagier matrices and its symplectic properties}
Recall that given an assignment of shape parameters $\mathbf{z}\in (\CC\setminus\{0,1\})^N$, the logarithmic holonomy around an edge $e$ of the triangulation, denoted by $\mathrm{H}^{\CC}_{X,e}(\mathbf{z})$, is given by the sum of the logarithm of the shape parameters around $e$. For example, for the edge $e_1$ in Figure \ref{edgehol}, we have
$$ \mathrm{H}^{\CC}_{X,e_1}(\mathbf{z}) = 2\Log z_1 + \Log z_2' + \Log z_2'' + \Log z_3'. $$
Besides, given a non-trivial simple closed curve $\gamma$ on a boundary torus, the logarithmic holonomy of $\gamma$, denoted by $\mathrm{H}^{\CC}_{X,\gamma}(\mathbf{z})$, is given by
the sum of the logarithm of shape parameters on the left of $\gamma$ minus the sum of the logarithm of shape parameters on the right of $\gamma$. For example, for the red curve $l_1$ in Figure \ref{edgehol}, we have 
$$\mathrm{H}^{\CC}_{X,l_1}(\mathbf{z}) = \Log z_1 - \Log z_2' + \Log z_3' - \Log z_1 + \Log z_2'' - \Log z_3'
= - \Log z_2' + \Log z_2'' .$$

Now, for $j=1,\dots,k$, let $l_j$ be a non-trivial simple closed curve on the $j$-th boundary torus $\mathbb{T}_j$.
Let $\{e_1,\dots, e_{N-k}\}$ be a system of independent edges. 
Let $\mathbf{G}, \mathbf{G'}$ and $\mathbf{G''}$ be the $N$ by $N$ matrices whose first $N-k$ rows are rows of $\mathbf{G_e}, \mathbf{G_e'}$ and $\mathbf{G_e''}$, and whose last $k$ rows are rows $\mathbf{G}_{\boldsymbol{l}},\mathbf{G}_{\boldsymbol{l}}'$ and $\mathbf{G}_{\boldsymbol{l}}''$. 
Let $\mathbf{A} = \mathbf{G}-\mathbf{G'}$ and $\mathbf{B}=\mathbf{G''}-\mathbf{G'}$. It is known that by suitably choosing the quad type, one can assume that $\mathbf{B}$ is invertible (see \cite[Lemma A.3(b)]{DG}). Throughout the rest of the paper, we  assume that $\mathbf{B}$ is invertible. With these notations, using the equation that $\Log z_j+ \Log z_j' + \Log z_j''=\pi i$ for $j=1,\dots,N$, we can eliminate $\Log z_1',\dots, \Log z_N'$ and rewrite Equation \eqref{NZintro} in the form
\begin{align*}
    \mathbf{A}\BLog\mathbf{z} + \mathbf{B} \BLog\mathbf{z''} = i\boldsymbol{\nu}+\boldsymbol{ u},
\end{align*}
where $\boldsymbol{\nu} \in \pi\ZZ^N$ and $\boldsymbol{u} = (0,\dots,0,\xi_1,\dots,\xi_k)^{\!\top}$.
We write $\mathbf{A}, \mathbf{B}$ and $\mathbf{B}^{-1}$ schematically in the forms
\begin{align}\label{ABBinvdecompose}
\mathbf{A}
=
\begin{pmatrix}
    \mathbf{A}_{N-k} \\
    \mathbf{l}_\mathbf{A}
\end{pmatrix},\quad 
\mathbf{B}
= \begin{pmatrix}
    \mathbf{B}_{N-k} \\ \mathbf{l}_\mathbf{B}
\end{pmatrix}
\quad\text{and}\quad
\mathbf{B}^{-1}
= \begin{pmatrix}
    \mathbf{B}^{-1}_{N-k} & \mathbf{l}_\mathbf{B}^{-1}
\end{pmatrix},
\end{align}
where $\mathbf{A}_{N-k},\mathbf{B}_{N-k} \in M_{N-k,N}(\ZZ)$, $\mathbf{B}^{-1}_{N-k} \in M_{N,N-k}(\QQ)$, $\mathbf{l}_{A},\mathbf{l}_{B} \in M_{k,N}(\ZZ)$ and $\mathbf{l}_{\mathbf{B}}^{-1} \in M_{N,k}(\QQ)$. Note that they satisfy 
$$
\mathrm{Id}_N 
= \mathbf{B} \mathbf{B}^{-1}
= \begin{pmatrix}
    \mathbf{B}_{N-k} \mathbf{B}_{N-k}^{-1} & \mathbf{B}_{N-k} \mathbf{l}_\mathbf{B}^{-1} \\
    \mathbf{l}_{\mathbf{B}}\mathbf{B}_{N-k}^{-1} & \mathbf{l}_{\mathbf{B}} \mathbf{l}_{\mathbf{B}}^{-1}
\end{pmatrix}
$$
and
$$
\mathrm{Id}_N 
= \mathbf{B}^{-1}\mathbf{B} 
= \mathbf{B}_{N-k}^{-1} \mathbf{B}_{N-k} + \mathbf{l}_{\mathbf{B}}^{-1} \mathbf{l}_{\mathbf{B}}.
$$
These equations imply that 
\begin{align}
    \mathbf{B}_{N-k}\mathbf{B}_{N-k}^{-1} &= \mathrm{Id}_{N-k} \label{BB^-1}, \\
    \mathbf{l}_{\mathbf{B}}\mathbf{l}_{\mathbf{B}}^{-1} &= \mathrm{Id}_{k}, \\
    (\mathbf{l}_{\mathbf{B}}^{-1})^{\!\top} \mathbf{B}_{N-k}^{\!\top} &= \mathbf{0} \label{lBTT},
\end{align}
where $\mathrm{Id}_{n}$ denotes the $n$ by $n$ identity matrix and the matrix $\mathbf{0}$ in \eqref{lBTT} denotes the $k$ by $N-k$ zero matrix. 

From the result in \cite{NZ}, the augmented matrix $\begin{pmatrix}
    \mathbf{A} & \mathbf{B}
\end{pmatrix}$ is the upper part of an symplectic matrix. Moreover, geometrically, one can also construct symplectic duals by using the Neumann--Zagier data of genuine simple closed curves as follows. For each $j=1,\dots,k$, let $m_j$ be a simple closed curve on the $j$-th toroidal boundary such that the algebraic intersection number of $m_j$ and $l_j$ is 1. Assume that for $j=1,\dots,k$, we have
$$
\mathrm{H}^{\CC}_{X,m_j}(\mathbf{z})
= \big(c_1^j,\dots,c_N^j\big) \cdot \BLog \mathbf{z} + \big(d_1^j,\dots, d_N^j\big)\cdot \BLog\mathbf{z''}
- i \pi \nu_{m_j}
$$
for some $c_1^j,\dots,c_N^j,d_1^j,\dots,d_N^j,\nu_{m_j} \in \ZZ$. By the result in \cite{NZ}, 
there exists a $N$ by $2N$ augmented matrix $\begin{pmatrix}
    \mathbf{C} & \mathbf{D} 
\end{pmatrix}$ whose last $k$ rows are $\big(c^j_1,\dots, c^j_N, d^j_1,\dots, d^j_N\big)$ for $j=1,\dots, k$ such that the matrix
$$
\begin{pmatrix}
    \mathbf{A} & \mathbf{B} \\
    \mathbf{C} & \mathbf{D}
\end{pmatrix}
$$
is symplectic, i.e. $M^{\!\top}\Omega M = \Omega$ for the anti-symmetric form
$$
\Omega=
\begin{pmatrix}
    0& \mathrm{Id_N} \\
    -\mathrm{Id_N} & 0
\end{pmatrix}.
$$ 
Note that if $M$ is symplectic, $M^{\!\top}$ will also be symplectic. In particular, by expanding the equations $M^{\!\top}\Omega M = \Omega$ and $M\Omega M^{\!\top} = \Omega$, one can check that $\mathbf{A}\mathbf{D}^{\!\top} - \mathbf{B}\mathbf{C}^{\!\top}
    = \mathrm{Id}_N$, and that the matrices $\mathbf{A}\mathbf{B}^{\!\top}$ and $\mathbf{B}^{\!\top}\mathbf{D}$ are symmetric. Under the assumption that $\mathbf{B}^{-1}$ is invertible, the equation $$\mathbf{A}\mathbf{B}^{\!\top} = (\mathbf{A}\mathbf{B}^{\!\top})^{\!\top} = \mathbf{B}\mathbf{A}^{\!\top}$$
implies that 
$$\mathbf{B}^{-1}\mathbf{A} = \mathbf{A}^{\!\top} (\mathbf{B}^{-1})^{\!\top}
= (\mathbf{B}^{-1}\mathbf{A})^{\!\top}.$$
Similarly, the equation
$$
\mathbf{B}^{\!\top}\mathbf{D}
= (\mathbf{B}^{\!\top}\mathbf{D})^{\!\top}
= \mathbf{D}^{\!\top} \mathbf{B}
$$
implies that
$$
\mathbf{D}\mathbf{B}^{-1}
= (\mathbf{B}^{-1})^{\!\top} \mathbf{D}^{\!\top}
= (\mathbf{D}\mathbf{B}^{-1})^{\!\top}.
$$
Altogether, the matrices $\mathbf{B}^{-1}\mathbf{A}$ and $\mathbf{D}\mathbf{B}^{-1}$ are symmetric. 

Note that in the special case where $\mathbf{B}$ is invertible, algebraically the augmented matrix $(\mathbf{A} \quad \mathbf{B})$ admits a natrual symplectic completion given by
$$
\begin{pmatrix}
    \mathbf{A} & \mathbf{B} \\
    -(\mathbf{B^{-1}})^{\!\top} & \mathbf{0}
\end{pmatrix}
$$
which satisfies
$$
\begin{pmatrix}
    \mathbf{A} & \mathbf{B} \\
    -(\mathbf{B^{-1}})^{\!\top} & \mathbf{0}
\end{pmatrix}^{\!\top}
\begin{pmatrix}
    \mathbf{0} & \mathrm{Id}_N \\
    -\mathrm{Id}_N & \mathbf{0}
\end{pmatrix}
\begin{pmatrix}
    \mathbf{A} & \mathbf{B} \\
    -(\mathbf{B^{-1}})^{\!\top} & \mathbf{0}
\end{pmatrix}
=
\begin{pmatrix}
    \mathbf{0} & \mathrm{Id}_N \\
    -\mathrm{Id}_N & \mathbf{0}
\end{pmatrix}.
$$
In particular, the last $k$ rows of the matrix $\begin{pmatrix}
-(\mathbf{B}^{-1})^{\!\top} & \mathbf{0}    
\end{pmatrix}$ give an algebraic dual to the $k$ simple closed curves $l_1,\dots, l_k$ with respect to the symplectic pairing. Note that with the notations introduced in Equations \eqref{ABBinvdecompose}, we have
$$
(\mathbf{B}^{-1})^{\!\top}
= \begin{pmatrix}
    (\mathbf{B}_{N-k}^{-1})^{\!\top}\\
    (\mathbf{l_B^{-1}})^{\!\top}
\end{pmatrix}.
$$
We write 
$$ (\mathbf{l}_\mathbf{B}^{-1})
=
\begin{pmatrix}
    (\mathbf{l}_1^{-1})  & \dots & (\mathbf{l}_{k}^{-1})
\end{pmatrix},
$$
where $(\mathbf{l}_1^{-1}),\dots, (\mathbf{l}_k^{-1}) \in M_{N,1}(\QQ)$. Moreover, we write $\BLog \mathbf{z} = (\Log z_1,\dots, \Log z_N)^{\!\top}, \BLog \mathbf{z'} = (\Log z_1',\dots, \Log z_N')^{\!\top}$, $ \BLog \mathbf{z''} = (\Log z_1'',\dots, \Log z_N'')^{\!\top}$ and $\boldsymbol{\pi}=(\pi,\dots, \pi)^{\!\top}$. Given any angle structure $\alpha$, we write $\boldsymbol{a}=(a_1,\dots,a_N)^{\!\top}$.
Lemma \ref{n1n2defn} shows that the vectors $(\mathbf{l}_1^{-1})^{\!\top}, \dots, (\mathbf{l}_k^{-1})^{\!\top}$ can be interpreted as the Neumann--Zagier data of peripheral curves with rational coefficients.
\begin{lemma}\label{n1n2defn}
There exist a set of rational numbers $\{n_{p,q}\}_{p,q=1}^k$ with $n_{p,q}=n_{q,p}$ for $p,q=1,\dots,k$, and another set of rational numbers $\{n_2^1,\dots, n_2^k\}$ such that the following properties hold.
\begin{enumerate}
    \item For any $\boldsymbol{\xi}=(\xi_1,\dots,\xi_k) \in \CC^k$ and for any shape parameters $\mathbf{z} = \mathbf{z}(\boldsymbol{\xi})$ satisfying the equation
    $$ \mathbf{A}\BLog\mathbf{z} + \mathbf{B} \BLog\mathbf{z''} = i\boldsymbol{\nu}+\boldsymbol{u},$$
    where $\boldsymbol{u}=(0,\dots,0,\xi_1,\dots,\xi_k)^{\!\top}$, we have for $j=1,\dots,k$,
    $$
    2(\mathbf{l}_{j}^{-1})^{\!\top}(-\BLog \mathbf{z} + i\boldsymbol{\pi}) 
    = \mathrm{H}^\CC_{X,m_j}(\mathbf{z})
    +\sum_{p=1}^k n_{p,j} \xi_p
    + in_2^j\pi,
    $$
    where $\mathrm{H}^\CC_{X,m_j}(\mathbf{z})$ is the logarithmic holonomy of $m_j$ with respect to $\mathbf{z}$.
    \item For any angle structure $\alpha$, we have
    $$
    2(\mathbf{l}_{j}^{-1})^{\!\top}(- \boldsymbol{a} + \boldsymbol{\pi}) 
    = \mu_{j}(\alpha) 
    + \sum_{p=1}^{k}n_{p,j}\lambda_{p}(\alpha) 
    + n_2^j\pi,
    $$
    where $\mu_j(\alpha)$ and $\lambda_p(\alpha)$ are the angular holonomies along $m_j$ and $l_p$ respectively.
    \item We have
    $$
    2(\mathbf{l}_{j}^{-1})^{\!\top}(-\BLog \mathbf{z} + i\boldsymbol{a})
    = \mathrm{H}^\CC_{X,m_j}(\mathbf{z}) - i\mu_{j}(\alpha)
    + \sum_{p=1}^{k}n_{p,j}(\xi_p - i\lambda_{p}(\alpha) )
    $$
    and
    $$
    2(\mathbf{l}_{j}^{-1})^{\!\top}(-\im\BLog \mathbf{z} + \boldsymbol{a})
    = \im \mathrm{H}^\CC_{X,m_j}(\mathbf{z}) - \mu_{j}(\alpha)
    + \sum_{p=1}^{k}n_{p,j}(\im \xi_p - \lambda_{p}(\alpha) ).
    $$
\end{enumerate}
\end{lemma}
\begin{proof}
    Since $\mathbf{A}\mathbf{D}^{\!\top} - \mathbf{B}\mathbf{C}^{\!\top}
    = \mathrm{Id}_N$, we have $\mathbf{B}^{-1} = (\mathbf{B}^{-1}\mathbf{A})\mathbf{D}^{\!\top} - \mathbf{C}^{\!\top}$. Besides, recall that $\mathbf{B}^{-1}\mathbf{A}$ is symmetric. As a result, we have $(\mathbf{B}^{-1})^{\!\top} = \mathbf{D} (\mathbf{B}^{-1}\mathbf{A}) - \mathbf{C}$. In particular , we have
    \begin{align*}
        (\mathbf{B}^{-1})^{\!\top}(-\BLog\mathbf{z} + i\boldsymbol{\pi})
        &= -\left(\mathbf{D}(\mathbf{B}^{-1}\mathbf{A} \BLog \mathbf{z}) - \mathbf{C}\BLog\mathbf{z}\right) + (\mathbf{B}^{-1})^{\!\top}(i\boldsymbol{\pi}).
     \end{align*}
    Since $\mathbf{A}\BLog\mathbf{z} + \mathbf{B} \BLog\mathbf{z''} = i\boldsymbol{\nu}+\boldsymbol{u}$, we have
    $
    \mathbf{B}^{-1}\mathbf{A}\BLog\mathbf{z} = 
    - \BLog\mathbf{z''} + \mathbf{B}^{-1}(i\boldsymbol{\nu}+\boldsymbol{u}).
    $ Thus,
    \begin{align*}
        (\mathbf{B}^{-1})^{\!\top}(-\BLog\mathbf{z} + i\boldsymbol{\pi})
        &=  \mathbf{C}\BLog\mathbf{z} + 
        \mathbf{D} \BLog\mathbf{z''}
        -\mathbf{D}\mathbf{B}^{-1}(i\boldsymbol{\nu}+\mathbf{u}) + (\mathbf{B}^{-1})^{\!\top}(i\boldsymbol{\pi}).
    \end{align*}
    Recall that the last $j$ entries of $\mathbf{C}\BLog \mathbf{z} + \mathbf{D}\BLog\mathbf{z''} $ are 
    $$ (\mathrm{H}^{\CC}_{X,m_1}(\mathbf{z}) + i \pi \nu_{m_1})/2, \dots, (\mathrm{H}^{\CC}_{X,m_k}(\mathbf{z}) + i \pi \nu_{m_k})/2.$$
    By considering the last $j$ entries of the equation above, for $j=1,\dots,k$ we have
    $$
    2(\mathbf{l}_{j}^{-1})^{\!\top}
    (-\BLog\mathbf{z} + i\boldsymbol{\pi})
    = \mathrm{H}^\CC_{X,m_j}(\mathbf{z}) 
    + \sum_{p=1}^k n_{p,j}\xi_p
    + in_2^j\pi,
    $$
    where $n_2^j \in \QQ$, and $n_{p,j}$'s are the $(p,j)$-th entry of the bottom right $k$ by $k$ block of the matrix $\mathbf{D}\mathbf{B}^{-1}$.
    Since $\mathbf{D}\mathbf{B}^{-1}$ is symmetric, we have $n_{p,j} = n_{j,p}$ for $j,p=1,\dots, k$. This proves the first claim.

    For the second claim, by the same argument as above, we have
    \begin{align*}
        (\mathbf{B}^{-1})^{\!\top}(-\boldsymbol{a} + \boldsymbol{\pi})
        &=  \mathbf{C}\boldsymbol{a}+ 
        \mathbf{D} \boldsymbol{a''}
        -\mathbf{D}\mathbf{B}^{-1}(\boldsymbol{\nu}+(0,\dots,0,\lambda_1(\alpha), \dots, \lambda_{k}(\alpha)))^{\!\top} + (\mathbf{B}^{-1})^{\!\top}\boldsymbol{\pi}.
    \end{align*}
    This implies the second result. The last result follows directly from the first two results.
\end{proof}

\begin{lemma}\label{n1n2defn2}
    With the same notations as in Lemma \ref{n1n2defn}, the following statements hold.
    \begin{enumerate}
    \item For any shape parameters $\mathbf{w}$ that satisfy the equation
    $$ \mathbf{A}\BLog\mathbf{w} + \mathbf{B} \BLog\mathbf{w''} = -i\boldsymbol{\nu}-\boldsymbol{u},$$
    where $\boldsymbol{u}=(0,\dots,0,\xi_1,\dots,\xi_k)^{\!\top}$, we have
    $$
    2(\mathbf{l}_{j}^{-1})^{\!\top}(-\BLog \mathbf{w} - i\boldsymbol{\pi}) 
    = \mathrm{H}^\CC_{X,m_j}(\mathbf{w})
    -\sum_{p=1}^k n_{p,j}\xi_p
    - in_2^j\pi.
    $$
    \item we have
    $$
    2(\mathbf{l}_{j}^{-1})^{\!\top}(-\BLog \mathbf{w} - i\boldsymbol{a}) 
    = \mathrm{H}^\CC_{X,m_j}(\mathbf{w})
    -\sum_{p=1}^k n_{p,j}\xi_p
    +i\left( 
    \mu_{j}(\alpha) 
    + \sum_{p=1}^{k}n_{p,j}\lambda_{p}(\alpha)
    \right).
    $$
\end{enumerate}
\end{lemma}
\begin{proof}
    Suppose 
    $$ \mathbf{A}\BLog\mathbf{w} + \mathbf{B} \BLog\mathbf{w''} = -i\boldsymbol{\nu}-\boldsymbol{u}.$$
    By taking conjugation on both sides, we have
    $$ \mathbf{A}\BLog\mathbf{\overline{w}} + \mathbf{B} \BLog\mathbf{\overline{w}''} = i\boldsymbol{\nu}-\overline{\boldsymbol{u}}.$$
    By Lemma \ref{n1n2defn}, we have
    $$
    2(\mathbf{l}_{j}^{-1})^{\!\top}(-\BLog \mathbf{\overline{w}} + i\boldsymbol{\pi}) 
    = \mathrm{H}^\CC_{X,m_j}(\mathbf{\overline{w}})
    -\sum_{p=1}^k n_{p,j}\overline{\xi_p}
    + in_2^j\pi.
    $$
    By taking complex conjugate on both sides, we have
    $$
    2(\mathbf{l}_{j}^{-1})^{\!\top}(-\BLog \mathbf{w} - i\boldsymbol{\pi}) 
    = \mathrm{H}^\CC_{X,m_j}(\mathbf{w})
    -\sum_{p=1}^k n_{p,j}\xi_p
    - in_2^j\pi.
    $$
    By Lemma \ref{n1n2defn}(2), we have
    $$
    2(\mathbf{l}_{j}^{-1})^{\!\top}(- \boldsymbol{a} + \boldsymbol{\pi}) 
    = \mu_{j}(\alpha) 
    + \sum_{p=1}^{k}n_{p,j}\lambda_{p}(\alpha) 
    + n_2^j\pi.
    $$
    As a result, we have
    $$
    2(\mathbf{l}_{j}^{-1})^{\!\top}(-\BLog \mathbf{w} - i\boldsymbol{a}) 
    = \mathrm{H}^\CC_{X,m_j}(\mathbf{w})
    -\sum_{p=1}^k n_{p,j}\xi_p
    +i\left( 
    \mu_{j}(\alpha) 
    + \sum_{p=1}^{k}n_{p,j}\lambda_{p}(\alpha)
    \right)
    .
    $$
    This completes the proof.
\end{proof}

\subsection{Computation of gauge fixing and the kernel}
Next, we investigate the Dirac delta distribution corresponding to the gauge fixing. Lemma \ref{gfdelta} provides a geometric interpretation of the linear subspace defined by the constraints. It implies that the values assigned to the redundant edges $e_{N-k+1},\dots,e_N$ are zero.
\begin{lemma}\label{gfdelta}
    Let $\{e_1,\dots,e_{N-k}\}$ be the set of independent edges. Then  the linear subspace defined by 
    $$ \langle \boldsymbol{\lambda}, \mathbf{s}\rangle = \mathbf{0}$$
    is the subspace $\{ s\in \RR^{\Delta_1(X)} \mid s(e_{N-k+1})=\dots=s(e_N)= 0\}$.
\end{lemma}
\begin{proof}
    Recall from the definition that any gauge fixing $\lambda$ at a vertex $c_j$ satisfies 
    $$ \langle \lambda, bg\rangle = g(c_j).$$
    Notices that as a matrix, the linear gauge map $b$ is the transpose of the $k$ by $N$ incidence matrix whose $(i,j)$-th entry is the number of endpoints the $j$-th edge that is incident at the $i$-th vertex. From \cite{NZ}, it is known that the sequence 
    $$ 
    \CC^{2N} \xrightarrow[]{(\mathbf{A_e} \text{ } \mathbf{B_e})} \CC^N 
    \xrightarrow[]{b^{\!\top}} \CC^k \to 0
    $$
    is exact. We first consider the submatrix $Q$ consisting of the first $N-k$ rows of $(\mathbf{A_e} \quad \mathbf{B_e})$. By assumption, since $\{e_1,\dots,e_{N-k}\}$ is an independent set of edges, rank of $Q$ equals $N-k$. In particular, $Q$ consists of $N-k$ linearly independent columns. Consider the corresponding columns $C_1,\dots, C_{N-k}$ of the full matrix $(\mathbf{A_e}\quad \mathbf{B_e})$. Note that $\{C_1,\dots, C_{N-k}\}$ are linearly independent, since the corresponding vectors in $Q$ are. Since rank$(\mathbf{A_e}\quad \mathbf{B_e})=N-k$, $\{C_1,\dots, C_{N-k}\}$ is a basis of the image of $(\mathbf{A_e} \quad \mathbf{B_e})$. We claim that for every non-zero vector in the image of $(\mathbf{A_e} \quad \mathbf{B_e})$, at least one of the first $N-k$ entries is non-zero. To see this, from the above discussion, every vector $v \in \im(\mathbf{A_e}\quad \mathbf{B_e})$ is of the form 
    $$ v= a_1 C_1 + a_2C_2 + \dots + a_{N-k} C_{N-k} $$
    for some coefficients $a_1,\dots,a_{N-k}\in \CC$.
    Consider the first $N-k$ rows of such a vector. If all entries are zero, then by the linearly independence of the first $N-k$ entries of $C_1,\dots, C_{N-k}$, we have $a_1=\dots=a_{N-k} = 0$. This implies $v=0$.  

    As a result, every non-zero vector in the $k$ dimensional subspace
    $$ W = \{(0,\dots,0, c_{N-k+1},\dots, c_{N})\mid c_{N-k+1},\dots, c_N \in \RR\}$$
    lies outside $\im (\mathbf{A_e}\quad \mathbf{B_e}) = \ker b^{\!\top}$. In particular, $b^{\!\top}$ is injective on $W$. Since rank $b^{\!\top} = k$, $b^{\!\top}$ is an isomorphism from $W$ to $\im b^{\!\top} \cong \CC^k$. In particular, there exist linearly independent vectors $w_{1},\dots,w_{k} \in W$ such that $b^{\!\top} w_j = e_j$ for $j=1,\dots, k$. As a result, for each $j=1,\dots, k$, we can take $\lambda_j = w_{j}$ to be a gauge fixing such that
    $$
    \langle \lambda , bg \rangle
    = (b^{\!\top} w_{j}) \cdot (g(c_1),\dots, g(c_k))
    = e_{j} \cdot (g(c_1),\dots, g(c_k))
    = g(c_{j}).
    $$
    Now, consider the system of equations 
   $$w_{1} \cdot (s(e_1),\dots, s(e_N)) = \dots = w_{k} \cdot (s(e_1),\dots, s(e_N))  = 0$$
    given by these choices of gauge fixings. Since $w_1,\dots, w_k$ are linearly independent, they form a basis of $W$. Thus, the system of equations can also be written as
    $$ e_{N-k+1} \cdot (s(e_1),\dots, s(e_N)) = \dots
    = e_{N} \cdot (s(e_1),\dots, s(e_N))
    = 0,$$
    which implies $s(e_j) = 0$ for $j=N-k+1,\dots, N$. 
\end{proof}

Using Lemma \ref{gfdelta}, we can now express the kernel in Proposition \ref{KLVexpress1} in terms of Neumann--Zagier matrices $\mathbf{A}$ and $\mathbf{B}$. 

\begin{lemma}\label{Kfor}
    When $\det\mathbf{B} \neq 0$, up to a constant that is independent of $\alpha$ and $\hbar$, we have
    $$
    \mathscr{K}(X,\mathbf{t},\mathbf{\hat{t}}) 
    = \frac{1}{\det(b^{\!\top}
    )}e^{-\pi i \mathbf{t}^{\!\top} \mathbf{B}^{-1}\mathbf{A}\mathbf{t}} e^{
         \pi i \mathbf{\hat{t}}^{\!\top} \mathbf{B}^{-1}\mathbf{A}\mathbf{\hat{t}}}
    \delta((\mathbf{l}_\mathbf{B}^{-1})^{\!\top}(\mathbf{t}-\mathbf{\hat{t}})^{\!\top}).
    $$
\end{lemma}
\begin{proof}
    First, from Lemma \ref{gfdelta}, we have $s(e_{N-k+1})=\dots=s(e_N)=0.$
    Next, we consider the system of linear equations
    \begin{align}\label{Beqn1}
        -\mathbf{B_e^{\!\top}} \mathbf{s} - \mathbf{t} + \mathbf{\hat{t}} =0,
    \end{align} where $\mathbf{s} = (s(e_1),\dots,s(e_N))^{\!\top}$ corresponds to the edge variables, $\mathbf{t}=(t_1,\dots,t_N)^{\!\top}$ and $\mathbf{\hat{t}}=(\hat{t}_1,\dots,\hat{t}_N)^{\!\top}$. 
    Note that $\mathbf{B_e^{\!\top}}$ can be written in the form 
    $$ \mathbf{B_e^{\!\top}} = \begin{pmatrix}
        \mathbf{B_{N-k}^{\!\top}} & *
    \end{pmatrix}.$$
    Since $s(e_{N-k+1}) = \dots = s(e_N) = 0$ on the subspace with $\langle \boldsymbol{\lambda}, \mathbf{s}\rangle = \mathbf{0}$, Equation \eqref{Beqn1}
    implies that 
    \begin{align}\label{BsN-1}
    -\mathbf{B}_{N-k}^{\!\top} \mathbf{s}_{N-k} = \mathbf{t} - \mathbf{\hat{t}},
    \end{align}
    where $\mathbf{s}_{N-k}=(s(e_1),\dots,s(e_{N-k}))^{\!\top}$. From Equations \eqref{lBTT} and \eqref{BsN-1}, we have
    \begin{align}
    0 = - (\mathbf{l}_{\mathbf{B}}^{-1})^{\!\top}\mathbf{B}_{N-k}^{\!\top} \mathbf{s}_{N-k} = (\mathbf{l}_{\mathbf{B}}^{-1})^{\!\top}\big(\mathbf{t}-\mathbf{\hat{t}}\big). \label{extradelta}
    \end{align}
    Besides, by Equation \eqref{BB^-1}, since $\mathbf{B}_{N-k}\mathbf{B}_{N-k}^{-1}= \mathrm{Id}_{N-k}$, we have $(\mathbf{B}_{N-k}^{-1})^{\!\top}\mathbf{B}_{N-k}^{\!\top} = \mathrm{Id}_{N-k}.$ Combining with Equation \eqref{BsN-1}, we have
    \begin{align}
        \mathbf{s}_{N-k} = -(\mathbf{B}_{N-k}^{-1})^{\!\top}(\mathbf{t}-\mathbf{\hat{t}}).
    \end{align}
    Since $s(e_{N-k+1}) = \dots = s(e_N) = 0$ on the subspace with $\langle \boldsymbol{\lambda}, \mathbf{s}\rangle = \mathbf{0}$, we have 
    $$
    e^{\pi i(\mathbf{t}+\mathbf{\hat{t}})^{\!\top}\mathbf{A_e^{\!\top}}\mathbf{s}}
    = e^{\pi i (\mathbf{t}+\mathbf{\hat{t}})^{\!\top}\mathbf{\mathbf{A}_{N-k}^{\!\top}}\mathbf{s}_{N-k}}
    = e^{-\pi i (\mathbf{t}+\mathbf{\hat{t}})^{\!\top}\mathbf{A}_{N-k}^{\!\top} (\mathbf{B}_{N-k}^{-1})^{\!\top}(\mathbf{t}-\mathbf{\hat{t}})}.
    $$
    Note that 
    $$
    \mathbf{A}^{\!\top} (\mathbf{B}^{-1})^{\!\top}
    = \begin{pmatrix}
        \mathbf{A}_{N-k}^{\!\top} & \mathbf{l}_\mathbf{A}^{\!\top}
    \end{pmatrix}
    \begin{pmatrix}
        (\mathbf{B}_{N-k}^{-1})^{\!\top} \\ (\mathbf{l}_{\mathbf{B}}^{-1})^{\!\top}
    \end{pmatrix}
    = \mathbf{A}_{N-k}^{\!\top}(\mathbf{B}_{N-k}^{-1})^{\!\top}
    + \mathbf{l}_\mathbf{A}^{\!\top}(\mathbf{l}_{\mathbf{B}}^{-1})^{\!\top}.
    $$
    As a result, under the condition that $(\mathbf{l}_{\mathbf{B}}^{-1})^{\!\top}(\mathbf{s}-\mathbf{t})=0$ imposed by Equation \eqref{extradelta}, we have
    \begin{align*}
        (\mathbf{t}+\mathbf{\hat{t}})^{\!\top}\mathbf{A}_{N-1}^{\!\top} (\mathbf{B}_{N-1}^{-1})^{\!\top}(\mathbf{t}-\mathbf{\hat{t}})
        =&\ (\mathbf{t}+\mathbf{\hat{t}})^{\!\top}\mathbf{A}^{\!\top} (\mathbf{B}^{-1})^{\!\top}(\mathbf{t}-\mathbf{\hat{t}})
        - (\mathbf{l}_\mathbf{A}(\mathbf{t}+\mathbf{\hat{t}}))^{\!\top} (\mathbf{l}_{\mathbf{B}}^{-1})^{\!\top}(\mathbf{t}-\mathbf{\hat{t}}) \\
        =&\ (\mathbf{t}+\mathbf{\hat{t}})^{\!\top}\mathbf{A}^{\!\top} (\mathbf{B}^{-1})^{\!\top}(\mathbf{t}-\mathbf{\hat{t}}).
    \end{align*}
    Since $\mathbf{B}^{-1}\mathbf{A}$ is symmetric, we have
    \begin{align*}
        &\ (\mathbf{t}+\mathbf{\hat{t}})^{\!\top}\mathbf{A}^{\!\top} (\mathbf{B}^{-1})^{\!\top}(\mathbf{t}-\mathbf{\hat{t}})\\
        =&\ (\mathbf{t}-\mathbf{\hat{t}})^{\!\top}\mathbf{B}^{-1}\mathbf{A} (\mathbf{t}+\mathbf{\hat{t}})\\
        =&\ \mathbf{t}^{\!\top} \mathbf{B}^{-1}\mathbf{A}\mathbf{t}
        - \mathbf{\hat{t}}^{\!\top} \mathbf{B}^{-1}\mathbf{A}\mathbf{\hat{t}}
        + \mathbf{t}^{\!\top}\mathbf{B}^{-1}\mathbf{A} \mathbf{\hat{t}}
        - \mathbf{\hat{t}}^{\!\top}\mathbf{B}^{-1}\mathbf{A} \mathbf{t} \\
        =&\ \mathbf{t}^{\!\top} \mathbf{B}^{-1}\mathbf{A}\mathbf{t}
        - \mathbf{\hat{t}}^{\!\top} \mathbf{B}^{-1}\mathbf{A}\mathbf{\hat{t}},
    \end{align*}
    where the last equality follows from the fact that 
    $$
    \mathbf{t}^{\!\top}\mathbf{B}^{-1}\mathbf{A} \mathbf{\hat{t}}
    = (\mathbf{t}^{\!\top}\mathbf{B}^{-1}\mathbf{A} \mathbf{\hat{t}})^{\!\top}
    = \mathbf{\hat{t}}^{\!\top}(\mathbf{B}^{-1}\mathbf{A})^{\!\top} \mathbf{t}
    = \mathbf{\hat{t}}^{\!\top}\mathbf{B}^{-1}\mathbf{A} \mathbf{t}.
    $$
    This implies the desired result.
\end{proof}

\subsection{Dimofte-Garoufalidis 1-loop invariant}\label{HKSintro}
In this section, we give a brief review of Dimofte-Garoufalidis 1-loop conjecture. See \cite{DG} and \cite{PW} for more details on the conjecture.

\begin{definition}\label{SCF}
A strong combinatorial flattening consists of three vectors 
\begin{align*}
\mathbf{f} = (f_1,\dots,f_N),\quad
\mathbf{f'} = (f_1',\dots,f_N'),\quad
\mathbf{f}'' = (f_1'',\dots,f_N'') \in \mathbb{Z}^N
\end{align*}such that 
\begin{itemize}
\item for $j=1,\dots, n$, we have $f_j + f_j' + f_j'' = 1$ and
\item for {\textbf {\textit any}} system of simple closed curves $\boldsymbol{l} = (l_1,\dots,l_k)$, the $j$-th entry of the vector 
$$\mathbf{G}\cdot \mathbf{f}^{\!\top} + \mathbf{G'} \cdot {\mathbf{f}'}^{\!\top} + \mathbf{G''} \cdot {\mathbf{f}''}^{\!\top} $$  is equal to $2$ for $j=1,\dots, N-k$ and is equal to $0$ for $j=N-k+1,\dots, N$.
\end{itemize}
\end{definition}
\begin{remark}
By \cite[Lemma 6.1]{N}, a strong combinatorial flattening exists for any ideal triangulation.
\end{remark}

\begin{definition}\label{defn1loop}\cite{DG} 
Let $X$ be an ideal triangulation of $M$ and $\mathbf{z}$ a shape structure on $X$.
The 1-loop invariant of $(M,\boldsymbol{l},X, \mathbf{z})$ is defined by
$$\tau(M, \boldsymbol{l}, X, \mathbf{z}) = 
\pm \frac{1}{2} \mathrm{det}\Big( \mathbf{A} \Delta_{\mathbf{z}''} + \mathbf{B} \Delta_{\mathbf{z}}^{-1}\Big) \prod_{j=1}^N \Big(z_j^{f_j''} z_j''^{-f_j}\Big)
,$$
where $\mathbf{A}=\mathbf{G} - \mathbf{G'}, \mathbf{B} = \mathbf{G''} - \mathbf{G}'$, $(\mathbf{f},\mathbf{f}',\mathbf{f}'')$ is any strong combinatorial flattening, 
$$\Delta_{\mathbf{z}} 
= \begin{pmatrix}
z_1 & 0 & 0 & \dots & 0 \\
0 & z_2 & 0 & \dots & 0 \\
\vdots & \vdots & \vdots & \vdots & \vdots \\
0 & 0 & 0 & 0 & z_N
\end{pmatrix}
\quad \text{and} \quad
\Delta_{\mathbf{z}''} 
= \begin{pmatrix}
z_1'' & 0 & 0 & \dots & 0 \\
0 & z_2'' & 0 & \dots & 0 \\
\vdots & \vdots & \vdots & \vdots & \vdots \\
0 & 0 & 0 & 0 & z_N''
\end{pmatrix}.$$
\end{definition}

Suppose a representation $\rho:\pi_1(M)\to\mathrm{PSL}(2;\CC)$, such as the holonomy representation coming from the complete hyperbolic structure of $M$, is recovered by the shape parameters $\mathbf{z}$. The Dimofte--Garoufalidis 1-loop conjecture, which has recently been resolved by Ming--Wu \cite[Theorem 1.1]{MW}, suggests that the 1-loop invariant coincides with the adjoint twisted Reidemeister torsion $\mathrm{Tor}(M,\rho,\boldsymbol{l})$ defined by Porti \cite{P} associated with $\rho$ and $\boldsymbol{l}$. 

\begin{theorem}\label{MWtor}\cite{MW}
Let $M$ be a hyperbolic 3-manifold with toroidal boundary and let $\boldsymbol{\gamma}=\{\gamma_1,\dots,\gamma_k\mid \gamma_j \subset \mathbb{T}_j\}$ be a system of non-trivial simple closed curves. Let $\rho_0:\pi_1(M)\to \mathrm{PSL}(2;\CC)$ be the discrete faithful representation associated with the complete hyperbolic structure. Assume that there exist shape parameters that realize the representation $\rho_0$. Then for any shape parameters $\mathbf{z}$ such that the induced representation $\rho$ is $\boldsymbol{\gamma}$-regular, we have
$$
\tau(M, \boldsymbol{\gamma}, X, \mathbf{z})
= \pm
\mathrm{Tor}(M,\rho,\boldsymbol{\gamma}).
$$
\end{theorem}

The next lemma shows that the 1-loop term in Theorem \ref{thm3} is independent of the choices of the pairs $(\boldsymbol{m},\boldsymbol{l})$.
\begin{lemma}\label{HKSindep}
    Let $(\boldsymbol{m},\boldsymbol{l})$ and $(\boldsymbol{m}',\boldsymbol{l}')$ be two systems of simple closed curves with $$m_i \cdot l_i = m_i' \cdot l_i' = 1.$$
    Then we have
    $$
\det\left( \im\left(\frac{\partial \omega^{loc}_{m_p}}{\partial \omega^{loc}_{l_q}}\right)\right)\left| \mathrm{Tor}(M,\rho,\boldsymbol{l})\right|^2
=
\det\left( \im\left(\frac{\partial \omega^{loc}_{m'_p}}{\partial \omega^{loc}_{l'_q}}\right)\right)\left| \mathrm{Tor}(M,\rho,\boldsymbol{l}')\right|^2
    $$
\end{lemma}
\begin{proof}
    Since $(\boldsymbol{m},\boldsymbol{l})$ and $(\boldsymbol{m}',\boldsymbol{l}')$ are related by some matrix in $SL(2;\ZZ)$, it suffices to check that the expression is invariant under the transformations $(\boldsymbol{m},\boldsymbol{l}) \mapsto (\boldsymbol{m+l},\boldsymbol{l})$ and $(\boldsymbol{m},\boldsymbol{l}) \mapsto (-\boldsymbol{l},\boldsymbol{m})$. For the first transformation, it follows from the observation that
    $$
    \im\left(\frac{\partial \omega^{loc}_{m_p+l_p}}{\partial \omega^{loc}_{l_q}}\right)
    = \im\left(\frac{\partial \omega^{loc}_{m_p}}{\partial \omega^{loc}_{l_q}} + \mathrm{Id}_k\right)
    = \im\left(\frac{\partial \omega^{loc}_{m_p}}{\partial \omega^{loc}_{l_q}}\right),
    $$
    where $\mathrm{Id}_k$ denotes the $k \times k$ identity matrix. For the second transformation, note that by the change of curves formula of the torsion \cite[Theorem 4.1]{P}, we have 
    \begin{align*}
        \det\left( \im\left(\frac{\partial \omega^{loc}_{-l_p}}{\partial \omega^{loc}_{m_q}}\right)\right)\left| \mathrm{Tor}(M,\rho,\boldsymbol{m})\right|^2 
        =&\ \det\left( -\im\left(\frac{\partial \omega^{loc}_{l_p}}{\partial \omega^{loc}_{m_q}}\right)\right)\left| \mathrm{Tor}(M,\rho,\boldsymbol{l})\right|^2 \left|\det\left(\frac{\partial \omega^{loc}_{m_p}}{\partial \omega^{loc}_{l_q}}\right) \right|^2 .
    \end{align*}
    Moreover, we have
    \begin{align*}
       &\  \det\left( -\im\left(\frac{\partial \omega^{loc}_{l_p}}{\partial \omega^{loc}_{m_q}}\right)\right)
         \left|\det\left(\frac{\partial \omega^{loc}_{m_p}}{\partial \omega^{loc}_{l_q}}\right) \right|^2 \\
         =&\ \det\left(\frac{\partial \omega^{loc}_{m_p}}{\partial \omega^{loc}_{l_q}}\right)
         \det\left( -\frac{1}{2i}\left(\left(\frac{\partial \omega^{loc}_{l_p}}{\partial \omega^{loc}_{m_q}}\right) - \left( \overline{\frac{\partial \omega^{loc}_{l_p}}{\partial \omega^{loc}_{m_q}}}\right) \right)\right)
         \det\left(\overline{\frac{\partial \omega^{loc}_{m_p}}{\partial \omega^{loc}_{l_q}}}\right) \\
         =&\ \det\left( -\frac{1}{2i}\left( \left(\overline{\frac{\partial \omega^{loc}_{m_p}}{\partial \omega^{loc}_{l_q}}}\right) - \left( \frac{\partial \omega^{loc}_{m_p}}{\partial \omega^{loc}_{l_q}}\right) \right)\right) \\
         =&\ \det\left( \im\left(\frac{\partial \omega^{loc}_{m_p}}{\partial \omega^{loc}_{l_q}}\right)\right)
    \end{align*}
    Thus, we have
    $$
    \det\left( \im\left(\frac{\partial \omega^{loc}_{-l_p}}{\partial \omega^{loc}_{m_q}}\right)\right)\left| \mathrm{Tor}(M,\rho,\boldsymbol{m})\right|^2 
    = \det\left( \im\left(\frac{\partial \omega^{loc}_{m_p}}{\partial \omega^{loc}_{l_q}}\right)\right)\left| \mathrm{Tor}(M,\rho,\boldsymbol{l})\right|^2 .
    $$
    and the expression is invariant under the second transformation. This completes the proof.
\end{proof}

\subsection{Saddle point approximation}
The following version of the saddle point approximation is a special case of \cite[Proposition 5.1]{WY2}. We adapt the notations to our setting. 
\begin{proposition}[Proposition 5.1, \cite{WY2}]\label{saddle}
Let $D$ be a region in $\mathbb C^{N}$. Let $f(\mathbf z)$ and $g(\mathbf z)$ be complex valued functions on $D$  which are holomorphic in $\mathbf z$. For each $\hbar>0,$ let $f_\hbar(\mathbf z)$ be a complex valued function on $D$ holomorphic in $\mathbf z$ of the form
$$ f_\hbar(\mathbf z) = f(\mathbf z) + \upsilon_\hbar(\mathbf z)\hbar^2,$$
where $\upsilon_\hbar(\mathbf z)$ is another complex valued function on $D$ holomorphic in $\mathbf z$.
Let $S$ be an embedded $N$-dimensional disk and $\mathbf c$ be a critical point of $f$ on $S$. If for each $\hbar>0$
\begin{enumerate}
\item $\mathrm{Re}f(\mathbf c) > \mathrm{Re}f(\mathbf z)$ for all $\mathbf z \in S\setminus \{\mathbf c\},$
\item the domain $\{\mathbf z\in D\ |\ \mathrm{Re} f(\mathbf z) < \mathrm{Re} f(\mathbf c)\}$ deformation retracts to $S\setminus\{\mathbf c\},$
\item $g(\mathbf c)\neq 0$,
\item $|\upsilon_\hbar(\mathbf z)|$ is bounded from above by a constant independent of $\hbar$ on $D,$ and
\item  the Hessian matrix $\mathrm{Hess}(f)$ of $f$ at $\mathbf c$ is non-singular,
\end{enumerate}
then as $\hbar\to 0$,
\begin{equation*}
\begin{split}
 \int_{S} g(\mathbf z) e^{\frac{1}{\hbar}f_\hbar(\mathbf z)} d\mathbf z= \Big(2\pi\hbar\Big)^{\frac{{N}}{2}}\frac{g(\mathbf c)}{\sqrt{(-1)^{{N}}\det\mathrm{Hess}(f)(\mathbf c)}} e^{\frac{1}{\hbar}f(\mathbf c)} \Big( 1 + O (\hbar) \Big).
 \end{split}
 \end{equation*}
\end{proposition}

\section{KLV partition function and the Reshetikhin--Turaev type functions}\label{KLVRT}

Given an angle structure $\alpha\in \mathcal{A}_X$, we let $a_j := \alpha(q_j), a_j' = \alpha(q_j')$ and $a_j''=\alpha(q_j'')$ be the dihedral angles assigned to the edges of $T_j$ separated by the quad $q_j, q_j'$ and $q_j''$ respectively. We write $\boldsymbol{(\pi-a)} = (\pi-a_1,\dots, \pi -a_N)^{\!\top}$. Let $\boldsymbol{\lambda}(\alpha)=(\lambda_1(\alpha),\dots,\lambda_k(\alpha))$, where $\lambda_j(\alpha)$ is the angular holonomy of $l_j$ with respect to $\alpha$ for $j=1,\dots,k$. Let $\boldsymbol{\kappa}(\alpha) = (\kappa_1(\alpha),\dots,\kappa_k(\alpha))^{\!\top}$ with
$$
\kappa_j(\alpha) = 
\mu_{j}(\alpha) 
    + \sum_{p=1}^{k}n_{p,j}\lambda_{p}(\alpha)
    $$
    for $j=1,\dots,k$, where $n_{p,j}$'s are the rational numbers in Lemma \ref{n1n2defn}. Let $\boldsymbol{n_2}=(n_2^1,\dots,n_2^k)^{\!\top}$ be the rational numbers in Lemma \ref{n1n2defn} and let 
$
\mathcal{D} = \{\boldsymbol{x}\in \RR^k+i\boldsymbol{\kappa}(\alpha) \mid \alpha \in \mathcal{A}_X\}
$. 
The following definition of $\mathrm{RT}_\hbar^{(X,\boldsymbol{l},\square)}$ in Definition \ref{defnRT} is motivated by Theorem \ref{KLVRTsquare} below. We will study the well-definedness and properties of $\mathrm{RT}_\hbar^{(X,\boldsymbol{l},\square)}$ in Proposition \ref{RTprop}. 
\begin{definition}\label{defnRT}
    Associated with any choice of quad type $\square$ with an invertible $\mathbf{B}$, we define the Reshetikhin--Turaev type function $\mathrm{RT}_\hbar^{(X,\boldsymbol{l},\square)} : \mathcal{D} \to \CC$ by
\begin{align*}
         \mathrm{RT}_\hbar^{(X,\boldsymbol{l},\square)}(\boldsymbol{x}) =&\  \frac{1}{(2\pi\sqrt{\hbar})^{N}}
        \int_{\mathbf{t}\in \RR^{N}+i\boldsymbol{(\pi-a)}} 
        \delta\left(2(\mathbf{l}_\mathbf{B}^{-1})^{\!\top}\mathbf{t}- i\pi\boldsymbol{n_2} -\boldsymbol{x}\right)
        \\
        &\ \times  
         \left[e^{-\frac{i}{4\pi\hbar} \mathbf{t}^{\!\top}  \mathbf{B}^{-1}\mathbf{A} \mathbf{t}  
    + \frac{1}{2\pi\hbar}  \mathbf{t}^{\!\top}(\mathbf{B}^{-1}\boldsymbol{\nu} - \mathbf{B}^{-1}\mathbf{A}\boldsymbol{\pi} )}
        \prod_{j=1}^N\Phi_\B\left( \frac{t_j}{2\pi\sqrt{\hbar}} \right) \right]
        d\mathbf{t}.
    \end{align*}
\end{definition}

\begin{proposition}\label{KLVRTsquare}
   Up to a constant that is independent of $\alpha$ and $\hbar$, We have
    \begin{align*}
        W_\B(X,\alpha) = 
        \int_{\boldsymbol{x}\in \RR^k + i\boldsymbol{\kappa}(\alpha)}\left|\mathrm{RT}_\hbar^{(X,\boldsymbol{l},\square)}( \boldsymbol{x}) e^{\frac{1}{4\pi\hbar} \boldsymbol{x}^{\!\top}\cdot \boldsymbol{\lambda}(\alpha)} \right|^2   d\boldsymbol{x},
    \end{align*} 
    where $\boldsymbol{x}=(x_1,\dots,x_k)$.
    In particular, the KLV partition function depends on the prescribed angle structure only through its peripheral angular holonomies.
\end{proposition}
\begin{proof}
    From Proposition \ref{KLVexpress1} and Lemma \ref{Kfor}, we have 
    \begin{align*}
        &\ W_\B(X,\alpha) =\int_{(\mathbf{t},\mathbf{\hat{t}})\in \RR^{2N}} 
        \delta\big((\mathbf{l}_\mathbf{B}^{-1})^{\!\top}\big(\mathbf{t}-\mathbf{\hat{t}}\big)\big)\\
        &\
        \times\left(e^{-\pi i \mathbf{t}^{\!\top} \mathbf{B}^{-1}\mathbf{A}\mathbf{t}} 
        \prod_{j=1}^N \overline{\varphi}_{(a_j'',a_j')}(t_j)\right)
        \left(e^{
        \pi i \mathbf{\hat{t}}^{\!\top} \mathbf{B}^{-1}\mathbf{A}\mathbf{\hat{t}}}
        \prod_{j=1}^N \varphi_{(a_j'',a_j')}\big(\hat{t}_j\big) \right)
        d\mathbf{t} d\mathbf{\hat{t}}.
    \end{align*}
    Since $(\mathbf{l}_\mathbf{B}^{-1})^{\!\top}$ has rank $k$, by using the distributional identity
    $$
\delta\big((\mathbf{l}_\mathbf{B}^{-1})^{\!\top}\big(\mathbf{t}-\mathbf{\hat{t}}\big)\big)
    = \frac{1}{(4\pi\sqrt{\hbar})^k} \int_{\boldsymbol{x}\in \RR^k}
    \delta\left((\mathbf{l}_\mathbf{B}^{-1})^{\!\top}\mathbf{t}-\frac{1}{4\pi\sqrt{\hbar}}\boldsymbol{x}\right)
    \delta\left((\mathbf{l}_\mathbf{B}^{-1})^{\!\top}\mathbf{\hat{t}}- \frac{1}{4\pi\sqrt{\hbar}}\boldsymbol{x}\right)
    d\boldsymbol{x}
    $$
    together with Fubini's theorem, we have
    \begin{align}\label{Wsplit}
         &\ W_\B(X,\alpha) \notag\\
         =&\ \frac{1}{(4\pi\sqrt{\hbar})^k}\int_{\boldsymbol{x}\in \RR^k}
        \left(\int_{\mathbf{t}\in \RR^{N}} 
        \delta\left((\mathbf{l}_\mathbf{B}^{-1})^{\!\top}\mathbf{t}-\frac{1}{4\pi\sqrt{\hbar}}\boldsymbol{x}\right)\left(e^{-\pi i \mathbf{t}^{\!\top} \mathbf{B}^{-1}\mathbf{A}\mathbf{t}} 
        \prod_{j=1}^N \overline{\varphi}_{(a_j'',a_j')}(t_j)\right)d\mathbf{t} \right) \notag\\
        & \times \left(\int_{\mathbf{\hat{t}}\in \RR^{N}} \delta\left((\mathbf{l}_\mathbf{B}^{-1})^{\!\top}\mathbf{\hat{t}}-\frac{1}{4\pi\sqrt{\hbar}}\boldsymbol{x}\right)
        \left(e^{
        \pi i \mathbf{\hat{t}}^{\!\top} \mathbf{B}^{-1}\mathbf{A}\mathbf{\hat{t}}}
        \prod_{j=1}^N \varphi_{(a_j'',a_j')}(\hat{t}_j) \right)
        d\mathbf{\hat{t}}\right),
        \end{align}
        which implies
        \begin{align*}
        &\ W_\B(X,\alpha) \\
         =&\  \frac{1}{(4\pi\sqrt{\hbar})^k} \int_{\boldsymbol{x}\in \RR^k}
        \left|\int_{\mathbf{t}\in \RR^{N}} 
        \delta\left((\mathbf{l}_\mathbf{B}^{-1})^{\!\top}\mathbf{t}-\frac{1}{4\pi\sqrt{\hbar}}\boldsymbol{x}\right)
        e^{-\pi i \mathbf{t}^{\!\top} \mathbf{B}^{-1}\mathbf{A}\mathbf{t}} 
        \prod_{j=1}^N \overline{\varphi}_{(a_j'',a_j')}(t_j)d\mathbf{t}\right|^2
        d\boldsymbol{x}.
    \end{align*}
    For $j=1,\dots,N$ and $t_j\in \RR$, by Proposition \ref{prop:quant:dilog}(1), we have
    $$
\overline{\varphi_{(a_j'',a_j')}(t_j)}
    = e^{\frac{1}{\sqrt{\hbar}}a_j'' t_j}
    \Phi_\B\left(t_j + \frac{i}{2\sqrt{\hbar}}(\pi - a_j) \right).
    $$ 
    As a result,
    \begin{align*}
        e^{-\pi i \mathbf{t}^{\!\top} \mathbf{B}^{-1}\mathbf{A}\mathbf{t}}
        \prod_{j=1}^N \overline{\varphi}_{(a_j'',a_j')}(t_j)
        &= e^{-\pi i \mathbf{t}^{\!\top} \mathbf{B}^{-1}\mathbf{A}\mathbf{t} + \frac{1}{\sqrt{\hbar}}\sum_{j=1}^N a_j''t_j}
        \prod_{j=1}^N\Phi_\B\left(t_j+\frac{i}{2\pi\sqrt{\hbar}}(\pi-a_j) \right).
    \end{align*}
    Recall that the matrix $\mathbf{B}^{-1}\mathbf{A}$ is symmetric. By applying the change of variables that replace $t_j$ by $$
    2\pi\sqrt{\hbar}\left(t_j + \frac{i}{2\pi\sqrt{\hbar}}(\pi-a_j)\right),$$ 
    we have
    \begin{align*}
         W_\B(X,\alpha) 
         = &\ \frac{1}{2^k(2\pi\sqrt{\hbar})^{k+N} }\int_{\boldsymbol{x}\in \RR^k}
        \left|\int_{\mathbf{t}\in \RR^{N}+i\boldsymbol{(\pi-a)}} 
        \delta\left(\frac{1}{4\pi\sqrt{\hbar}}\left(2(\mathbf{l}_\mathbf{B}^{-1})^{\!\top}\mathbf{t}- 2i(\mathbf{l}_\mathbf{B}^{-1})^{\!\top}\boldsymbol{(\pi-a)}-\boldsymbol{x}\right)\right) \right.\\
        &\ \times \left. e^{-\frac{i}{4\pi\hbar} \mathbf{t}^{\!\top}  \mathbf{B}^{-1}\mathbf{A} \mathbf{t}  
    - \frac{1}{2\pi\hbar}  (\boldsymbol{\pi-a})^{\!\top}(\mathbf{B}^{-1}\mathbf{A})^{\!\top}\mathbf{t} + \frac{1}{2\pi\hbar}\sum_{j=1}^N a_j'' t_j }
        \prod_{j=1}^N\Phi_\B\left( \frac{t_j}{2\pi\sqrt{\hbar}} \right)\right|^2
        d\mathbf{t}d\boldsymbol{x},
    \end{align*}
   where $(\boldsymbol{\pi-a}) = (\pi-a_1,\dots, \pi - a_N)$ and an extra factor 
   $$
   e^{\frac{i}{4\pi\hbar}\boldsymbol{(\pi - a)}^{\!\top}\mathbf{B^{-1}A}\boldsymbol{(\pi-a)}
   - \frac{i}{2\pi\hbar}\sum_{j=1}^Na_j(\pi-a_j)}
   $$
   of modulus $1$ is removed from the expression.
    Since $(a_1,a_1',a_1'',\dots,a_N,a_N',a_N'')\in \mathcal{A}_X,$ 
    we have 
    $$\mathbf{A} \boldsymbol{a} + \mathbf{B}\boldsymbol{a}'' = \boldsymbol{\nu} + \boldsymbol{u},$$ 
    where $\boldsymbol{a}=(a_1,\dots,a_N)^{\!\top}, \boldsymbol{a''}=(a_1'',\dots,a_N'')^{\!\top}$, $\boldsymbol{\nu}\in \pi\ZZ^N$ and $\boldsymbol{u}=(0,\dots,0,\lambda_{1}(\alpha), \dots, \lambda_{k}(\alpha))$. In particular, we have
    \begin{align*}
        -\mathbf{B}^{-1}\mathbf{A}(\boldsymbol{\pi-a}) + \boldsymbol{a''}
        = \mathbf{B}^{-1}(\mathbf{A} \boldsymbol{a} + \mathbf{B}\boldsymbol{a}'') - \mathbf{B}^{-1}\mathbf{A}\boldsymbol{\pi}
        = \mathbf{B}^{-1}(\boldsymbol{\nu}+\boldsymbol{u}) - \mathbf{B}^{-1}\mathbf{A}\boldsymbol{\pi}.
    \end{align*}
    As a result, we have
    \begin{align*}
        &\ W_\B(X,\alpha) \\
        = &\ \frac{1}{(2\pi\sqrt{\hbar})^{N}} \int_{\boldsymbol{x}\in \RR^k}
        \left|\int_{\mathbf{t}\in \RR^{N}+i\boldsymbol{(\pi-a)}} 
        \delta\left(2(\mathbf{l}_\mathbf{B}^{-1})^{\!\top}\mathbf{t}- 2i(\mathbf{l}_\mathbf{B}^{-1})^{\!\top}\boldsymbol{(\pi-a)}-\boldsymbol{x}\right) \right.\\
        &\ \times \left. \left[e^{-\frac{i}{4\pi\hbar} \mathbf{t}^{\!\top}  \mathbf{B}^{-1}\mathbf{A} \mathbf{t}  
    + \frac{1}{2\pi\hbar}  \mathbf{t}^{\!\top}(\mathbf{B}^{-1}\boldsymbol{\nu} - \mathbf{B}^{-1}\mathbf{A}\boldsymbol{\pi} )}
        \prod_{j=1}^N\Phi_\B\left( \frac{t_j}{2\pi\sqrt{\hbar}} \right) \right] e^{\frac{1}{2\pi\hbar}\big (\sum_{j=1}^k (\mathbf{l}_{j}^{-1})^{\!\top}\mathbf{t}\lambda_{j}(\alpha) \big)}\right|^2
        d\mathbf{t}d\boldsymbol{x}.
    \end{align*}
    Finally, for $j=1,\dots,k$, we apply the change of variables 
    $\boldsymbol{x} = (x_1,\dots, x_k)$ that replaces $x_j$ with
    $$x_j + i\mu_{j}(\alpha) 
    + i\sum_{p=1}^{k}n_{p,j}\lambda_{p}(\alpha) ,  $$
    where $\mu_{j}(\alpha)$ and $\lambda_{p}(\alpha)$ the angular holonomies of $m_j$ and $l_p$ respectively, and $n_{p,j}$'s are the rational numbers in Lemma \ref{n1n2defn}. This gives the desired result. 
\end{proof}

\begin{proposition}\label{RTprop} We have the following results.
\begin{enumerate}
    \item The function $\mathrm{RT}_\hbar^{(X,\boldsymbol{l},\square)}(\boldsymbol{x})$ is an absolute convergent integral.
    \item $|\mathrm{RT}_\hbar^{(X,\boldsymbol{l},\square)}(\boldsymbol{x})|$ is independent of the angle structure $\alpha$.
    \item $\mathrm{RT}_\hbar^{(X,\boldsymbol{l},\square)}(\boldsymbol{x})$ is independent of the choice of independent edges and the choices of representatives of the curves $l_1,\dots,l_k$.
    \item $\mathrm{RT}_\hbar^{(X,\boldsymbol{l},\square)}(\boldsymbol{x})$ agrees with the Jones function constructed in \cite{BAW} for FAMED triangulations.
\end{enumerate}
\end{proposition}
\begin{proof}
For (1), following the argument in the proof of Theorem \ref{KLVRTsquare}, we have
\begin{align*}
         \left|\mathrm{RT}_\hbar^{(X,\boldsymbol{l},\square)}(\boldsymbol{x}) \right|=&\ \left|\frac{1}{(4\pi\sqrt{\hbar})^k}
        \int_{\mathbf{t}\in \RR^{N}} \delta\left((\mathbf{l}_\mathbf{B}^{-1})^{\!\top}\mathbf{t}-\frac{1}{4\pi\sqrt{\hbar}}\boldsymbol{x}\right)\left(e^{-\pi i \mathbf{t}^{\!\top} \mathbf{B}^{-1}\mathbf{A}\mathbf{t}} 
        \prod_{j=1}^N \overline{\varphi}_{(a_j'',a_j')}(t_j)\right)d\mathbf{t}\right|. 
    \end{align*}
    The result then follows from Proposition \ref{prop:quant:dilog}(3) and  Definition \ref{defnanglestr}.

    For (2), note that the integrand is independent of angle structure. From (1), since the integral converges absolutely for all $\alpha\in \mathcal{A}_X$, the result follows by a deformation of integration multi-contour.
    
For (3), to show the independence of the choice of independent edges, we consider the $(N+k)\times N$ matrix
\begin{align*}
    \begin{pmatrix}
        \mathbf{A_e} & \mathbf{B_e}\\
        \mathbf{l}_{\mathbf{A}} & \mathbf{l}_{\mathbf{B}}
    \end{pmatrix}.
\end{align*}
 Let  $\{e_1',\dots, e_{N-k}'\}$ be another set of independent edges and let $\mathbf{A}'_{N-k}, \mathbf{B}'_{N-k}$ be the corresponding submatrices of $\mathbf{A_e}$ and $\mathbf{B_e}$ respectively. Since the matrices
$$
\begin{pmatrix}
    \mathbf{A} & \mathbf{B}
\end{pmatrix}
=
\begin{pmatrix}
    \mathbf{A}_{N-k} & \mathbf{B}_{N-k} \\
    \mathbf{l}_{\mathbf{A}} & \mathbf{l}_{\mathbf{B}}
\end{pmatrix} \quad\text{ and }\quad 
\begin{pmatrix}
    \mathbf{A}' & \mathbf{B}'
\end{pmatrix}
= \begin{pmatrix}
    \mathbf{A}'_{N-k} & \mathbf{B}'_{N-k} \\
    \mathbf{l}_{\mathbf{A}} & \mathbf{l}_{\mathbf{B}}
\end{pmatrix}
$$
have rank $N$, they form two bases of the row spaces of 
\begin{align*}
    \begin{pmatrix}
        \mathbf{A_e} & \mathbf{B_e}\\
        \mathbf{l}_{\mathbf{A}} & \mathbf{l}_{\mathbf{B}}
    \end{pmatrix}.
\end{align*}
As a result, there exists an invertible matrix $\mathbf{E}$ such that
$$
\begin{pmatrix}
    \mathbf{EA} & \mathbf{EB}
\end{pmatrix}
=
\mathbf{E}
\begin{pmatrix}
    \mathbf{A} & \mathbf{B}
\end{pmatrix}
= \begin{pmatrix}
    \mathbf{A}' & \mathbf{B}'
\end{pmatrix}.
$$
In particular, we have $\det\mathbf{B}\neq 0$ if and only if $\det\mathbf{B}' \neq 0$. Moreover, we have $$\mathbf{B}^{-1}\mathbf{A} 
= (\mathbf{EB})^{-1}(\mathbf{EA})
= \mathbf{B'}^{-1}\mathbf{A'}.$$
Besides, if 
$$
    \mathbf{A}
    \BLog \mathbf{z} +  \mathbf{B}
    \BLog \mathbf{z''}
=
    i\pi\boldsymbol{\nu} +
    \boldsymbol{u},
$$
then
$$
\mathbf{A'}
    \BLog \mathbf{z} +  \mathbf{B'}
    \BLog \mathbf{z''}
=
    i\pi \mathbf{E}\boldsymbol{\nu} +
    \boldsymbol{u}.
$$
As a result, we have
$$
\mathbf{B'}^{-1}\boldsymbol{\nu'}
= (\mathbf{B}^{-1} \mathbf{E}^{-1}) (\mathbf{E} \boldsymbol{\nu})
= \mathbf{B}^{-1} \boldsymbol{\nu}.
$$
Since the integrand in the definition of the Jones function remains unchanged, the Jones function is independent of the choice of independent edges. 

Next, given a set of independent edges $\{e_1,\dots, e_{N-k}\}$, different choices of representatives of $l_1,\dots,l_k$ are related by homotopies passing through vertices of the triangulations of boundary tori. Let $\mathbf{l_A'}, \mathbf{l_B'}$ be the Neumann--Zagier data of the new repersentatives of $l_1,\dots,l_k$. Then there exists a $k$ by $N-k$ matrix $\mathbf{K}$ corresponding to those homotopies such that
$$
\begin{pmatrix}
    \mathrm{Id}_{N-k} & \mathbf{0} \\
    \mathbf{K} & \mathrm{Id}_{k}
\end{pmatrix}
\begin{pmatrix}
    \mathbf{A}_{N-k} & \mathbf{B}_{N-k} \\
    \mathbf{l}_\mathbf{A} & \mathbf{l}_\mathbf{B}
\end{pmatrix}
=
\begin{pmatrix}
    \mathbf{A}_{N-k} & \mathbf{B}_{N-k} \\
    \mathbf{l}'_\mathbf{A} & \mathbf{l}'_\mathbf{B}
\end{pmatrix}.
$$
By the same argument as in the previous case, we have
$\mathbf{B}^{-1}\mathbf{A} = \mathbf{B'}^{-1}\mathbf{A'}$ and $\mathbf{B'}^{-1}\boldsymbol{\nu'}
= \mathbf{B}^{-1} \boldsymbol{\nu}$.
Besides, we have
$$
\begin{pmatrix}
    \mathbf{B}_{N-k} \\
    \mathbf{l}'_{\mathbf{B}}
\end{pmatrix}
=
\begin{pmatrix}
    \mathrm{Id}_{N-k} & \mathbf{0} \\
    \mathbf{K} & \mathrm{Id}_{k}
\end{pmatrix}
\begin{pmatrix}
    \mathbf{B}_{N-k} \\
    \mathbf{l}_{\mathbf{B}}
\end{pmatrix}.
$$
By taking inverse and transpose on both sides, we get
$$
\begin{pmatrix}
    (\mathbf{B}^{-1}_{N-k})^{\!\top} \\
    (\mathbf{l}_{\mathbf{B}}^{-1})^{\!\top}
\end{pmatrix}
=
\begin{pmatrix}
    \mathrm{Id}_{N-k} & -\mathbf{K}^{\!\top} \\
    \mathbf{0} & \mathrm{Id}_{k}
\end{pmatrix}
\begin{pmatrix}
    (\mathbf{B}^{-1}_{N-k})^{\!\top} \\
    ((\mathbf{l}')_{\mathbf{B}}^{-1})^{\!\top}
\end{pmatrix},
$$
which implies that $(\mathbf{l}_{\mathbf{B}}^{-1})^{\!\top} = ((\mathbf{l}')_{\mathbf{B}}^{-1})^{\!\top}.$ Again, since the integrand in the definition of the Jones function remains unchanged, the Jones function is independent of the choices of representatives of $l_1,\dots, l_k$.

    For (4), the result follows directly from \cite[Equation 4.5]{BAW} and \cite[Definition 1.1(4)]{BAW}.
\end{proof}

The next result provides another expression of the KLV partition function, which will be utilized to study its asymptotics using the saddle point method. We let $i\boldsymbol{\lambda}(\alpha) = (i\lambda_1(\alpha),\dots, i\lambda_k(\alpha))$.
\begin{proposition}\label{KLVexpress2} Let $\boldsymbol{\xi}\in \CC^k$ and let $\mathbf{z}$ be a solution of Equation \eqref{NZintro} with $\im z_j>0$ for $j=1,\dots,N$. Let $\alpha$ be the angle structure induced by $\mathbf{z}$. 
    Up to renumbering of the tetrahedra and a constant that is independent of $\alpha$ and $\hbar$, the KLV partition function can be written as
    \begin{align*}
       &\ W_\B(X,\alpha) \\
       =&\ 
     \frac{1}{(2\pi\sqrt{\hbar})^{2N+2k}}
    \int_{\boldsymbol{x}\in \RR^k} \int_{\mathbf{t}\in \RR^{N}+i\boldsymbol{(\pi-a)}} 
        \delta\left( 2\left( (\mathbf{l}_\mathbf{B}^{-1})^{\!\top}\mathbf{t}- 2i(\mathbf{l}_\mathbf{B}^{-1})^{\!\top}\boldsymbol{(\pi-a)}-\boldsymbol{x}\right)\right) U_1(\mathbf{t};i\boldsymbol{\lambda}(\alpha))
        d\mathbf{t}\\ 
    &\qquad\qquad\quad\quad\quad\hspace{10pt} \times 
        \int_{\mathbf{\hat{t}}\in \RR^{N}+i\boldsymbol{(\pi-a)}} 
        \delta\left(2\left((\mathbf{l}_\mathbf{B}^{-1})^{\!\top}\mathbf{\hat{t}}- 2i(\mathbf{l}_\mathbf{B}^{-1})^{\!\top}\boldsymbol{(\pi-a)} + \boldsymbol{x}\right)\right) U_2(\mathbf{t};i\boldsymbol{\lambda}(\alpha))
        d\mathbf{\hat{t}},
    \end{align*}
    where 
    $$U_1(\mathbf{t};i\boldsymbol{\lambda}(\alpha)) = H_1(\mathbf{t})e^{\frac{1}{2\pi\hbar}\big (\sum_{j=1}^k (\mathbf{l}_{j}^{-1})^{\!\top}\mathbf{t}\lambda_{j}(\alpha) \big)} ,\quad U_2(\mathbf{\hat{t}};i\boldsymbol{\lambda}(\alpha))=H_2(\mathbf{\hat{t}})e^{-\frac{1}{2\pi\hbar}\left( \sum_{j=1}^k(\mathbf{l}_{j}^{-1})^{\!\top}\mathbf{\hat{t}} \lambda_{j}(\alpha) \right) }$$ with
    \begin{align*}
        H_1(\mathbf{t})
        =&\  
        e^{-\frac{i}{4\pi\hbar} \mathbf{t}^{\!\top}  \mathbf{B}^{-1}\mathbf{A} \mathbf{t}  
    + \frac{1}{2\pi\hbar}  \mathbf{t}^{\!\top}(\mathbf{B}^{-1}\boldsymbol{\nu} - \mathbf{B}^{-1}\mathbf{A}\boldsymbol{\pi} )}
\prod_{j=1}^N\Phi_\B\left( \frac{t_j}{2\pi\sqrt{\hbar}} \right), \\
        H_2(\mathbf{\hat{t}})
        =&\ e^{\frac{i}{4\pi\hbar} \mathbf{\hat{t}}^{\!\top}  \mathbf{B}^{-1}\mathbf{A} \mathbf{\hat{t}}  
    - \frac{1}{2\pi\hbar}  \mathbf{ \hat{t}}^{\!\top}(\mathbf{B}^{-1}\boldsymbol{\nu}-\mathbf{B}^{-1}\mathbf{A}\boldsymbol{\pi})}
        \frac{1}{\prod_{j=1}^N\Phi_\B\left( -\frac{\hat{t}_j}{2\pi\sqrt{\hbar}} \right)}.
    \end{align*}
    As a result, we have
    \begin{align*}
       &\ W_\B(X,\alpha) \\
    =&\  \frac{1}{(2\pi\sqrt{\hbar})^{2N+2k}}\ \int_{\boldsymbol{x}\in \RR^k}
      \left(\int_{t_{k+1}\in \RR+i(\pi-a_{k+1})}\dots \int_{t_{N}\in \RR+i(\pi-a_{N})}
    K_1\left(\boldsymbol{x},t_{k+1},\dots,t_N;\boldsymbol{\xi}\right) d\mathbf{t}\right)
    \\
    &\ \qquad\qquad\quad\quad\quad\hspace{10pt}\times \left(\int_{\hat{t}_{k+1}\in \RR+i(\pi-a_{k+1})}\dots\int_{\hat{t}_{N}\in \RR+i(\pi-a_{N})} K_2\left(\boldsymbol{x},\hat{t}_{k+1},\dots,\hat{t}_N; \boldsymbol{-\overline{\xi}}\right)
            d\mathbf{\hat{t}} \right)
            d\boldsymbol{x},
    \end{align*}
    where $K_1,K_2$ are given by
    \begin{align*}
        K_1\left(\boldsymbol{x},t_{k+1},\dots,t_N;\boldsymbol{\xi}\right) 
        =&\  H_1\left(L_1(\boldsymbol{x},t_{k+1},\dots,t_N)\right) e^{\frac{-i}{4\pi\hbar}\boldsymbol{x} \cdot \boldsymbol{\xi} }, \\
        K_2\left(\boldsymbol{x},\hat{t}_{k+1},\dots,\hat{t}_N;-\boldsymbol{\overline{\xi}}\right) 
        =&\ H_2\left(L_2(\boldsymbol{x},\hat{t}_{k+1},\dots,\hat{t}_N)\right) e^{\frac{i}{4\pi\hbar}\boldsymbol{x} \cdot \boldsymbol{\overline{\xi}} }
    \end{align*}
    with two affine isomorphisms $L_1,L_2:\CC^N \to \CC^N$ satisfying $\det(L_1)=(-1)^k\det(L_2)$.
\end{proposition}
        \begin{proof}
       From Equation \eqref{Wsplit}, we have
       \begin{align*}
         &\ W_\B(X,\alpha) \notag\\
         =&\ \frac{1}{(2\pi\sqrt{\hbar})^k}\int_{\boldsymbol{x}\in \RR^k}
        \left(\int_{\mathbf{t}\in \RR^{N}} 
        \delta\left((\mathbf{l}_\mathbf{B}^{-1})^{\!\top}\mathbf{t}-\frac{1}{4\pi\sqrt{\hbar}}\boldsymbol{x}\right)\left(e^{-\pi i \mathbf{t}^{\!\top} \mathbf{B}^{-1}\mathbf{A}\mathbf{t}} 
        \prod_{j=1}^N \overline{\varphi}_{(a_j'',a_j')}(t_j)\right)d\mathbf{t} \right) \notag\\
        & \times \left(\int_{\mathbf{\hat{t}}\in \RR^{N}} \delta\left((\mathbf{l}_\mathbf{B}^{-1})^{\!\top}\mathbf{\hat{t}}-\frac{1}{4\pi\sqrt{\hbar}}\boldsymbol{x}\right)
        \left(e^{
        \pi i \mathbf{\hat{t}}^{\!\top} \mathbf{B}^{-1}\mathbf{A}\mathbf{\hat{t}}}
        \prod_{j=1}^N \varphi_{(a_j'',a_j')}(\hat{t}_j) \right)
        d\mathbf{\hat{t}}\right).
        \end{align*}
        Following the argument in the proof of Theorem \ref{KLVRTsquare}, for the first integral in $\mathbf{t}$, by applying the change of variables that replace $t_j$ with
        $$\frac{1}{2\pi\sqrt{\hbar}}\left(t_j - i(\pi-a_j)\right),$$ 
        we have
        \begin{align*}
           &\ \int_{\mathbf{t}\in \RR^{N}} 
        \delta\left((\mathbf{l}_\mathbf{B}^{-1})^{\!\top}\mathbf{t}-\frac{1}{4\pi\sqrt{\hbar}}\boldsymbol{x}\right)\left(e^{-\pi i \mathbf{t}^{\!\top} \mathbf{B}^{-1}\mathbf{A}\mathbf{t}} 
        \prod_{j=1}^N \overline{\varphi}_{(a_j'',a_j')}(t_j)\right)d\mathbf{t} \\
        =&\ \frac{e^{\frac{i}{4\pi\hbar}\boldsymbol{(\pi - a)}^{\!\top}\mathbf{B^{-1}A}\boldsymbol{(\pi-a)}
   - \frac{i}{2\pi\hbar}\sum_{j=1}^Na_j(\pi-a_j)}}{(4\pi\sqrt{\hbar})^{k}} \\
    &\ \times 
        \int_{\mathbf{t}\in \RR^{N}+i\boldsymbol{(\pi-a)}} 
        \delta\left(\left(2(\mathbf{l}_\mathbf{B}^{-1})^{\!\top}\mathbf{t}- 2i(\mathbf{l}_\mathbf{B}^{-1})^{\!\top}\boldsymbol{(\pi-a)}-\boldsymbol{x}\right)\right) U_1(\mathbf{t})
        d\mathbf{t}.
        \end{align*}
        Consider the reduced row echelon form of $(\mathbf{l}_{\mathbf{B}}^{-1})^{\!\top}$. Since the matrix has rank $k$, there exist exactly $k$ pivots. Up to renumbering of the tetrahedra, we assume that the positions of the pivots are $1,\dots, k$. Then on the affine subspace 
\begin{align}\label{linsub1}
2(\mathbf{l}_{\mathbf{B}}^{-1})^{\!\top}\mathbf{t} = \boldsymbol{x} + 2i(\mathbf{l}_{\mathbf{B}}^{-1})^{\!\top}(\boldsymbol{\pi}-\boldsymbol{a}),
\end{align}
$t_1,\dots,t_{k}$ can be parametrized by $x_1,\dots, x_k, t_{k+1},\dots, t_{N}$ via some affine isomorphism $L_1: \CC^N \to \CC^N$. In particular, we have
\begin{align}\label{extrafactor1}
    e^{\frac{1}{2\pi\hbar}\big (\sum_{j=1}^k (\mathbf{l}_{j}^{-1})^{\!\top}\mathbf{t}\lambda_{j}(\alpha) \big)}
    =&\ e^{\frac{1}{4\pi\hbar}\big (\sum_{j=1}^k (x_j\lambda_{j}(\alpha) + 2i (\mathbf{l}_{j}^{-1})^{\!\top}(\boldsymbol{\pi-a})\lambda_j(\alpha)\big)} \notag\\
    =&\ e^{\frac{1}{4\pi\hbar}\big (\sum_{j=1}^k (- i x_j\xi_{j} + i x_j \Re \xi_j + 2i (\mathbf{l}_{j}^{-1})^{\!\top}(\boldsymbol{\pi-a})\lambda_j(\alpha)\big)}.
\end{align}

        Similarly, for the second integral, by applying the change of variables that replaces $\hat{t}_j$ with $$ -\frac{1}{2\pi\sqrt{\hbar}}
        \left(\hat{t}_j - i(\pi-a_j)\right),$$ 
        we have
        \begin{align*}
        &\ \int_{\mathbf{\hat{t}}\in \RR^{N}} \delta\left((\mathbf{l}_\mathbf{B}^{-1})^{\!\top}\mathbf{\hat{t}}-\frac{1}{4\pi\sqrt{\hbar}}\boldsymbol{x}\right)
        \left(e^{
        \pi i \mathbf{\hat{t}}^{\!\top} \mathbf{B}^{-1}\mathbf{A}\mathbf{\hat{t}}}
        \prod_{j=1}^N \varphi_{(a_j'',a_j')}(\hat{t}_j) \right)
        d\mathbf{\hat{t}} \\
        = &\ \frac{e^{-\frac{i}{4\pi\hbar}\boldsymbol{(\pi - a)}^{\!\top}\mathbf{B^{-1}A}\boldsymbol{(\pi-a)}
    + \frac{i}{2\pi\hbar}\sum_{j=1}^Na_j(\pi-a_j)}}{(-4\pi\sqrt{\hbar})^{N}} \\ 
   &\ \times 
        \int_{\mathbf{\hat{t}}\in \RR^{N}+i\boldsymbol{(\pi-a)}} 
        \delta\left(\left(2(\mathbf{l}_\mathbf{B}^{-1})^{\!\top}\mathbf{\hat{t}}- 2i(\mathbf{l}_\mathbf{B}^{-1})^{\!\top}\boldsymbol{(\pi-a)} + \boldsymbol{x}\right)\right) U_2(\mathbf{\hat{t}})
        d\mathbf{\hat{t}}.
    \end{align*}
On the affine subspace defined by 
\begin{align}\label{linsub2}2(\mathbf{l}_{\mathbf{B}}^{-1})^{\!\top}\mathbf{\hat{t}} = -\boldsymbol{\hat{x}} + 2i(\mathbf{l}_{\mathbf{B}}^{-1})^{\!\top}(\boldsymbol{\pi}-\boldsymbol{a}), 
\end{align}
$\hat{t}_1,\dots,\hat{t}_{k}$ can be parametrized by an affine map in $\hat{x}_1,\dots, \hat{x}_k, \hat{t}_{k+1},\dots, \hat{t}_{N}$ via some affine isomorphism $L_2:\CC^N \to \CC^N$. In particular, we have
\begin{align}\label{extrafactor2}
    e^{-\frac{1}{2\pi\hbar}\big (\sum_{j=1}^k (\mathbf{l}_{j}^{-1})^{\!\top}\mathbf{\hat{t}}\lambda_{j}(\alpha) \big)}
    =&\ e^{\frac{1}{4\pi\hbar}\big (\sum_{j=1}^k (x_j\lambda_{j}(\alpha) - 2i (\mathbf{l}_{j}^{-1})^{\!\top}(\boldsymbol{\pi-a})\lambda_j(\alpha)\big)} \notag\\
    =&\ e^{\frac{1}{4\pi\hbar}\big (\sum_{j=1}^k (-ix_j(-\overline{\xi_{j}}) - i x_j \Re \xi_j - 2i (\mathbf{l}_{j}^{-1})^{\!\top}(\boldsymbol{\pi-a})\lambda_j(\alpha)\big)}.
\end{align}
Equations \eqref{extrafactor1} and \eqref{extrafactor2} together give the exponential factors in the definitions of $K_1$ and $K_2$.

Finally, by comparing Equation \eqref{linsub1} and Equation \eqref{linsub2}, we can see that
\begin{align}\label{detL1L2}
\det(L_1)=(-1)^k \det(L_2).
\end{align}
This completes the proof.
\end{proof}

\subsection{Potential functions and their properties}
In this section, based on different expressions of the KLV potential function in Proposition \ref{KLVexpress2}, we define 
\begin{itemize}
    \item potential functions $F_1$ and $F_2$ in Equations \eqref{defnF1} and \eqref{defnF2} associated with $U_1$ and $U_2$;
    \item potential functions $F_1^J$ and $F_2^J$ in Equations \eqref{defnF1J} and \eqref{defnF2J} associated with $K_1$ and $K_2$;
    \item potential functions $J_1, J_2$ in Equations \eqref{defnJ1} and \eqref{defnJ2} associated with $H_1 \circ L_1$ and $H_2 \circ L_2$; and
    \item the joint potential function $F$ in Equation \eqref{defnF}.
\end{itemize}

\subsubsection{Geometry of the critical point equations}
Given $\boldsymbol{\xi}= (\xi_1,\dots,\xi_k)\in \CC^k$, we consider the following potential functions
\begin{align}
    F_1\big(\mathbf{t};\boldsymbol{\xi}\big)
    :=&\ -\frac{i}{2}\mathbf{t}^{\!\top} \mathbf{B}^{-1}\mathbf{A} \mathbf{t} + \mathbf{t}^{\!\top}(\mathbf{B}^{-1}(\boldsymbol{\nu}-i\boldsymbol{ u}) - \mathbf{B}^{-1}\mathbf{A} \boldsymbol{\pi})
    - i\sum_{j=1}^N \Li\left(-e^{t_j}\right), \label{defnF1}\\
    F_2\big(\mathbf{\hat t};\boldsymbol{\xi}\big)
    :=&\ \frac{i}{2}\mathbf{\hat t}^{\!\top} \mathbf{B}^{-1}\mathbf{A} \mathbf{\hat t} - \mathbf{\hat t}^{\!\top}(\mathbf{B}^{-1}(\boldsymbol{\nu}-i\boldsymbol{ u}) - \mathbf{B}^{-1}\mathbf{A} \boldsymbol{\pi})
    + i\sum_{j=1}^N \Li\left(-e^{- \hat t_k}\right) \label{defnF2},
\end{align}
where $\boldsymbol{{u}} = (0,\dots, 0, \xi_1,\dots,\xi_k)^{\!\top}\in \CC^N$. 
Given any vector $\mathbf{v}=(v_1,\dots,v_N) \in \CC^N$, we let $\overline{\mathbf{v}} = (\overline{v_1},\dots,\overline{v_N}),$
where $\overline{v_j}$ is the complex conjugate of $v_j$ for $j=1,\dots,N$.
\begin{proposition}\label{critThurscorrespondence} Given $\boldsymbol{\xi}\in \CC^k$, let $\mathbf{u}=(0,\dots,0,\xi_1,\dots,\xi_k)^{\!\top}$.
\begin{enumerate}
    \item Via the bijection $t_j = -\Log z_j + i \pi$, we have the correspondence between the critical point equation 
    $$ \nabla F_1(\mathbf{t};\boldsymbol{\xi})=\mathbf{0}$$
    with the edge and holonomy equations
    \begin{align}\label{hol1}
         \mathbf{A}\BLog \mathbf{z} + \mathbf{B} \BLog\mathbf{z''} = i \boldsymbol{\nu} +\boldsymbol{u}.
    \end{align}
    \item Via the bijection $\hat{t}_j = \Log \hat{z}_j + i \pi$, we have the correspondence between the critical point equation 
    $$ \nabla F_2(\mathbf{\hat t};\boldsymbol{\xi})=\mathbf{0}$$
    with the edge and holonomy equations
    \begin{align}\label{hol2}
         \mathbf{A}\BLog \mathbf{\hat{z}} + \mathbf{B} \BLog\mathbf{\hat{z}''} = -i \boldsymbol{\nu}- {\boldsymbol{u}}.
    \end{align}
    \item If $\mathbf{z}(\boldsymbol{\xi})$ is a solution of Equation \eqref{hol1}, then $\mathbf{\hat{z}}(\boldsymbol{\xi}) = \overline{\mathbf{z}(-\boldsymbol{\overline{\xi}})}$ is a solution of Equation \eqref{hol2}. In particular, if
     Equation \eqref{hol1} admits a solution  $z_j = r_je^{i\phi_j}$ with $\phi_j>0$. Then for $j=1,\dots,N$, $t_{c,j} = -\log r_j + i(\pi-\phi_j)$ and $\hat{t}_{c,j} = \log r_j + i(\pi-\phi_j)$ are solutions of $\nabla F_1(\mathbf{t};\boldsymbol{\xi})=\mathbf{0}$ and $\nabla F_2(\mathbf{\hat{t}};-\boldsymbol{\overline{\xi}})=\mathbf{0}$ respectively.
\end{enumerate}
\end{proposition}
\begin{proof}
    Since $t_j = -\Log z_j + i\pi$, we have $\Log z_j = -t_j+i\pi$ and $\Log z_j'' = \Log\big(1+e^{t_j}\big)$. As a result, since $\mathbf{B}^{-1}\mathbf{A}$ is symmetric, we have
    \begin{align}\label{nablaF_1}
        \nabla F_1(\mathbf{t};\boldsymbol{\xi})
        =&\  -i\mathbf{B}^{-1}\mathbf{A}\mathbf{t}
        + (\mathbf{B}^{-1}(\boldsymbol{\nu}-i\boldsymbol{u}) - \mathbf{B}^{-1}\mathbf{A} \boldsymbol{\pi})
        + i\sum_{j=1}^N\BLog\big(1+e^{t_j}\big) \notag\\
        =&\ i\left(\mathbf{B}^{-1}\mathbf{A}(-\mathbf{t}+i\boldsymbol{\pi}) - i\mathbf{B}^{-1}(\boldsymbol{\nu}-i\boldsymbol{u}) + \sum_{j=1}^N\BLog\big(1+e^{t_j}\big)\right) \notag\\
        =&\ i\left( \mathbf{B}^{-1}\mathbf{A}\BLog \mathbf{z} + \BLog\mathbf{z''} - i\mathbf{B}^{-1}(\boldsymbol{\nu}-i\boldsymbol{u}) \right).
    \end{align}
This proves (1). For (2), similarly, since $\hat{t}_j = \Log \hat{z}_j + i\pi$, we have $\Log \hat{z}_j = \hat{t}_j - i\pi$ and $\Log \hat{z}_j'' = \Log(1+e^{-\hat{t}_j})$. Since $\mathbf{B}^{-1}\mathbf{A}$ is symmetric, we have
\begin{align}\label{nablaF_2}
        \nabla F_2(\mathbf{\hat{t}}; \boldsymbol{\xi})
        =&\  i\mathbf{B}^{-1}\mathbf{A}\mathbf{\hat{t}}
        - (\mathbf{B}^{-1}(\boldsymbol{\nu}-i\boldsymbol{u}) - \mathbf{B}^{-1}\mathbf{A} \boldsymbol{\pi})
        + i\sum_{j=1}^N\BLog\big(1+e^{-\hat{t}_j}\big) \notag\\
        =&\ i\left(\mathbf{B}^{-1}\mathbf{A}(\mathbf{\hat{t}}-i\boldsymbol{\pi}) + i\mathbf{B}^{-1}(\boldsymbol{\nu}-i\boldsymbol{u}) + \sum_{j=1}^N\BLog\big(1+e^{-\hat{t}_j}\big)\right) \notag \\
        =&\ i\left( \mathbf{B}^{-1}\mathbf{A}\BLog \mathbf{\hat{z}} + \BLog\mathbf{\hat{z}''} + i\mathbf{B}^{-1}(\boldsymbol{\nu}-i\boldsymbol{u}) \right).
    \end{align}
This proves (2). Finally, for (3), note that by taking the complex conjugate of both sides of \eqref{hol1} with $\boldsymbol{u}=(0,\dots,0,\xi_1,\dots,\xi_k)$, one gets a solution of \eqref{hol2} with $\boldsymbol{u} = (0,\dots, 0, -\overline{\xi_1},\dots, \overline{\xi_k})$ given by $\BLog \mathbf{\hat{z}} = \overline{\BLog \mathbf{z}}$. In particular, if $z_j = r_j e^{i\phi_j}$ with $\phi_j>0$, then we have
$$
t_j = -\Log z_j + i\pi = -\log r_j + i(\pi-\phi_j)
$$
and
$$
\hat{t}_j = \Log \hat{z}_j + i\pi = \log r_j + i(\pi-\phi_j).
$$
This completes the proof.
\end{proof}

\subsubsection{Critical values of potential functions}
\begin{proposition}\label{realpartpotential}
    Write $t_j = h_j + i d_j \in \CC$ and $\hat{t}_j = \hat{h}_j + i \hat{d}_j \in \CC$ for $j = 1,\dots, N$. Let $\mathbf{h}=(h_1,\dots, h_N)$ and $\mathbf{d}=(d_1,\dots,d_N)$. Under the correspondence in Proposition \ref{critThurscorrespondence}, 
    we have 
\begin{align*}
\mathrm{Re}F_1 (\mathbf{t}; \boldsymbol{\xi})
=&\ - \sum_{j=1}^N D(z_j) + 
\sum_{j=1}^k \left( (\mathbf{l}_{j}^{-1})^{\!\top}\mathbf{d} \Re\xi_j \right)
    + \sum_{j=1}^N h_j \frac{\partial}{\partial h_j} 
\mathrm{Re}F_1\left (\mathbf{t}; \boldsymbol{\xi}\right ), \\
\mathrm{Re}F_2 (\mathbf{\hat{t}}; -\boldsymbol{\overline{\xi}})
=&\ \sum_{j=1}^N D(\hat{z}_j) + 
\sum_{j=1}^k \left( (\mathbf{l}_{j}^{-1})^{\!\top}\mathbf{d} \Re\xi_j \right)
    + \sum_{j=1}^N \hat{h}_j \frac{\partial}{\partial \hat{h}_j} 
\mathrm{Re}F_2\left (\mathbf{\hat{t}}; -\boldsymbol{\overline{\xi}}\right )
\end{align*}
where $D(z)$ is the Bloch-Wigner dilogarithm function given by
$$ D(z) 
= \mathrm{Im} \mathrm{Li}_2(z) + \log|z| \mathrm{Arg}(1-z).$$
In particular, if we let $\mathbf{t}_c = (t_{c,1},\dots, t_{c,N})$ and $\mathbf{\hat{t}}_c = (\hat{t}_{c,1},\dots,\hat{t}_{c,N})$ be the critical point of $F_1$ and $F_2$ described in Proposition \ref{critThurscorrespondence}(3) respectively, then we have
$$
\mathrm{Re}F_1\left (\mathbf{t}_c(\alpha);\boldsymbol{\xi}\right ) 
+
\mathrm{Re}F_2\left (\mathbf{\hat{t}}_c(\alpha); -\boldsymbol{\overline{\xi}}\right )
=  - 2 \Vol(
M; l, \boldsymbol{\xi}) + 
2\sum_{j=1}^k \left( (\mathbf{l}_{j}^{-1})^{\!\top}\boldsymbol{(\pi - a)} \Re\xi_j \right),
$$
 where $ \Vol(
M; l, \boldsymbol{\xi})$ is the  volume of the representation induced by $\mathbf{z}$. 
\end{proposition}
\begin{proof}
Note that for $t=h+id$, we have
$$
\frac{d}{d t} \Big(i\mathrm{Li}_2(-e^{t}) \Big) =  -i \Log \big(1+e^{t}\big).
$$
By using the Cauchy-Riemann equation, we have
$$
\frac{\partial}{\partial h} \mathrm{Re}\Big(i \mathrm{Li}_2\big(-e^{t}\big)\Big)  = \mathrm{Re} \Bigg( \frac{d}{dt} \Big(i \mathrm{Li}_2\big(-e^{t}\big)\Big) \Bigg) =  \mathrm{Arg} \big(1+e^{-t}\big).
$$
As a result, by the definition of the Bloch-Wigner dilogarithm function,
$$
\mathrm{Re}\Big(i \mathrm{Li}_2\big(-e^{t}\big)\Big)  
=
-\mathrm{Im} \mathrm{Li}_2\big(-e^t\big)
=  - D\big(-e^t\big) + h \mathrm{Arg}\big(1+e^t\big)
= -D\big(-e^t\big) + h\frac{\partial}{\partial h} \mathrm{Re}\Big(i \mathrm{Li}_2\big(-e^t\big)\Big).
$$
Thus, for $j=1,\dots,N$, we have
$$
\mathrm{Re}\Big(- i \mathrm{Li}_2\big(-e^{t_j}\big)\Big)  
= D\big(-e^{t_j}\big) + h_j\frac{\partial}{\partial h_j} \mathrm{Re}\Big(-i \mathrm{Li}_2\big(-e^{t_j}\big)\Big)
= - D(z_j) + h_j\frac{\partial}{\partial h_j} \mathrm{Re}\Big(-i \mathrm{Li}_2\big(-e^{t_j}\big)\Big).
$$ 
Similarly, we have
$$
\frac{d}{d t} \Big(i\mathrm{Li}_2(-e^{-t}) \Big) =  i \Log (1+e^{-t}).
$$
By using the Cauchy-Riemann equation, we have
$$
\frac{\partial}{\partial h} \mathrm{Re}\Big(i \mathrm{Li}_2(-e^{-t})\Big)  = \mathrm{Re} \Bigg( \frac{d}{dt} \Big(i \mathrm{Li}_2(-e^{-t})\Big) \Bigg) =  -\mathrm{Arg} (1+e^{-t}).
$$
As a result, by the definition of the Bloch-Wigner dilogarithm function, for $j=1,\dots, N$,
\begin{align*}
\mathrm{Re}\Big(i \mathrm{Li}_2(-e^{-\hat{t}_j})\Big)  
=&\ 
-\mathrm{Im} \mathrm{Li}_2(-e^{-\hat{t}_j})
=  - D(-e^{-\hat{t}_j}) - \hat{h}_j \mathrm{Arg}(1+e^{-\hat{t}_j})\\
=&\ D(w_j) + \hat{h}_j\frac{\partial}{\partial \hat{h}_j} \mathrm{Re}\Big(i \mathrm{Li}_2(-e^{-\hat{t}_j})\Big).
\end{align*}
Besides, let $\mathbf{h}=(h_1,\dots, h_N)$ and $\mathbf{d}=(d_1,\dots,d_N)$ be the vectors of real and imaginary parts of $\mathbf{t}$ respectively. Note that
\begin{align*}
    \Re\left( \mathbf{t}^{\!\top} \mathbf{B}^{-1}(-i\boldsymbol{u}) \right)
    =&\  \Re\left( \sum_{j=1}^k (\mathbf{l}_{j}^{-1})^{\!\top}\mathbf{t} (-i\xi_j) \right)\\
    =&\  \sum_{j=1}^k \left( (\mathbf{l}_{j}^{-1})^{\!\top}\mathbf{h} \lambda_j(\alpha) + (\mathbf{l}_{j}^{-1})^{\!\top}\mathbf{d} \Re\xi_j \right)\\
    =&\ \sum_{j=1}^k \left( (\mathbf{l}_{j}^{-1})^{\!\top}\mathbf{d} \Re\xi_j \right)
    + \sum_{j=1}^N h_j \frac{\partial}{\partial h_j}\Re\left( \mathbf{t}^{\!\top} \mathbf{B}^{-1}(-i\boldsymbol{u}) \right).
\end{align*}
Similarly, we have
\begin{align*}
    \Re\left( \mathbf{\hat{t}}^{\!\top} \mathbf{B}^{-1}(-i\boldsymbol{\overline{u}}) \right)
    =&\  \Re\left( \sum_{j=1}^k (\mathbf{l}_{j}^{-1})^{\!\top}\mathbf{\hat{t}} (-i\overline{\xi_j}) \right)\\
    =&\  \sum_{j=1}^k \left( -(\mathbf{l}_{j}^{-1})^{\!\top}\mathbf{\hat{h}} \lambda_j(\alpha) + (\mathbf{l}_{j}^{-1})^{\!\top}\mathbf{\hat{d}} \Re\xi_j \right)\\
    =&\ \sum_{j=1}^k \left( (\mathbf{l}_{j}^{-1})^{\!\top}\mathbf{\hat{d}} \Re\xi_j \right)
    + \sum_{j=1}^N \hat{h}_j \frac{\partial}{\partial \hat{h}_j}\Re\left( \mathbf{\hat{t}}^{\!\top} \mathbf{B}^{-1}(-i\boldsymbol{\overline{u}}) \right).
\end{align*}

Finally, note that the real part of all the remaining non-dilogarithm terms in $F_1$ and $F_2$ are linear in $\{h_j \mid j=1,\dots, N\}$ and linear in $\{\hat{h}_j \mid j=1,\dots, N\}$ respectively. Thus, we have
\begin{align*}
    &\ \Re\left(-\frac{i}{2}\mathbf{t}^{\!\top} \mathbf{B}^{-1}\mathbf{A} \mathbf{t} + \mathbf{t}^{\!\top}(\mathbf{B}^{-1}\boldsymbol{\nu} - \mathbf{B}^{-1}\mathbf{A} \boldsymbol{\pi}) \right)\\
    =&\ \sum_{j=1}^N h_j \frac{\partial}{\partial h_j}\Re\left(-\frac{i}{2}\mathbf{t}^{\!\top} \mathbf{B}^{-1}\mathbf{A} \mathbf{t} + \mathbf{t}^{\!\top}(\mathbf{B}^{-1}\boldsymbol{\nu} - \mathbf{B}^{-1}\mathbf{A} \boldsymbol{\pi}) \right)
\end{align*}
and
\begin{align*}
    &\ \Re\left(\frac{i}{2}\mathbf{\hat t}^{\!\top} \mathbf{B}^{-1}\mathbf{A} \mathbf{\hat t} - \mathbf{\hat t}^{\!\top}(\mathbf{B}^{-1}\boldsymbol{\nu} - \mathbf{B}^{-1}\mathbf{A} \boldsymbol{\pi})\right) \\
    =&\ \sum_{j=1}^N \hat{h}_j \frac{\partial}{\partial \hat{h}_j}\left(\frac{i}{2}\mathbf{\hat t}^{\!\top} \mathbf{B}^{-1}\mathbf{A} \mathbf{\hat t} - \mathbf{\hat t}^{\!\top}(\mathbf{B}^{-1}\boldsymbol{\nu} - \mathbf{B}^{-1}\mathbf{A} \boldsymbol{\pi})\right).
\end{align*}
By adding the above equations together, we obtain the differential equations for $\mathrm{Re}F_1 (\mathbf{t}; \boldsymbol{\xi})$ and $\mathrm{Re}F_2 (\mathbf{\hat{t}}; \boldsymbol{\xi})$. Finally, we have
\begin{align*}
    \mathrm{Re}F_1 (\mathbf{t_c(\alpha)}; \boldsymbol{\xi})
    =&\ -\sum_{j=1}^N D(z_{c,l}(\boldsymbol{\xi}))
    + 
\sum_{j=1}^k \left( (\mathbf{l}_{j}^{-1})^{\!\top}\mathbf{d_c} \Re\xi_j \right) \\
    =&\ -\Vol\Big(
M; l, \boldsymbol{\xi}
\Big)+ 
\sum_{j=1}^k \left( (\mathbf{l}_{j}^{-1})^{\!\top}\mathbf{d_c} \Re\xi_j \right).
\end{align*}
Similarly, by Proposition  \ref{critThurscorrespondence}(3), since $\mathbf{\hat{d}_c} = \mathbf{d_c}$, we have
\begin{align*}
    \mathrm{Re}F_2 (\mathbf{\hat{t}_c(\alpha)}; -\boldsymbol{\overline{\xi}})
    =&\ \sum_{j=1}^N D(\hat{z}_{c,l}(\boldsymbol{-\overline{\xi}}))
    + 
\sum_{j=1}^k \left( (\mathbf{l}_{j}^{-1})^{\!\top}\mathbf{\hat{d}_c} \Re(\xi_j) \right)\\
    =&\ \sum_{j=1}^N D\big(\overline{z_{c,l}(\boldsymbol{\xi})}\big)
    + 
\sum_{j=1}^k \left( (\mathbf{l}_{j}^{-1})^{\!\top}\mathbf{d_c} \Re\xi_j \right) \\
    =&\ -\Vol\Big(
M; l, \boldsymbol{\xi}
\Big) + 
\sum_{j=1}^k \left( (\mathbf{l}_{j}^{-1})^{\!\top}\mathbf{d_c} \Re\xi_j \right) ,
\end{align*}
where in the last equality, we use the fact that
$$
\sum_{j=1}^N D\big(\overline{z_{c,l}(\boldsymbol{\xi})}\big)
= -\sum_{j=1}^N D(z_{c,l}(\boldsymbol{\xi}))
= -\Vol\Big(
M; l, \boldsymbol{\xi}
\Big).
$$
Finally, note that at the critical points $\mathbf{t_c}$ and $\mathbf{\hat{t}_c}$, we have 
$$
\mathbf{d_c} = \mathbf{\hat{d}_c} = \boldsymbol{\pi-a}.
$$
This completes the proof.
\end{proof}

\subsubsection{Hessian of the potential functions at critical points}
\begin{proposition}\label{concavity}
Write $t_j=h_j+id_j$ and $\hat{t}_j=\hat{h}_j+i\hat{d}_j$ for $j=1,\dots, N$. For any angle structure $\alpha$ with $(a_1,\dots,a_N) = (\alpha(q_1),\dots,\alpha(q_N))$, on the real multi-contour 
$$  \prod_{j=1}^N (\RR + i (\pi-a_j)),$$ 
$\Re F_1(\mathbf{t}; \boldsymbol{\xi})$ and $\Re F_2(\mathbf{\hat{t}}; -\boldsymbol{\overline{\xi}})$
are strictly concave in the real variables $h_j$ and $\hat{h}_j$ respectively.
\end{proposition}
\begin{proof}
First, note that the hessian of $\Re F_1$ on the real multi-contour is the same as the real part of the holomorphic hessian of $F_1$. By a direct computation, for any $\mathbf{t} = \mathbf{h} + i \mathbf{d}$ with $ \mathbf{h} = (h_1,\dots,  h_N), \mathbf{d} = (d_1,\dots,d_N) \in \RR^{N}$, we have 
\begin{align*}
\big(\mathrm{Hess}(\Re (F_1)|_{\mathscr{Y}^0_\alpha}) \big) (\mathbf{h} + i \mathbf{d}; \boldsymbol{\xi})
= \mathrm{Re} ( \mathrm{Hess}( F_1 ) (\mathbf{h} + i \mathbf{d}; \boldsymbol{\xi}))=\Delta_1
\end{align*} 
where $\Delta_1$ is the $N \times N$ diagonal matrix with entries 
$$ - \mathrm{Im} \bigg( \frac{1}{1+e^{-h_j - i d_j}}\bigg).$$
By Definition \ref{defnanglestr}, since $\pi-a_j \in (0,\pi)$ for $j=1,\dots,N$, the matrix $\Delta_1$ is negative definite for all $\mathbf{h} \in \RR^{N}$. This proves the first claim. For the second claim, note that 
\begin{align*}
\big(\mathrm{Hess}(\Re (F_2)|_{\mathscr{Y}^0_\alpha}) \big) (\mathbf{\hat{h}} + i \mathbf{\hat{d}}; -\boldsymbol{\overline\xi})
= \mathrm{Re} ( \mathrm{Hess}( F_2 ) (\mathbf{\hat{h}} + i \mathbf{\hat{d}}; -\boldsymbol{\overline{\xi}}))=\Delta_2
\end{align*} 
where $\Delta_2$ is the $N \times N$ diagonal matrix with entries 
$$ \mathrm{Im} \bigg( \frac{1}{1+e^{\hat{h}_j + i \hat{d}_j}}\bigg).$$
By Definition \ref{defnanglestr}, since $\pi-a_j \in (0,\pi)$ for $j=1,\dots,N$, the matrix $\Delta_2$ is negative definite for all $\mathbf{h} \in \RR^{N}$. This proves the second claim.
\end{proof}

For $\boldsymbol{\xi}\in \CC^k$,
    let 
    $\mathbf{t_c}(\boldsymbol{\xi})=(t_{c,1}(\boldsymbol{\xi}),\dots,t_{c,N}(\boldsymbol{\xi}))$
    and
    $\mathbf{\hat{t}_c}(-\boldsymbol{\overline{\xi}})=
    (\hat{t}_{c,1}(-\boldsymbol{\overline{\xi}}),\dots,\hat{t}_{c,N}(-\boldsymbol{\overline{\xi}}))$
    be critical points 
    of $ F_1$ and  of $ F_2$ respectively. 
\begin{proposition}\label{detHessF12}
    Under the correspondence described in Proposition \ref{critThurscorrespondence}, we have 
\begin{align*}
 \det \Hess (F_1)(\mathbf{t_c}(\boldsymbol{\xi});\boldsymbol{\xi})
&= 2 i^N \det\mathbf{B}^{-1} \left(\prod_{j=1}^N z_j^{-f_j''} z_j''^{f_j - 1} \right)
\tau(M, l, X, \mathbf{z})
\end{align*}
and
\begin{align*}
 \det \Hess (F_2)(\mathbf{\hat{t}_c}(-\boldsymbol{\overline{\xi}});-\boldsymbol{\overline{\xi}})
&=  (-1)^N 2 i^N \det\mathbf{B}^{-1} \left(\prod_{j=1}^N \hat{z}_j^{-f_j''} \hat{z}_j''^{f_j - 1} \right)
\tau(M, l, X, \mathbf{\hat{z}}),
\end{align*}
where $z_j=-e^{-t_{c,j}(\boldsymbol{\xi})}$, $\hat{z}_j = -e^{\hat{t}_{c,k}(-\boldsymbol{\overline{\xi}})}$, $(\mathbf{f},\mathbf{f}',\mathbf{f}'')$ is any strong combinatorial flattening defined in Definition \ref{SCF} and $\tau(M, l, X, \mathbf{\hat{z}})$ is the 1-loop invariant defined in Definition \ref{defn1loop}. In particular, we have
$$
\det \Hess (F_1)(\mathbf{t_c}(\boldsymbol{\xi});\boldsymbol{\xi})
= \overline{\det \Hess (F_2)(\mathbf{\hat{t}_c}(-\boldsymbol{\overline{\xi}});-\boldsymbol{\overline{\xi}})}.
$$
\end{proposition}
\begin{proof}
Note that by (\ref{nablaF_1}), since 
\begin{align*}
\nabla F_1(\mathbf{t};\boldsymbol{\xi})
&= i\left( \mathbf{B}^{-1}\mathbf{A}\BLog \mathbf{z} + \BLog\mathbf{z''} - i\mathbf{B}^{-1}(\boldsymbol{\nu}-i\boldsymbol{u}) \right) .
\end{align*}
we have
\begin{align*}
- i \mathbf{B} 
\nabla  F_1
= \mathbf{A} \BLog \mathbf{z}
+ \mathbf{B} \BLog \mathbf{z''}
- i(\boldsymbol \nu -i\boldsymbol{u}).
\end{align*}
Under the bijection that $t_j = -\Log z_j + i\pi$, we have $z_j = -e^{-t_j}$. In particular, we have 
$$ \frac{\partial}{\partial t_j } =  \frac{\partial z_j}{\partial t_j } \cdot \frac{\partial}{\partial z_j } = - z_j \frac{\partial}{\partial z_j } . $$ 
Thus,
\begin{align*}
\det \Hess (F_1 (\mathbf{t});\boldsymbol{\xi})
=&\  \det D_\mathbf{t} \nabla F_1 (\mathbf{t};\boldsymbol{\xi}) \\
=&\ (-1)^N i^{N} \det \mathbf{B}^{-1} \left(\prod_{j=1}^N -z_j \right)\det D_{\mathbf{z}} \left(\mathbf{A} \BLog \mathbf{z}
+ \mathbf{B} \BLog \mathbf{z''} \right) \\
=&\ i^{N} \det \mathbf{B}^{-1}  \left(\prod_{j=1}^N z_j \right)\det (\mathbf{A} \Delta_{z''} + \mathbf{B} \Delta_{z}^{-1}) \left(\prod_{j=1}^N \frac{1}{1-z_j} \right)\\
=&\  2 i^N \det \mathbf{B}^{-1}   \left(\prod_{j=1}^N z_j \right)\left( \prod_{j=1}^N z_j^{-f_j''} z_j''^{f_j} \right) 
\tau(M,X, l, \mathbf{z})
\left(\prod_{j=1}^N \frac{1}{1-z_j} \right)\\
=&\ 2 i^N \det \mathbf{B}^{-1} \left(\prod_{j=1}^N z_j^{-f_j''} z_j''^{f_j - 1} \right)
\tau(M,X, l, \mathbf{z}).
\end{align*}
This proves the first claim. For the second claim, since $\hat{t}_j = \Log \hat{z}_j$, we have $\hat{z}_j = -e^{\hat{t}_j}$. In particular, we have
\begin{align*}
    \frac{\partial}{\partial {\hat{t}}_j} = \frac{\partial \hat{z}_j}{\partial {\hat{t}}_j} \cdot \frac{\partial}{\partial \hat{z}_j} = \hat{z}_j \frac{\partial}{\partial \hat{z}_j} 
\end{align*}

By the same argument, we have
\begin{align*}
& \det \Hess (F_2 (\mathbf{t});\boldsymbol{\xi})
= (-1)^N 2 i^N \det \mathbf{B}^{-1} \left(\prod_{j=1}^N \hat{z}_j^{-f_j''} \hat{z}_j''^{f_j - 1} \right)
\tau(M,X, l, \mathbf{w}).
\end{align*}
By Proposition \ref{critThurscorrespondence} (3), we have the last result.
\end{proof}

\subsection{Potential functions for the Reshetikhin--Turaev type functions}

Let $L_1,L_2$ be the affine isomorphisms introduced in Proposition \ref{KLVexpress2}.
We define
\begin{align}\label{defnF1J}
F_1^J(\boldsymbol{x},t_{k+1},\dots,t_{N} ; \boldsymbol{\xi})
:=&\ F_1( L_1(\boldsymbol{x},t_{k+1},\dots,t_{N}) ;\boldsymbol{\xi}) \\
=&\ J_1(\boldsymbol{x},t_{k+1},\dots,t_{N}) - \frac{i\boldsymbol{\xi} \cdot\boldsymbol{x}}{2} +  \boldsymbol{\xi}(\mathbf{l}_{\mathbf{B}}^{-1})^{\!\top}(\boldsymbol{\pi-a}) \notag,
\end{align}
where
\begin{align}\label{defnJ1}
    J_1(\boldsymbol{x},t_{k+1},\dots,t_{N})
    = -\frac{i}{2}\mathbf{t}^{\!\top} \mathbf{B}^{-1}\mathbf{A} \mathbf{t} + \mathbf{t}^{\!\top}(\mathbf{B}^{-1}\boldsymbol{\nu} - \mathbf{B}^{-1}\mathbf{A} \boldsymbol{\pi})
    - i\sum_{j=1}^N \Li\left(-e^{t_j}\right)
\end{align}
with $(t_1,\dots,t_N)=L_1(\boldsymbol{x},t_{k+1},\dots,t_N)$. Similarly, we define
\begin{align}\label{defnF2J}
F_2^J(\boldsymbol{\hat{x}},\hat{t}_{k+1},\dots,\hat{t}_{N}; \boldsymbol{\xi})
:=&\ F_2( L_2(\boldsymbol{\hat{x}},\hat{t}_{k+1},\dots,\hat{t}_{N}); \boldsymbol{\xi}) \\
=&\ J_2(\boldsymbol{\hat{x}},\hat{t}_{k+1},\dots,\hat{t}_{N}) - \frac{i\boldsymbol{\xi} \cdot \boldsymbol{\hat{x}}}{2} -  \boldsymbol{\xi}(\mathbf{l}_{\mathbf{B}}^{-1})^{\!\top}(\boldsymbol{\pi-a}) \notag,
\end{align}
where
\begin{align}\label{defnJ2}
    J_2(\boldsymbol{\hat{x}},\hat{t}_{k+1},\dots,\hat{t}_{N})
    = \frac{i}{2}\mathbf{\hat{t}}^{\!\top} \mathbf{B}^{-1}\mathbf{A} \mathbf{\hat{t}} - \mathbf{\hat{t}}^{\!\top}(\mathbf{B}^{-1}\boldsymbol{\nu} - \mathbf{B}^{-1}\mathbf{A} \boldsymbol{\pi})
    + i\sum_{j=1}^N \Li\left(-e^{-\hat{t}_j}\right)
\end{align}
with $(\hat{t}_1,\dots,\hat{t}_N)=L_2(\boldsymbol{\hat{x}},\hat{t}_{k+1},\dots,\hat{t}_N)$.

\subsubsection{Geometric interpretations of $\boldsymbol{x}$ and $\boldsymbol{\hat{x}}$}

\begin{lemma}\label{xitox}
With the notations in Proposition \ref{critThurscorrespondence}, we have the following results. 
    \begin{enumerate}
        \item Under the correspondence in Proposition \ref{critThurscorrespondence}, at the critical point of $F_1(\mathbf{t};\boldsymbol{\xi})$, we have
    $$ 
    x_j(\boldsymbol{\xi}) = \mathrm{H}^\CC_{X,m_j}(\mathbf{z}(\boldsymbol{\xi})) 
        + \sum_{p=1}^{k}n_{p,j}\xi_p- i\left(\mu_{j}(\alpha)
    + \sum_{p=1}^{k}n_{p,j} \lambda_{p}(\alpha) \right)
    $$
    for $j=1,\dots, k$.
    \item Under the correspondence in Proposition \ref{critThurscorrespondence}, at the critical point of $F_2(\mathbf{\hat t};\boldsymbol{\xi})$, we have
    \begin{align*}
    \hat{x}_j(\boldsymbol{\xi}) 
    =&\ \mathrm{H}^\CC_{X,m_j}(\mathbf{\hat{z}}(\boldsymbol{\xi})) 
        - \sum_{p=1}^{k}n_{p,j}\xi_p +i\left( 
    \mu_{j}(\alpha) 
    + \sum_{p=1}^{k}n_{p,j}\lambda_{p}(\alpha)
    \right) \\
    =&\ \overline{\mathrm{H}^\CC_{X,m_j}(\mathbf{z}(-\overline{\boldsymbol{\xi}})) 
        + \sum_{p=1}^{k}n_{p,j}(-\overline{\xi}_p })
        +i\left( 
    \mu_{j}(\alpha) 
    + \sum_{p=1}^{k}n_{p,j}\lambda_{p}(\alpha)
    \right) 
    \end{align*}
    for $j=1,\dots, k$.
    \end{enumerate} 
\end{lemma}
\begin{proof}
    By Lemma \ref{n1n2defn}, we have
    \begin{align*}
        x_j(\boldsymbol{\xi})=&\ 2
        (\mathbf{l}_j^{-1})^{\!\top}(\mathbf{t_c(\boldsymbol{\xi)}} - i(\boldsymbol{\pi} - \boldsymbol{a})) \\
        = &\ 2(\mathbf{l}_j^{-1})^{\!\top}(- \BLog \mathbf{z}\boldsymbol{(\xi}) + i \boldsymbol{a})\\
        = &\ \mathrm{H}^\CC_{X,m_j}(\mathbf{z}(\boldsymbol{\xi)}) 
        + \sum_{p=1}^{k}n_{p,j}\xi_p
        - i\left(\mu_{j}(\alpha)
    + \sum_{p=1}^{k}n_{p,j} \lambda_{p}(\alpha) \right) .
    \end{align*}
    Similarly, by Lemma \ref{n1n2defn2},
    \begin{align*}
        \hat{x}_j (\boldsymbol{\xi})=&\ 2(\mathbf{l}_j^{-1})^{\!\top}(-\mathbf{\hat{t}_c}(\boldsymbol{\xi})+i(\boldsymbol{\pi}-\boldsymbol{a}))\\
        =&\  2(\mathbf{l}_j^{-1})^{\!\top}(-\BLog \mathbf{\hat{z}}(\boldsymbol{\xi}) -i\boldsymbol{a})\\
        =&\ \mathrm{H}^\CC_{X,m_j}(\mathbf{\hat{z}}(\boldsymbol{\xi}))
    -\sum_{p=1}^k n_{p,j}\xi_p
    +i\left( 
    \mu_{j}(\alpha) 
    + \sum_{p=1}^{k}n_{p,j}\lambda_{p}(\alpha)
    \right).
    \end{align*}
    The last equality follows from Proposition \ref{critThurscorrespondence}(3).
\end{proof}

\begin{lemma}\label{xxhat}
    For $\boldsymbol{u}=(0,\dots,0,\xi_1,\dots,\xi_k)$, suppose the equation 
    \begin{align*}
         \mathbf{A}\BLog \mathbf{z} + \mathbf{B} \BLog\mathbf{z''} = i \boldsymbol{\nu} + \boldsymbol{u}
    \end{align*}
    admits a solution 
    $$\mathbf{z}(\boldsymbol{\xi}) = (z_1(\boldsymbol{\xi}), \dots, z_N(\boldsymbol{\xi})),$$
    where $z_j(\boldsymbol{\xi}) = r_je^{i\phi_j}$ with $\phi_j\in(0,\pi)$ for all $j=1,\dots, N$. Let $\alpha$ be the angle structure induced by $\mathbf{z}(\boldsymbol{\xi})$. Then we have $x_j(\boldsymbol{\xi}) = \hat{x}_j(-\boldsymbol{\overline{\xi}})$ for $j=1,\dots, N$.
\end{lemma}
\begin{proof}
Let $\boldsymbol{\phi}=(\phi_1,\dots,\phi_N)$.
    By Proposition \ref{critThurscorrespondence}(3), we have
    $$
    x_j(\boldsymbol{\xi})=
     2(\mathbf{l}_j^{-1})^{\!\top}(- \BLog \mathbf{z}\boldsymbol{(\xi}) + i \boldsymbol{\phi})
     = 2(\mathbf{l}_j^{-1})^{\!\top}(- \BLog \mathbf{r})
    = 2(\mathbf{l}_j^{-1})^{\!\top}(-\BLog \mathbf{\hat{z}}(-\boldsymbol{\overline{\xi}}) -i\boldsymbol{\phi})
    = \hat{x}_j(-\boldsymbol{\overline{\xi}}).
    $$
    This completes the proof.
\end{proof}

For $j=1,\dots,k$, on the $j$-th toroidal boundary component, consider the pair of simple closed curves $(l_j,m_j)$. Let $(\omega^{loc}_{m_1},\dots,\omega^{loc}_{m_k})$ and $(\omega^{loc}_{l_1},\dots,\omega^{loc}_{l_k})$ be the logarithmic holonomies of $m_1,\dots,m_k$ and $l_1,\dots,l_k$ as local coordinates of the $\mathrm{PSL}(2;\CC)$-character variety of $M$.
\begin{lemma}\label{biholo}
For all $\boldsymbol{\xi}\in \CC^k$ sufficiently close to $\mathbf{0}$, we have the following results.
\begin{enumerate}
    \item The map
    $$ \Psi  : (w^{loc}_{l_1},\dots,w^{loc}_{l_k}) \mapsto \left(w^{loc}_{m_1}+\sum_{p=1}^{k}n_{p,1} w^{loc}_{l_p}, \dots, w^{loc}_{m_k}+\sum_{p=1}^{k}n_{p,k} w^{loc}_{l_p}\right)  $$
    is a local biholomorphism that sends $\mathbf{0}$ to $\mathbf{0}$.
    \item The map $$\hat{\Psi}(\boldsymbol{\xi}) = \overline{\Psi(-\overline{\boldsymbol{\xi}})}$$ is a local biholomorphism that satisfies $\hat{\Psi}\left(\mathbf{0}\right)=0$ and $\det(D\hat{\Psi}(0)) = (-1)^k\overline{\det(D\Psi(\mathbf{0}))}$.
\end{enumerate}
    
\end{lemma}
\begin{proof} For (1), note that the Jacobian of $\Psi$ is given by
    $$
    \left(\frac{\partial w^{loc}_{m_p}}{\partial w^{loc}_{l_q}}\right)_{p,q}
    + (n_{p,q})_{p,q}.
    $$
    By \cite[Theorem 3]{NZ}, $w^{loc}_{m_p}$ is odd in $w_{l_p}$ and even in $w_{l_q}$ for $q\neq p$. As a result, for $q\neq p$, $\partial w^{loc}_{m_p}/\partial w^{loc}_{l_q}$ is odd in $w_{l_q}$. This implies that $\partial w^{loc}_{m_p}/\partial w^{loc}_{l_q} = 0$ at $\mathbf{0}$ for all $q\neq p$. Besides, for $j=1,\dots,k$, $\partial w^{loc}_{m_j}/\partial w^{loc}_{l_j}$ at $\mathbf{0}$ is equal to moduli of the $j$-th cusp. In particular, it has positive imaginary parts. Altogether, the Jacobian of $\Psi$ at $\mathbf{0}$ is of the form $S_1 + iS_2$, where $S_1,S_2$ are both real and symmetric matrices, and $S_2$ is positive definite. 

    Following the argument in \cite[Lemma 7.2]{BAGPN}, 
    we show that the Jacobian matrix is invertible. Suppose $(S_1+iS_2)\boldsymbol{v} = 0$ for some $\boldsymbol{v}\in \CC^N$. Since $S_1,S_2$ are real symmetric. we have $$\overline{\boldsymbol{v}}^{\!\top}S_1\boldsymbol{v}, \overline{\boldsymbol{v}}^{\!\top}S_2\boldsymbol{v} \in \RR.$$ 
    By considering the imaginary part of the equation
    $$ \overline{\boldsymbol{v}}^{\!\top}(S_1+iS_2)\boldsymbol{v} = 0,$$
    we have $\overline{\boldsymbol{v}}^{\!\top}S_2\boldsymbol{v}=0$. Since $S_2$ is positive definite, we have $\boldsymbol{v}=\mathbf{0}$.

    Altogether, since the Jacobian matrix at $\mathbf{0}$ is invertible, the local biholomorphicity follows from the inverse function theorem. The last statement follows from the fact that at the complete hyperbolic structure, we have $w^{loc}_{l_j}=w^{loc}_{m_j}=0$ for $i,j=1,\dots, N$.

    For (2), we prove the result for the case where $k=1$. The proof of the general case is similar. 
    First, by checking the Cauchy-Riemann equations, one can show that if $\Psi(\xi)$ is holomorphic, then $\hat{\Psi}(\xi)=\overline{\Psi(-\overline{\xi})}$ is also holomorphic with derivative $\hat{\Psi}'(\xi)=-\overline{(\Psi)'(-\overline{\xi})}$. In particular, we have 
    $$\det(\hat{\Psi}'(0))=\det(-\overline{(\Psi)'(0)})= -\overline{\det(\Psi)'(0)}\neq 0.$$
    Moreover, we have $\hat{\Psi}(0) = \overline{\Psi(0)} =0$. The result follows from the inverse function theorem.
\end{proof}

\begin{corollary}\label{biholo2}
    The maps $\boldsymbol{x}(\boldsymbol{\xi})$ and $\boldsymbol{\hat{x}}(\boldsymbol{\xi})$ in Lemma \ref{xitox} are biholomorphisms with
    $$
    \frac{\partial x_j(\boldsymbol{\xi)}}{\partial \xi_j} = 
    -\overline{\frac{\partial \hat{x}_j(-\boldsymbol{\overline{\xi})}}{\partial \xi_j}}.
    $$
\end{corollary}
\begin{proof}
    Note that these two maps are the maps in Lemma \ref{biholo} followed by translations. This proves the first claim. For the second claim, we write $\xi_j = h_j + i d_j$ for $l=1,\dots,k$ and  $$x_j(\boldsymbol{\xi}) = u(h_1,d_1,\dots,h_k,d_k) + i v(h_1,d_1,\dots,h_k,d_k),$$
    where $u$ and $v$ are the imaginary parts of $x_j$ respectively. By using the Cauchy-Riemann equations, we have
    $$
    \frac{\partial x_j(\boldsymbol{\xi)}}{\partial \xi_j}
    = \frac{\partial u}{\partial h_j}(h_1,d_1,\dots,h_k,d_k)
    + i \frac{\partial v}{\partial h_j}(h_1,d_1,\dots,h_k,d_k).
    $$
    Besides, by Lemma \ref{xitox}, we have $\hat{x}_j(\boldsymbol{\xi})=\overline{x_j(\boldsymbol{-\overline{\xi}})}$. 
    Thus,
    $$\hat{x}_j(\boldsymbol{\xi}) = u(-h_1,d_1,\dots,-h_k,d_k) - i v(-h_1,d_1,\dots,-h_k,d_k).$$
    By using the Cauchy-Riemann equations, we have
    $$
    \frac{\partial \hat{x}_j(\boldsymbol{\xi)}}{\partial \xi_j}
    = -\frac{\partial u}{\partial h_j}(-h_1,d_1,\dots,-h_k,d_k)
    + i \frac{\partial v}{\partial h_j}(-h_1,d_1,\dots,-h_k,d_k).
    $$
    and hence
    $$
    \frac{\partial \hat{x}_j(-\boldsymbol{\overline{\xi})}}{\partial \xi_j}
    = -\frac{\partial u}{\partial h_j}(h_1,d_1,\dots,h_k,d_k)
    + i \frac{\partial v}{\partial h_j}(h_1,d_1,\dots,h_k,d_k)
    = \frac{\partial x_j(\boldsymbol{\xi)}}{\partial \xi_j}.
    $$
    This gives the desired result.
\end{proof}

\subsubsection{Properties of $J_1$ and $J_2$}

From Corollary \ref{biholo2}, given $\boldsymbol{x}\in \CC^k$ sufficiently close to $\boldsymbol{x}(\boldsymbol{0})$, we have a corresponding value
$\boldsymbol{\xi}(\boldsymbol{x})$.
Let $(t_{c,1}(\boldsymbol{x}),\dots,t_{c,N}(\boldsymbol{x}))$ be the critical point of $F_1(\mathbf{t};\boldsymbol{\xi}(\boldsymbol{x}))$ with respect to $t_1,\dots,t_N$. 

\begin{proposition}\label{J1andNZ}
    For any fixed $\boldsymbol{x}$ sufficiently close to $\boldsymbol{x}(\mathbf{0})$, the point $t_{c,k+1}(\boldsymbol{x}),\dots, t_{c,N}(\boldsymbol{x})$ is a critical point of $J_1(\boldsymbol{x},t_{k+1},\dots,t_N)$ with respect to $t_{k+1},\dots, t_{N}$. Furthermore, the critical value satisfies
    \begin{align*}
        \frac{\partial}{\partial x_j} J_1(\boldsymbol{x}, t_{c,k+1}(\boldsymbol{x}),\dots,t_{c,N}(\boldsymbol{x}))
        = \frac{i \xi_j(\boldsymbol{x})}{2}.
    \end{align*}

\end{proposition}
\begin{proof}
    Note that, since $F_1^J=F_1\circ L_1$ and $L_1$ is an affine isomorphism, we have
    $$\nabla_{t_1,\dots,t_N}  F_1= 0$$
    if and only if
    $$
    \nabla_{x_1,\dots,x_k,t_{k+1},\dots,t_{N}} F_1^J =
    ((J_1)_{x_1}-i\xi_1/2,\dots, (J_1)_{x_k}-i\xi_k/2,(J_1)_{t_{k+1}},\dots, (J_1)_{t_{N}})=0.
    $$
    This implies that 
    \begin{align}\label{J1critktoN}
    (J_1)_{t_{k+1}}(\boldsymbol{x},t_{c,k+1}(\boldsymbol{x}),\dots, t_{c,N}(\boldsymbol{x})) 
    = \dots 
    = (J_1)_{t_{N}}(\boldsymbol{x},t_{c,k+1}(\boldsymbol{x}),\dots, t_{c,N}(\boldsymbol{x}))
    = 0.
    \end{align}
    Moreover, for $j=1,\dots,k$,
    \begin{align*}
    &\ \frac{\partial}{\partial x_j}
    (J_1(\boldsymbol{x} ,t_{c,k+1}(\boldsymbol{x}),\dots, t_{c,N}(\boldsymbol{x}))) 
    = (J_1)_{x_j} + \sum_{l=k+1}^N (J_1)_{t_l} \frac{\partial t_{c,l}(\boldsymbol{x})}{\partial x_j}
    = (J_1)_{x_j}
    = \frac{i\xi_j(\boldsymbol{x})}{2},
    \end{align*}
    where the second equality follows Equation \eqref{J1critktoN}. This completes the proof.
\end{proof}

\begin{proposition}\label{1loopJ1}
For $\boldsymbol{x}$ sufficiently close to $\boldsymbol{x}(\mathbf{0})$, we have
\begin{align*} &\ \det(\Hess_{t_1,\dots,t_N} F_1) (\mathbf{t_c}(\boldsymbol{x});\boldsymbol{\xi}(\boldsymbol{x})) \\
=&\ (\det(L_1))^2  \left(\frac{i}{2}\right)^{k}\det\left(\frac{\partial \xi_p(\boldsymbol{x})}{\partial x_q}\right)   \det (\Hess_{{t_{k+1},\dots,t_N}
} J_1)(\boldsymbol{x}, t_{c,k+1}(\boldsymbol{x}), \dots, t_{c,N}(\boldsymbol{x})) . 
\end{align*}
In particular, we have
\begin{align*}
   &\ \det (\Hess_{{t_{k+1},\dots,t_N}
} J_1)(\boldsymbol{x}, t_{c,k+1}(\boldsymbol{x}), \dots, t_{c,N}(\boldsymbol{x})) \\
=&\ 
2^{N+k} i^{N-k} \det\mathbf{B}^{-1} \left(\prod_{j=1}^N z_j^{-f_j''} z_j''^{f_j - 1} \right)
\det\left(\frac{\partial x_p(\boldsymbol{\xi})}{\partial \xi_q}\right) \tau(M, l, X, \mathbf{z}).
\end{align*}
\end{proposition}

\begin{proof}
Recall that 
$$ F_1^J(\boldsymbol{x},t_{k+1},\dots, t_N; \boldsymbol{\xi}) =  F_1(L_1(\boldsymbol{x},t_{k+1},\dots, t_N);\boldsymbol{\xi}),$$ where $L_1: \CC^N \to \CC^N$ is the affine isomorphism in Proposition \ref{KLVexpress2}.
In particular, we have
$$ (\Hess_{x_1,\dots,x_k,t_{k+1},\ldots,t_N}  F_1^J) (\boldsymbol{x}, t_{c,k+1}(\boldsymbol{x}), \dots, t_{c,N}(\boldsymbol{x}); \boldsymbol{\xi})  
= (L_1)^{\!\top} (\Hess_{t_1,\dots,t_N}  F_1) (\mathbf{t_c}(\boldsymbol{x});\boldsymbol{\xi}) L_1 $$
and 
\begin{align*}
&\ \det(\Hess_{x_1,\dots,x_k,t_{k+1},\ldots,t_N} F_1^J) (\boldsymbol{x}, t_{c,k+1}(\boldsymbol{x}), \dots, t_{c,N}(\boldsymbol{x}); \boldsymbol{\xi})  \\
=&\ \det(L_1)^{2} \det(\Hess_{t_1,\dots,t_N} F_1) (\mathbf{t_c}(\boldsymbol{x}); \boldsymbol{\xi}) .
\end{align*}
Thus, it suffices to compute 
$$\det(\Hess_{x_1,\dots,x_k,t_{k+1},\ldots,t_N} F_1^J) (\boldsymbol{x}, t_{c,k+1}(\boldsymbol{x}), \dots, t_{c,N}(\boldsymbol{x}); \boldsymbol{\xi}).$$
We claim that 
\begin{enumerate}
\item[(1)] for $p,q \in \{k+1,\dots, N\}$, 
\begin{align*}
\frac{\partial^2 F_1^J}{\partial t_p \partial t_q} (\boldsymbol{x}, t_{c,k+1}(\boldsymbol{x}), \dots, t_{c,N}(\boldsymbol{x});\boldsymbol{\xi})= 
\frac{\partial^2 J_1}{\partial t_p \partial t_q} 
(\boldsymbol{x}, t_{c,k+1}(\boldsymbol{x}), \dots, t_{c,N}(\boldsymbol{x})),
\end{align*}
\item[(2)] for $p \in \{1,\dots,k\}$ and $q \in \{k+1,\dots, N\}$, 
\begin{align*}
&\ \frac{\partial^2 F_1^J}{\partial t_q \partial x_p} (\boldsymbol{x}, t_{c,k+1}(\boldsymbol{x}), \dots, t_{c,N}(\boldsymbol{x});\boldsymbol{\xi}) \\
=&\ - 
\sum_{l=k+1}^N\frac{\partial^2 J_1}{ \partial t_l\partial t_q}(\boldsymbol{x}, t_{c,k+1}(\boldsymbol{x}), \dots, t_{c,N}(\boldsymbol{x}) ) \cdot \Bigg( \frac{\partial t_{c,l}(\boldsymbol{x})}{\partial x_p} \Bigg),
\end{align*}
\item[(3)] for $p,q\in \{1,\dots, k\}$,
\begin{align*}
&\ \frac{\partial^2  F_1^J}{\partial x_p \partial x_q}
(\boldsymbol{x}, t_{c,k+1}(\boldsymbol{x}), \dots, t_{c,N}(\boldsymbol{x});\boldsymbol{\xi})\\
=&\  \frac{i}{2} \frac{\partial \xi_p(\boldsymbol{x})}{\partial x_q} + \sum_{l_1,l_2=k+1}^N \Bigg(\frac{\partial^2 J_1}{\partial t_{l_1} \partial t_{l_2}} (\boldsymbol{x}, t_{c,k+1}(\boldsymbol{x}), \dots, t_{c,N}(\boldsymbol{x})) \cdot \bigg( \frac{\partial t_{c,l_1}(\boldsymbol{x})}{\partial x_p} \bigg)  \bigg( \frac{\partial t_{c,l_2}(\boldsymbol{x})}{\partial x_q} \bigg) \Bigg).
\end{align*}
\end{enumerate}
Assuming these claims, then we have 
$$
\Hess_{x_1,\dots,x_k,t_2,\ldots,t_N}  F_1^J (\boldsymbol{x},  t_{c,k+1}(\boldsymbol{x}), \dots, t_{c,N}(\boldsymbol{x});\boldsymbol{\xi})
= P \cdot D \cdot P^{\!\top},
$$
where 
$$
P=
\begin{pNiceArray}{c | c}
\mathrm{Id}_{k} &  - \left(\frac{\partial t_{c,p}(\boldsymbol{x})}{\partial x_q} \right)  \\ \hline
  0  &  \mathrm{Id}_{N-k}  \\
\end{pNiceArray}
\quad \text{ and } \quad
D=
\begin{pNiceArray}{c | c}
\left(\frac{i}{2}\frac{\partial \xi_p}{\partial x_q}\right)  & 0  \\ \hline
  0  &   \Hess J_1  \\
\end{pNiceArray}.
$$
This implies the desired result. Thus, it suffices to prove Claims (1)-(3). Claim (1) follows from the definition of $F_1^J$. Next, for $q\in \{k+1,\dots, N\}$, since
$$  \frac{\partial F_1^J}{ \partial t_q}(\boldsymbol{x}, t_{c,k+1}(\boldsymbol{x}), \dots, t_{c,N}(\boldsymbol{x}) ;\boldsymbol{\xi} ) = 0,$$
by differentiating both sides with respect to $x_p$, by the chain rule, we have
$$
\frac{\partial^2 F_1^J}{\partial x_p \partial t_q}(\boldsymbol{x}, t_{c,k+1}(\boldsymbol{x}), \dots, t_{c,N}(\boldsymbol{x});\boldsymbol{\xi}) + 
\sum_{l=k+1}^N\frac{\partial^2 F_1^J}{ \partial t_l\partial t_q}(\boldsymbol{x}, t_{c,k+1}(\boldsymbol{x}), \dots, t_{c,N}(\boldsymbol{x});\boldsymbol{\xi})\cdot \Bigg( \frac{\partial t_{c,l}(\boldsymbol{x})}{\partial x_p} \Bigg) = 0
$$
and Claim (2) follows. Finally, for (3), by Proposition \ref{J1andNZ}, since 
$$\frac{\partial}{\partial x_p} J_1(\boldsymbol{x}, t_{c,k+1}(\boldsymbol{x}), \dots, t_{c,N}(\boldsymbol{x})) = \frac{i \xi_p(\boldsymbol{x})}{2},$$ 
by differentiating both sides with respect to $x_q$, we have
\begin{align*}
   &\ \frac{\partial^2 J_1}{\partial x_p\partial x_q}(\boldsymbol{x}, t_{c,k+1}(\boldsymbol{x}), \dots, t_{c,N}(\boldsymbol{x})) 
+ \sum_{l=k+1}^N \frac{\partial^2 J_1}{\partial x_p \partial t_l}(\boldsymbol{x}, t_{c,k+1}(\boldsymbol{x}), \dots, t_{c,N}(\boldsymbol{x}))
\cdot \Bigg( \frac{\partial t_{c,l}(\boldsymbol{x})}{\partial x_q} \Bigg) \\
= &\ \frac{i}{2} \frac{\partial  \xi_p(\boldsymbol{x})}{\partial x_q}.
\end{align*}
Observe that
$$
\frac{\partial^2 F_1^J}{\partial x_p \partial x_q}(\boldsymbol{x}, t_{c,k+1}(\boldsymbol{x}), \dots, t_{c,N}(\boldsymbol{x});\boldsymbol{\xi}) = \frac{\partial^2 J_1}{\partial x_p \partial x_q}(\boldsymbol{x}, t_{c,k+1}(\boldsymbol{x}), \dots, t_{c,N}(\boldsymbol{x}))$$
and
$$
\frac{\partial^2 F_1^J}{\partial x_p \partial t_l}(\boldsymbol{x}, t_{c,k+1}(\boldsymbol{x}), \dots, t_{c,N}(\boldsymbol{x});\boldsymbol{\xi})=\frac{\partial^2 J_1}{\partial x_p \partial t_l}(\boldsymbol{x},t_{c,k+1}(\boldsymbol{x}), \dots, t_{c,N}(\boldsymbol{x})).$$
Thus, we have
\begin{align*}
&\ \frac{\partial^2 F_1^J}{\partial x_p \partial x_q}(\boldsymbol{x}, t_{c,k+1}(\boldsymbol{x}), \dots, t_{c,N}(\boldsymbol{x});\boldsymbol{\xi}) \\
=&\ \frac{i}{2} \frac{\partial \xi_p(\boldsymbol{x})}{\partial x_q} - 
 \sum_{l=k+1}^N\frac{\partial^2 F_1^J}{\partial x_p \partial t_l}(\boldsymbol{x}, t_{c,k+1}(\boldsymbol{x}), \dots, t_{c,N}(\boldsymbol{x});\boldsymbol{\xi} )\cdot \Bigg( \frac{\partial t_{c,l}(\boldsymbol{x})}{\partial x_q} \Bigg).
\end{align*}
Claim (3) then follows from Claim (2). This proves the first statement. The second statement follows from Proposition \ref{detHessF12}.
\end{proof}

Similarly, from Lemma \ref{biholo}, given $\boldsymbol{\hat{x}}$ sufficiently close to $\boldsymbol{\hat{x}}(\mathbf{0})$, we have a corresponding value
$ \boldsymbol{\xi}(\boldsymbol{\hat{x}})$.
Let $(\hat{t}_{c,1}(\boldsymbol{\xi}(\boldsymbol{\hat{x}})),\dots,\hat{t}_{c,N}(\boldsymbol{\xi}(\boldsymbol{\hat{x}})))$ be the critical point of $F_2(\mathbf{\hat{t}};\boldsymbol{\xi}(\boldsymbol{\hat{x}}))$ with respect to $\hat{t}_1,\dots,\hat{t}_N$. 

\begin{proposition}\label{J2andNZ}
    For any fixed $\boldsymbol{\hat{x}}$ sufficiently close to $\boldsymbol{\hat{x}}(\mathbf{0})$, the point $\hat{t}_{c,k}(\boldsymbol{\xi}(\boldsymbol{\hat{x}})),\dots,\hat{t}_{c,N}(\boldsymbol{\xi}(\boldsymbol{\hat{x}}))$ is a critical point of $J_2(\boldsymbol{\hat{x}},\hat{t}_{k+1},\dots,\hat{t}_N)$ with respect to $t_{k+1},\dots, t_{N}$. Furthermore, the critical value satisfies
    \begin{align*}
        \frac{\partial}{\partial \hat{x}_j} J_2(\boldsymbol{x}, \hat{t}_{c,k}(\boldsymbol{\xi}(\boldsymbol{\hat{x}})),\dots,\hat{t}_{c,N}(\boldsymbol{\xi}(\boldsymbol{\hat{x}})))
        = \frac{i \xi_j(\boldsymbol{\hat{x}})}{2}.
    \end{align*}

\end{proposition}
\begin{proof}
    The proof is the same as that of Proposition \ref{J1andNZ}.
\end{proof}

\begin{proposition}\label{1loopJ2}
For $\boldsymbol{\hat{x}}$ sufficiently close to $\boldsymbol{\hat{x}}(\mathbf{0})$, we have
\begin{align*} &\ \det(\Hess_{\hat{t}_1,\dots,\hat{t}_N} F_2) (\mathbf{\hat{t}_c}(\boldsymbol{\hat{x}});\boldsymbol{\xi}(\boldsymbol{\hat{x}})) \\
=&\ (\det(L_1))^2  \left(\frac{i}{2}\right)^{k}\det\left(\frac{\partial \xi_p(\boldsymbol{\hat{x}})}{\partial \hat{x}_q}\right)   \det (\Hess_{{\hat{t}_{k+1},\dots,\hat{t}_N}
} J_2)(\boldsymbol{\hat{x}}, \hat{t}_{c,k+1}(\boldsymbol{\hat{x}}), \dots, \hat{t}_{c,N}(\boldsymbol{\hat{x}})) . 
\end{align*}
In particular, we have
\begin{align*}
   &\ \det (\Hess_{{\hat{t}_{k+1},\dots,\hat{t}_N}
} J_2)(\boldsymbol{\hat{x}}, \hat{t}_{c,k+1}(\boldsymbol{\hat{x}}), \dots, \hat{t}_{c,N}(\boldsymbol{\hat{x}})) \\
=&\ 
2^{N+k} i^{N-k} \det\mathbf{B}^{-1} \left(\prod_{j=1}^N \hat{z}_j^{-f_j''} \hat{z}_j''^{f_j - 1} \right)
\det\left(\frac{\partial \hat{x}_p(\boldsymbol{\xi})}{\partial \xi_q}\right) \tau(M, l, X, \mathbf{\hat{z}}).
\end{align*}
\end{proposition}
\begin{proof}
    The proof is the same as that of Proposition \ref{1loopJ1}.
\end{proof}

\subsection{Joint potential function and its properties}
Let 
\begin{align}\label{defnF}
&\ F(\boldsymbol{x}, t_{k+1},\dots, t_N, \hat{t}_{k+1},\dots, \hat{t}_N; \boldsymbol{\xi})  \notag \\
=&\ F_1^J(\boldsymbol{x}, t_{k+1},\dots, t_N; \boldsymbol{\xi}) -  \boldsymbol{\xi}(\mathbf{l}_{\mathbf{B}}^{-1})^{\!\top}(\boldsymbol{\pi-a})
+ F_2^J(\boldsymbol{x}, \hat{t}_{k+1},\dots, \hat{t}_N; -\boldsymbol{\overline{\xi}})
- \boldsymbol{\overline{\xi}}(\mathbf{l}_{\mathbf{B}}^{-1})^{\!\top}(\boldsymbol{\pi-a})
.  
\end{align} 

\begin{proposition}\label{expogrowthrate} The point
    $$( \boldsymbol{x}(\boldsymbol{\xi}),t_{c,k+1}(\boldsymbol{\xi}),\dots, t_{c,N}(\boldsymbol{\xi}),  \hat{t}_{c,k+1}(-\boldsymbol{\overline{\xi}}),\dots, \hat{t}_{c,N}(-\boldsymbol{\overline{\xi}}))$$ is a critical point of $F( \boldsymbol{x},t_{k+1},\dots, t_N, \hat{t}_{k+1},\dots, \hat{t}_N; \boldsymbol{\xi}) $
    with
    $$\Re F( \boldsymbol{x}(\boldsymbol{\xi}), t_{c,k+1}(\boldsymbol{\xi}),\dots, t_{c,N}(\boldsymbol{\xi}), \hat{t}_{c,k+1}(-\boldsymbol{\overline{\xi}}),\dots, \hat{t}_{c,N}(-\boldsymbol{\overline{\xi}});\boldsymbol{\xi})  = -2\Vol(M; \boldsymbol{l}, \boldsymbol{\xi}),$$
    where $\Vol(M; \boldsymbol{l}, \boldsymbol{\xi})$ is the volume of the representation induced by $\mathbf{z}$.
\end{proposition}
\begin{proof}
    By Proposition \ref{J1andNZ} and \ref{J2andNZ}, we have
     $$\nabla F_1^J(\boldsymbol{x}(\boldsymbol{\xi}),t_{c,k+1}(\boldsymbol{\xi}), \dots, t_{c,N}(\boldsymbol{\xi}); \boldsymbol{\xi})=0=
    \nabla F_2^J(\boldsymbol{\hat{x}}(-\boldsymbol{\overline{\xi}}),\hat{t}_{c,k+1}(-\boldsymbol{\overline{\xi}}), \dots, \hat{t}_{c,N}(-\boldsymbol{\overline{\xi}}); -\boldsymbol{\overline{\xi}}).$$
    By Lemma \ref{xxhat}, $\boldsymbol{x}(\boldsymbol{\xi})=\boldsymbol{\hat{x}}(-\boldsymbol{\overline{\xi}}).$ Altogether, we have
    \begin{align*}
   &\ \nabla F(\boldsymbol{x}(\boldsymbol{\xi}),  t_{c,k+1}(\boldsymbol{\xi}),\dots, t_{c,N}(\boldsymbol{\xi}), \hat{t}_{c,k+1}(-\boldsymbol{\overline{\xi}}),\dots, \hat{t}_{c,N}(-\boldsymbol{\overline{\xi}});\boldsymbol{\xi})\\ 
   =&\
    \nabla F_1^J(\boldsymbol{x}(\boldsymbol{\xi}),t_{c,k+1}(\boldsymbol{\xi}), \dots, t_{c,N}(\boldsymbol{\xi}); \boldsymbol{\xi})
    + \nabla F_2^J(\boldsymbol{\hat{x}}(-\boldsymbol{\overline{\xi}}),\hat{t}_{c,k+1}(-\boldsymbol{\overline{\xi}}), \dots, \hat{t}_{c,N}(-\boldsymbol{\overline{\xi}}); -\boldsymbol{\overline{\xi}})\\
    =&\ 0.
    \end{align*}
    Moreover, by Proposition \ref{realpartpotential}, we have
    \begin{align*}
        &\ \Re F(\boldsymbol{x}(\boldsymbol{\xi}),  t_{c,k+1}(\boldsymbol{\xi}),\dots, t_{c,N}(\boldsymbol{\xi}), \hat{t}_{c,k+1}(-\boldsymbol{\overline{\xi}}),\dots, \hat{t}_{c,N}(-\boldsymbol{\overline{\xi}});\boldsymbol{\xi})\\ 
   =&\
    \Re F_1^J(\boldsymbol{x}(\boldsymbol{\xi}),t_{c,k+1}(\boldsymbol{\xi}), \dots, t_{c,N}(\boldsymbol{\xi}); \boldsymbol{\xi})
    + \Re F_2^J(\boldsymbol{\hat{x}}(-\boldsymbol{\overline{\xi}}),\hat{t}_{c,k+1}(-\boldsymbol{\overline{\xi}}), \dots, \hat{t}_{c,N}(\boldsymbol{\xi}); -\boldsymbol{\overline{\xi}})\\
    =&\ -2\Vol(M; \boldsymbol{l}, \boldsymbol{\xi}).
    \end{align*}
    This completes the proof.
\end{proof}

\begin{proposition}\label{concaveF} 
    For any $\alpha \in \mathscr{A}_{X}$, the real part of $F(\boldsymbol{x}, t_{k+1},\dots, t_N, \hat{t}_{k+1},\dots, \hat{t}_N; \boldsymbol{\xi})$ is strictly concave on the contour on the multi-contour
    $$\RR^k \times \prod_{j=k+1}^N\left(\RR+i(\pi-a_j)\right)\times \prod_{j=k+1}^N\left(\RR+i(\pi-a_j)\right).$$
\end{proposition}
\begin{proof}
From the proof of Proposition \ref{concavity},  we know that $\Hess \Re F_1$ and $\Hess \Re F_2$ are negative definite on the integration multi-contour. Since $\Re F_1^J$ and $\Re F_2^J$ are related to $\Re F_1$ and $\Re F_2$ via composition of affine isomorphisms $L_1$ and $L_2$ respectively, the matricies $\Hess \Re F_1^J$ and $\Hess \Re F_2^J$ are also negative definite. 
Consider the matrix of $\Hess\Re F$. Note that for any non-zero vector $\boldsymbol{v}
= (\boldsymbol{v}^1,\boldsymbol{v}^2,\boldsymbol{v}^3)\in \RR^{N+k}$ with $\boldsymbol{v}^1\in \RR^k$ and $ \boldsymbol{v}^2,\boldsymbol{v}^3\in \RR^{N-k}$, we have
\begin{align*}
     \boldsymbol{v}^{\!\top}\Hess (\Re F) \boldsymbol{v} 
    = (\boldsymbol{v}^1,\boldsymbol{v}^2)^{\!\top}\Hess (\Re F_1^J) (\boldsymbol{v}^1,\boldsymbol{v}^2)
    +(\boldsymbol{v}^1,\boldsymbol{v}^3)^{\!\top}\Hess (\Re F_2^J) (\boldsymbol{v}^1,\boldsymbol{v}^3),
\end{align*}
which is a sum of two non-positive terms. Moreover, since $\boldsymbol{v}$ is non-zero, at least one of $(\boldsymbol{v}^1,\boldsymbol{v}^2)$ and $(\boldsymbol{v}^1,\boldsymbol{v}^3)$ is non-zero. As a result, we have 
$$ \boldsymbol{v}^{\!\top}\Hess(\Re F) \boldsymbol{v} < 0.$$
This shows that the $\Hess\Re F$ is negative definite. In particular, $\Re F$ is strictly concave. 
\end{proof}

\begin{proposition}\label{detHessF}
    We have
    \begin{align*}
        &\ \det \Hess F( \boldsymbol{x}(\boldsymbol{\xi}),t_{c,k+1}(\boldsymbol{\xi}),\dots, t_{c,N}(\boldsymbol{\xi}),  \hat{t}_{c,k+1}(-\boldsymbol{\overline{\xi}}),\dots, \hat{t}_{c,N}(-\boldsymbol{\overline{\xi}}); \boldsymbol{\xi}) \\
=&\ (-1)^N i^k 2^{2N+2k} (\det \mathbf{B}^{-1})^2
\left|\prod_{j=1}^N z_j^{-f_j''} z_j''^{f_j - 1} \right|^2
\det\left( \im \left(\frac{\partial \omega^{loc}_{m_p}}{\partial \omega^{loc}_{l_q}}\right)\right)
\left|
 \mathrm{Tor}(M, \rho,\boldsymbol{l})
\right|^2.
    \end{align*}
\end{proposition}
\begin{proof}
    Similar to the proof of Proposition \ref{1loopJ1}, a direct computation shows that
    \begin{align*}
        \Hess F( \boldsymbol{x}(\boldsymbol{\xi}),t_{c,k+1}(\boldsymbol{\xi}),\dots, t_{c,N}(\boldsymbol{\xi}),  \hat{t}_{c,k+1}(\boldsymbol{\xi}),\dots, \hat{t}_{c,N}(\boldsymbol{\xi}); \boldsymbol{\xi})
        = P \cdot D \cdot P^{\!\top},
    \end{align*}
where
$$
P=
\begin{pNiceArray}{c | c | c}
\mathrm{Id}_{k} &  - \left(\frac{\partial t_{c,p}(\boldsymbol{x})}{\partial x_q} \right) 
& - \left(\frac{\partial \hat{t}_{c,l_p}(\boldsymbol{x})}{\partial x_q} \right) \\ \hline
  0  &  \mathrm{Id}_{N-k} & 0  \\ \hline
    0  & 0 & \mathrm{Id}_{N-k} 
\end{pNiceArray}
$$
and
$$
D=
\begin{pNiceArray}{c | c | c}
\frac{i}{2}\left( 
\left(\frac{\partial x_j(\boldsymbol{\xi})}{\partial \xi_j}\right)^{-1} + \left(\frac{\partial \hat{x}_j (-\boldsymbol{\overline{\xi}})}{\partial \xi_j}\right)^{-1} \right)   & 0  & 0 \\ \hline
  0  &   \Hess J_1 & 0 \\
  \hline 0 & 0 & \Hess J_2
\end{pNiceArray}.
$$
As a result, by Proposition \ref{1loopJ1}, \ref{1loopJ2} and Corollary \ref{biholo2}, we have
\begin{align*}
    & \det\left(\Hess F_1^J (\boldsymbol{x}(\boldsymbol{\xi}),t_{c,k+1}(\boldsymbol{\xi}),\dots, t_{c,N}(\boldsymbol{\xi}),  \hat{t}_{c,k+1}(-\boldsymbol{\overline{\xi}}),\dots, \hat{t}_{c,N}(-\boldsymbol{\overline{\xi}}); \boldsymbol{\xi})\right)\\
=&\ 
\left(\frac{i}{2}\right)^{k}\det\left( \left(\frac{\partial x_p(\boldsymbol{\xi})}{\partial \xi_q}\right)^{-1} + \left(\frac{\partial \hat{x}_p(-\boldsymbol{\overline{\xi}})}{\partial \xi_q}\right)^{-1}\right) \det(\Hess J_1) \det(\Hess J_2)\\
=&\ \left(\frac{i}{2}\right)^{k}\det\left( \left(\frac{\partial x_p(\boldsymbol{\xi})}{\partial \xi_q}\right)^{-1} + \left(\frac{\partial \hat{x}_p(-\boldsymbol{\overline{\xi}})}{\partial \xi_q}\right)^{-1}\right)\\
&\ \times 
\left(2^{N+k} i^{N-k} \det\mathbf{B}^{-1} \left(\prod_{j=1}^N z_j(\boldsymbol{\xi})^{-f_j''} z_j(\boldsymbol{\xi})''^{f_j - 1} \right)
\det\left(\frac{\partial x_p(\boldsymbol{\xi})}{\partial \xi_q}\right) \tau(M, l, X, \mathbf{z}(\boldsymbol{\xi}))\right)\\
&\ \times
\left(2^{N+k} i^{N-k} \det\mathbf{B}^{-1} \left(\prod_{j=1}^N \hat{z}_j(-\boldsymbol{\overline{\xi}})^{-f_j''} \hat{z}_j(-\boldsymbol{\overline{\xi}})''^{f_j - 1} \right)
\det\left(\frac{\partial \hat{x}_p(-\boldsymbol{\overline{\xi}})}{\partial \xi_q}\right) \tau(M, l, X, \mathbf{\hat{z}})\right)\\
=&\ (-1)^N i^k 2^{2N+2k} (\det \mathbf{B}^{-1})^2
\left|\prod_{j=1}^N z_j(\boldsymbol{\xi})^{-f_j''} z_j(\boldsymbol{\xi})''^{f_j - 1} \right|^2
\det\left( \im \left(\frac{\partial x_p(\boldsymbol{\xi})}{\partial \xi_q}\right)\right)
\left|
 \tau(M, l, X, \mathbf{z}(\boldsymbol{\xi})
\right|^2.
\end{align*}
Notice that from Lemma \ref{xitox}, since $n_{p,j}\in \RR$ for all $p,j \in \{1,\dots,k\}$, we have
$$
\im \left(\frac{\partial x_p(\boldsymbol{\xi})}{\partial \xi_q}\right)
= \im \left(\frac{\partial \omega^{loc}_{m_p}}{\partial \omega^{loc}_{l_q}}\right).
$$
Finally, by using Theorem \ref{MWtor}, we have the desired result.
\end{proof}

Consider the following functions $h(\mathbf{t})$ and $h(\mathbf{\hat{t}})$ defined by
\begin{align}
    h_1(\mathbf{t}) =&\ \exp\left(\sum_{j=1}^N \left(\frac{i}{2\pi}t_j\log(1+e^{t_j}) + \frac{i}{\pi}\Li \left(- e^{t_j}\right) \right)\right), \label{defh1}\\
    h_2(\mathbf{\hat{t}}) =&\ \exp\left(-\sum_{j=1}^N \left(\frac{i}{2\pi}\hat{t}_j\log(1+e^{-\hat{t}_j}) + \frac{i}{\pi}\Li \left(- e^{-\hat{t}_j}\right) \right)\right). \label{defh2}
\end{align}
These functions show up when we apply Proposition \ref{prop:quant:dilog:uniform} (4) to approximate the quantum dilogarithm function with the classical dilogarithm function. The following proposition is an analog of \cite[Proposition 3.8]{BAW}.

\begin{proposition}\label{defnR} We have
    $$
     \sqrt{\prod_{j=1}^N z_j(\boldsymbol{\xi})^{f_j''} z_j(\boldsymbol{\xi})''^{1 -f_j } }\sqrt{\left(\prod_{j=1}^N \hat{z}_j(-\boldsymbol{\overline{\xi}})^{f_j''} \hat{z}_j(-\boldsymbol{\overline{\xi}})''^{1-f_j}\right)}
     h_1(\mathbf{t_c}(\boldsymbol{\xi}) )
     h_2(\mathbf{\hat{t}_c}(-\boldsymbol{\overline{\xi}}))
    = \pm
    \exp\left( \frac{i}{\pi}R(\boldsymbol{\xi})\right),
    $$
    where 
    \begin{align*}
    R(\boldsymbol{\xi}) 
    =&\ -\frac{1}{2}(\BLog \mathbf{z}(\boldsymbol{\xi}) - i\pi \mathbf{f})\cdot (\BLog \mathbf{z''}(\boldsymbol{\xi}) + i\pi \mathbf{f''}) + \sum_{k=1}^N \Li \left(e^{-\Log z_k(\boldsymbol{\xi})}\right) \\
    &\ -\frac{1}{2}(\BLog \mathbf{\hat{z}(-\boldsymbol{\overline{\xi}})} - i\pi \mathbf{f})\cdot (\BLog \mathbf{\hat{z}''(-\boldsymbol{\overline{\xi}})} + i\pi \mathbf{f''}) + \sum_{k=1}^N \Li \left(e^{-\Log \hat{z}_k(-\boldsymbol{\overline{\xi}})}\right).
    \end{align*}
\end{proposition}
\begin{proof}
    Note that 
    \begin{align*}
       &\ \sqrt{\left(\prod_{j=1}^N z_j(\boldsymbol{\xi})^{f_j''} z_j(\boldsymbol{\xi})''^{1-f_j}\right)}\\
        =&\ \exp\left( \frac{1}{2} \sum_{j=1}^N \left(f_j'' \Log z_j(\boldsymbol{\xi}) + (1-f_j) \Log z_j(\boldsymbol{\xi})'' \right)\right) \\
        =&\ \exp\left( \frac{1}{2} \sum_{j=1}^N \Log z_j(\boldsymbol{\xi})'' + \frac{i}{\pi} \sum_{j=1}^N \frac{1}{2}\left( -\pi if_j'' \Log z_j(\boldsymbol{\xi}) + \pi i f_j \Log z_j(\boldsymbol{\xi})'' \right)\right) .
    \end{align*}
    By the same argument, we have
    \begin{align*}
       &\ \sqrt{\left(\prod_{j=1}^N \hat{z}_j(-\boldsymbol{\overline{\xi}})^{f_j''} \hat{z}_j(-\boldsymbol{\overline{\xi}})''^{1-f_j}\right)} \\
        =&\ \exp\left( \frac{1}{2} \sum_{j=1}^N \Log \hat{z}_j(-\boldsymbol{\overline{\xi}})'' + \frac{i}{\pi} \sum_{j=1}^N \frac{1}{2}\left( -\pi if_j'' \Log \hat{z}_j(-\boldsymbol{\overline{\xi}}) + \pi i f_j \Log \hat{z}_j(-\boldsymbol{\overline{\xi}})'' \right)\right) .
    \end{align*}
    Recall that $t_j = -\Log z_j + i \pi$ and $\Log(1+e^{t_j}) = \Log(z_j'').$ Thus,
    \begin{align*}
       &\ h_1(\mathbf{t_c}(\boldsymbol{\xi}))\\
        =&\ \exp\left( \sum_{j=1}^N \left( \frac{i}{2\pi} (-\Log z_j(\boldsymbol{\xi}) + \pi i) \Log z_j(\boldsymbol{\xi})'' + \frac{i}{\pi} \Li \left(e^{-\Log z_j(\boldsymbol{\xi})}\right) \right)\right)\\
        =&\ \exp\left( -\sum_{j=1}^N \frac{1}{2} \Log z_j(\boldsymbol{\xi})'' + \frac{i}{\pi}\sum_{j=1}^N \left( -\frac{1}{2} \Log z_j(\boldsymbol{\xi}) \Log z_j(\boldsymbol{\xi})'' + \Li \left(e^{-\Log z_j(\boldsymbol{\xi})}\right) \right)\right).
    \end{align*}
    By the same argument, since $\hat{t}_j = \Log \hat{z}_j + i \pi$ and $\Log(1+e^{-\hat{t}_j}) = \Log(\hat{z}_j'')$, we have
    \begin{align*}
      &\  h_2(\mathbf{\hat{t}_c}(-\boldsymbol{\overline{\xi}}))\\
        =&\ \exp\left( -\sum_{j=1}^N \frac{1}{2} \Log \hat{z}_j(-\boldsymbol{\overline{\xi}})'' + \frac{i}{\pi}\sum_{j=1}^N \left( -\frac{1}{2} \Log \hat{z}_j(-\boldsymbol{\overline{\xi}}) \Log \hat{z}_j(-\boldsymbol{\overline{\xi}})'' + \Li \left(e^{-\Log \hat{z}_j(-\boldsymbol{\overline{\xi}})}\right) \right)\right).
    \end{align*}
    Finally, note that
    $$\exp\left( \frac{i}{\pi} \left((i\pi \mathbf{f}) \cdot (i\pi \mathbf{f''})\right)  \right)  = \pm 1.$$
    By Proposition \ref{critThurscorrespondence}(3), we have the desired result.
\end{proof}

\section{Asymptotics of KLV partition functions}\label{asymana}
\begin{proof}[Proof of Theorem \ref{thm3}]
Suppose Equation \eqref{NZintro} admits a solution $\mathbf{z}$ of shape parameters such that $\im z_j >0$ for $j=1,\dots, k$. Let $\alpha$ be the angle structure induced by $\mathbf{z}$. To simplify the notations, in this proof, we write $\mathbf{t}=(t_{k+1},\dots,t_N)$, $\mathbf{\hat{t}}=(\hat{t}_{k+1},\dots,\hat{t}_N)$ and $\boldsymbol{\pi-a}= (\pi-a_{k+1},\dots,\pi-a_{N})$. By Proposition \ref{KLVexpress2}, we have
\begin{align*}
       &\ W_\B(X,\alpha)
    =  \frac{1}{(2\pi\sqrt{\hbar})^{2N+2k}}\int_{\boldsymbol{x}\in\RR^k} \int_{\mathbf{t}\in \RR+i\boldsymbol{(\pi-a)}} 
    \int_{\mathbf{\hat{t}}\in \RR+i\boldsymbol{(\pi-a)}}
    e^{\frac{1}{2\pi\hbar}F(\boldsymbol{x},\mathbf{t},\mathbf{\hat{t}};\boldsymbol{\xi})+E_{\hbar}(\boldsymbol{x},\mathbf{t},\mathbf{\hat{t}})}
    d\mathbf{t}
   d\mathbf{\hat{t}} 
            d\boldsymbol{x},
    \end{align*}
    where $F$ is defined in Equation \eqref{defnF} and the error term $E_{\hbar}$ is defined by
    \begin{align}
        &\ E_{\hbar}(\boldsymbol{x},\mathbf{t},\mathbf{\hat{t}}) \notag\\
        = &\ e^{\sum_{j=1}^N \left[\Log\left(\Phi_\B\left(\frac{t_j}{2\pi\sqrt{\hbar}}\right)\right)
        +
        \frac{i}{2\pi\sqrt{\hbar}}\Li(-e^{t_j})\right]
        -\sum_{j=1}^N \left[\Log\left(\Phi_\B\left(\frac{-\hat{t}_j}{2\pi\sqrt{\hbar}}\right)\right)
        +
        \frac{i}{2\pi\sqrt{\hbar}}\Li(-e^{\hat{t}_j})\right]}.
    \end{align}

Let $(\boldsymbol{x^c}(\boldsymbol{\xi}),\mathbf{t^c}(\boldsymbol{\xi}),\mathbf{\hat{t}^c}(-\boldsymbol{\overline{\xi}}))$ be the critical point of $F$ described in Proposition \ref{expogrowthrate}. 
Let $r_0>0$ and $\Gamma$ be a real $2N-k$ dimensional ball
inside the integration multi-contour centered at the critical point $(\boldsymbol{x^c}(\boldsymbol{\xi}),\mathbf{t^c}(\boldsymbol{\xi}),\mathbf{\hat{t}^c}(-\boldsymbol{\overline{\xi}}))$ of $F$. We split the integral into the compact part $\Gamma$ and its complement $\Gamma^c$.
Similarly to \cite[Lemma 7.10]{BAGPN}, by using Proposition \ref{prop:quant:dilog:uniform} (4) with $\delta= \pi-\max_{j=1, \ldots,N} \{a_j\}>0$, there exists constants $A,B>0$ such that whenever $\hbar < A$,
\begin{align*}
\left|
\int_{\Gamma^c} 
e^{\frac{1}{2\pi\hbar}F(\boldsymbol{x},\mathbf{t},\mathbf{\hat{t}};\boldsymbol{\xi})+E_{\hbar}(\boldsymbol{x},\mathbf{t},\mathbf{\hat{t}})}d\mathbf{t}
   d\mathbf{\hat{t}} 
            d\boldsymbol{x}\right| 
&\ \leq
e^{2N(C_\delta +(C/\delta+C')\B^2)} \cdot 
\int_{\Gamma^c} 
e^{\frac{1}{2\pi \hbar} \Re F(\boldsymbol{x},\mathbf{t},\mathbf{\hat{t}};\boldsymbol{\xi})} 
d\mathbf{y} \\
&\ \leq
e^{2N(C_\delta +(C/\delta+C')\B^2)} \cdot B \cdot
e^{\frac{1}{2\pi \hbar} M} ,
\end{align*}
where $M$ is the supremum of $F(\boldsymbol{x},\mathbf{t},\mathbf{\hat{t}};\boldsymbol{\xi})$ on the closure of $\Gamma^c$. By Propositions \ref{expogrowthrate} and  \ref{concaveF}, we have 
$$M < \Vol\Big(M;\boldsymbol{l}, \boldsymbol{\xi}\Big).$$

Next, on the compact part $\Gamma$, by Proposition \ref{semihbar}, we have
\begin{align*}
&\ E_{\hbar}(\boldsymbol{x},\mathbf{t},\mathbf{\hat{t}}) 
= h_1(\mathbf{t})h_2(\mathbf{\hat{t}})(1+ \left ( O(\hbar)\right )),
\end{align*}
where 
$h_1$ and $h_2$ are defined in Equations \eqref{defh1} and \eqref{defh2}, 
and
$O(\hbar)$ is a holomorphic function defined in a neighborhood $\Gamma$. By Proposition \ref{saddle}, we have
\begin{align*}
& \int_{\Gamma} 
e^{\frac{1}{2\pi\hbar}F(\boldsymbol{x},\mathbf{t},\mathbf{\hat{t}};\boldsymbol{\xi})+E_{\hbar}(\boldsymbol{x},\mathbf{t},\mathbf{\hat{t}})}d\mathbf{t}
   d\mathbf{\hat{t}} 
            d\boldsymbol{x} \\
=&\  
\left(\frac{1}{2\pi \sqrt{\hbar}}\right)^{2N+k}  \big(2\pi \hbar\big)^{\frac{2N-k}{2}} \frac{h_1(\mathbf{t_c}(\boldsymbol{\xi}))h_2(\mathbf{\hat{t}_c}(\boldsymbol{\xi}))}{\sqrt{(-1)^{N}\frac{\det(\Hess F(\boldsymbol{x}(\boldsymbol{\xi}),\mathbf{t_c}(\boldsymbol{\xi}),\mathbf{\hat{t}_c(\boldsymbol{\xi})};\boldsymbol{\xi}))}{(2\pi)^{N}}}} e^{\frac{1}{2\pi \hbar}F(\boldsymbol{x}(\boldsymbol{\xi}),\mathbf{t_c}(\boldsymbol{\xi}),\mathbf{\hat{t}_c}(\boldsymbol{\xi});\boldsymbol{\xi})} \Big(1 + O(\hbar)\Big).
\end{align*}
The result then follows from Proposition \ref{expogrowthrate} and \ref{detHessF}.
\end{proof}


\begin{thebibliography}{99}
\bibitem{AH} J.E.~Andersen and S.K.~Hansen, \emph{Asymptotics of the quantum invariants for surgeries on the figure 8 knot}, J. Knot Theory Ramifications 15 (2006), no. 4, 479--548.
%
\bibitem{AK} J.E.~Andersen and R.~Kashaev, \emph{A TQFT from Quantum Teichm\"uller theory}, Communications in Mathematical Physics, 330 (3), 2014, p. 887--934.
%
\bibitem{AKnew} J.E.~Andersen and R.~Kashaev, \emph{A new formulation of the Teichm\"uller TQFT}, arXiv preprint 1305.4291.
%
\bibitem{AKicm} J.E.~Andersen and R.~Kashaev, \emph{A TQFT from Quantum Teichm\"uller theory}, Proc. Int. Cong. of Math., 2018, Rio de Janeiro, Vol. 2 (2527--2552).
%
\bibitem{BAGPN} F. Ben Aribi, F. Guéritaud and E. Piguet-Nakazawa, {\em Geometric triangulations and the Teichm\"{u}ller TQFT volume conjecture for twist knots}, Quantum Topol. 14 (2023), 285–406
%
\bibitem{BAGW} F. Ben Aribi, A. Guilloux and  K.H. Wong, \emph{FAMED by computer: proving the Andersen--Kashaev volume conjecture for 42,000 knots}, 
arXiv:2512.17437.

\bibitem{BAW} F. Ben Aribi, K.H. Wong, \emph{The Andersen-Kashaev volume conjecture for FAMED geometric triangulations}, arXiv:2410.10776 

\bibitem{BB}
S. Baseilhac and R. Benedetti, \emph{Quantum hyperbolic invariants of 3-manifolds
with $PSL(2,\CC)$-characters}, Topology 43 (2004) 1373--1423.
%
%
\bibitem{DKY}
R. Detcherry, E. Kalfagianni, T. Yang, {\em Turaev--Viro invariants, colored Jones polynomials, and volume}, Quantum Topol. 9 (2018), no. 4, pp. 775–813

\bibitem{DG}
T. Dimofte and S. Garoufalidis,  \emph{The quantum content of the gluing equations},  Geom. Topol. 17 (2013), no. 3, 1253--1315. 
%
\bibitem{FG} D. Futer and F. Gu\'eritaud, 
\emph{From angled triangulations to hyperbolic structures}, Interactions between hyperbolic geometry, quantum topology and number theory, volume 541 of Contemp. Math., pages 159--182, Amer. Math.Soc., Providence, RI, 2011.
%
\bibitem{G} H. Ge, \emph{Geometric ideal triangulations of hyperbolic 3-manifolds}, arXiv:2609.27635



\bibitem{HKS}
C. Hodgson, A. Kricker, R. Siejakowski, \emph{On the asymptotics of the meromorphic 3D-index}, arXiv:2109.05355

\bibitem{Ka}
R. Kashaev, {\em The hyperbolic volume of knots from the quantum dilogarithm}, Lett. Math. Phys. 39
(1997), no. 3, 269–275.

\bibitem{KLV}
R. Kashaev, F. Luo, G. Vartanov, \emph{A TQFT of Turaev–Viro Type on Shaped Triangulations}, Ann. Henri Poincaré 17, 1109–1143 (2016). https://doi.org/10.1007/s00023-015-0427-8

\bibitem{LMSWY}
T. Liu, S. Ming, X. Sun, B. Wu, T. Yang, {\em $\mathrm{U}_{q\tilde{q}}\mathfrak{sl}(2;\RR)$ Turaev-Viro invariants for cusped 3-manifolds}, arXiv:2608.16560 

\bibitem{M} J. Milnor, {\em Whitehead Torsion}, Bull. Amer. Math. Soc. 72 (1966), 358--426.
\bibitem{MW} S. Ming and B. Wu, {\em
Adjoint Reidemeister torsion from hyperbolic gluing
equations}, arXiv:2609.29299 

\bibitem{MM} H. Murakami and J. Murakami, {\em The colored Jones polynomials and the simplicial volume of a knot},
Acta Math. 186 (2001), no. 1, 85–104.

 \bibitem{N} W. Neumann, {\em Combinatorics of triangulations and the Chern–Simons invariant for hyperbolic 3-manifolds}, from: “Topology ’90”, (B Apanasov, W D Neumann, A W Reid, L Siebenmann, editors), Ohio State Univ. Math. Res. Inst. Publ. 1, de Gruyter, Berlin (1992) 243–271 MR1184415
%
\bibitem{NZ}
W. D. Neumann and D. Zagier, {\em  Volumes of hyperbolic three-manifolds},   Topology 24 (1985), no. 3, 307--332.
%
\bibitem{PW} T. Pandey and K. H. Wong, {\em Geometry of fundamental shadow link complements and applications to the 1-loop conjecture}, arXiv:2308.06643
%
\bibitem{P} J. Porti, {\em Torsion de Reidemeister pour les vari\'et\'es hyperboliques}, Mem. Amer. Math. Soc., 128 (612):x+139, 1997.
%
\bibitem{P2} J. Porti, {\em Reidemeister torsion, hyperbolic three-manifolds, and character varieties}, Handbook of group actions. Vol. IV, 447--507, Adv. Lect. Math. (ALM), 41, Int. Press, Somerville, MA, 2018.
%

\bibitem{WY2} K. H. Wong and T. Yang,  {\em Relative Reshetikhin--Turaev invariants, hyperbolic cone metrics and discrete Fourier transforms I}, Commun. Math. Phys. (2022).
%
\bibitem {WY} K. H. Wong and T. Yang,  {\em Computation of leading coefficients in asymptotics of relative quantum invariants}, Advances in Mathematics, Volume 503, Part B, 2026, 111229
%
\bibitem{Za} D. Zagier, \emph{The dilogarithm function},   Frontiers in number theory, physics, and geometry. II, 3--65, Springer, Berlin, 2007. 
%
\end{thebibliography}
\end{document}